\documentclass[12pt]{amsart}
\usepackage[colorlinks,
            linkcolor=black,
            anchorcolor=blue,
            citecolor=black]{hyperref}
\usepackage{graphicx}
\usepackage{amsmath}
\usepackage{amssymb}
\usepackage{setspace}
\usepackage{amsthm}
\newtheorem{thm}{Theorem}[section]
\newtheorem{cor}[thm]{Corollary}

\newtheorem{lem}[thm]{Lemma}

\theoremstyle{definition}
\newtheorem{defn}[thm]{Definition}
\theoremstyle{remark}
\newtheorem{rem}[thm]{Remark}
\theoremstyle{conclusion}

\theoremstyle{conjecture}
\newtheorem{conj}[thm]{Conjecture}
\numberwithin{equation}{section}

\begin{document}
\title[The method of scaling spheres]{Liouville type theorems for fractional and higher order H\'{e}non-Hardy type equations via the method of scaling spheres}

\author{Wei Dai, Guolin Qin}

\address{School of Mathematical Sciences, Beihang University (BUAA), Beijing 100191, P. R. China, and LAGA, UMR 7539, Institut Galil\'{e}e, Universit\'{e} Sorbonne Paris Nord, 93430 - Villetaneuse, France}
\email{weidai@buaa.edu.cn}

\address{Institute of Applied Mathematics, Chinese Academy of Sciences, Beijing 100190, and University of Chinese Academy of Sciences, Beijing 100049, P. R. China}
\email{qinguolin18@mails.ucas.ac.cn}

\thanks{Wei Dai is supported by the NNSF of China (No. 11971049 and No. 11501021), the Fundamental Research Funds for the Central Universities and the State Scholarship Fund of China (No. 201806025011).}

\begin{abstract}
In this paper, we aim to develop the \emph{(direct) method of scaling spheres}, its \emph{integral forms} and the \emph{method of scaling spheres in local way}. As applications, we investigate Liouville properties of nonnegative solutions to fractional and higher order H\'{e}non-Hardy type equations
\begin{equation*}
  (-\Delta)^{\frac{\alpha}{2}}u(x)=f(x,u(x)) \,\,\,\,\,\,\,\,\,\,\,\, \text{in} \,\,\, \mathbb{R}^{n}, \,\,\, \mathbb{R}^{n}_{+} \,\,\, \text{or} \,\,\, \Omega
\end{equation*}
with $n>\alpha$, $0<\alpha<2$ or $\alpha=2m$ with $1\leq m<\frac{n}{2}$. We first consider the typical case $f(x,u)=|x|^{a}u^{p}$ with $a\in(-\alpha,\infty)$ and $0<p<p_{c}(a):=\frac{n+\alpha+2a}{n-\alpha}$. By using the \emph{method of scaling spheres}, we prove Liouville theorems for the above H\'{e}non-Hardy equations and equivalent integral equations in $\mathbb{R}^{n}$ and $\mathbb{R}^{n}_{+}$. Our results improve the known Liouville theorems for some especially admissible subranges of $a$ and $1<p<\min\left\{\frac{n+\alpha+a}{n-\alpha},p_{c}(a)\right\}$ to the full range $a\in(-\alpha,\infty)$ and $p\in(0,p_{c}(a))$. When $a>0$, we covered the gap $p\in\big[\frac{n+\alpha+a}{n-\alpha},p_{c}(a)\big)$. In particular, when $\alpha=2$, our results give an affirmative answer to the conjecture posed by Phan and Souplet \cite{PS}. As a consequence, we derive a priori estimates of (possibly sign-changing) solutions to Navier problems involving higher order uniformly elliptic operators and existence of positive solutions to Navier problems for higher order Lane-Emden equations for all $1<p<\frac{n+2m}{n-2m}$. Our theorems improve the results in \cite{CFL,DPQ} remarkably to the maximal range of $p$. For bounded domains $\Omega$, we also apply the \emph{method of scaling spheres} to derive Liouville theorems for super-critical problems. Extensions to PDEs and IEs with general nonlinearities $f(x,u)$ are also included (Theorem \ref{general}). In addition to improving most of known Liouville type results to the \emph{sharp exponents} in a \emph{unified} way, we believe the \emph{method of scaling spheres} developed here can be applied conveniently to various fractional or higher order problems \emph{with singularities} or \emph{without translation invariance} or in the cases \emph{the method of moving planes in conjunction with Kelvin transforms do not work}.
\end{abstract}
\maketitle {\small {\bf Keywords:} The method of scaling spheres, H\'{e}non-Hardy type equations, Liouville theorems, nonnegative solutions, a priori estimates, existence of solutions. \\

}

\section{Introduction}
In this paper, our main purpose is to develop the \emph{(direct) method of scaling spheres}, its \emph{integral forms} and the \emph{method of scaling spheres in local way}. As applications, we use the \emph{method of scaling spheres} to establish Liouville properties for several fractional or higher order problems \emph{with singularities} or \emph{without translation invariance} in \emph{multiple kinds of domains}.

\subsection{Liouville theorems in $\mathbb{R}^{n}$}
We first investigate the Liouville property of nonnegative solutions to the following fractional or higher order H\'{e}non-Hardy equations
\begin{equation}\label{PDE}\\\begin{cases}
(-\Delta)^{\frac{\alpha}{2}}u(x)=|x|^{a}u^{p}(x) \,\,\,\,\,\,\,\,\,\, \text{in} \,\,\, \mathbb{R}^{n}, \\
u(x)\geq0, \,\,\,\,\,\,\,\, x\in\mathbb{R}^{n},
\end{cases}\end{equation}
where $n>\alpha$, $0<\alpha<2$ or $\alpha=2m$ with $1\leq m<\frac{n}{2}$, $-\alpha<a<+\infty$ and $0<p<\frac{n+\alpha+2a}{n-\alpha}$. For extensions to fractional or higher order equations with generalized H\'{e}non-Hardy type nonlinearities $f(x,u)$, please see subsection 1.5. When $0<\alpha<2$, the nonlocal fractional Laplacians $(-\Delta)^{\frac{\alpha}{2}}$ is defined by (see e.g. \cite{CI,CLM,DQ1,Si})
\begin{equation}\label{nonlocal defn}
  (-\Delta)^{\frac{\alpha}{2}}u(x)=C_{\alpha,n} \, P.V.\int_{\mathbb{R}^n}\frac{u(x)-u(y)}{|x-y|^{n+\alpha}}dy:=C_{\alpha,n}\lim_{\epsilon\rightarrow0}\int_{|y-x|\geq\epsilon}\frac{u(x)-u(y)}{|x-y|^{n+\alpha}}dy
\end{equation}
for functions $u\in C^{1,1}_{loc}\cap\mathcal{L}_{\alpha}(\mathbb{R}^{n})$, where the constant $C_{\alpha,n}=\left(\int_{\mathbb{R}^{n}}\frac{1-\cos(2\pi\zeta_{1})}{|\zeta|^{n+\alpha}}d\zeta\right)^{-1}$ and the function spaces
\begin{equation}\label{0-1}
  \mathcal{L}_{\alpha}(\mathbb{R}^{n}):=\Big\{u: \mathbb{R}^{n}\rightarrow\mathbb{R}\,\Big|\,\int_{\mathbb{R}^{n}}\frac{|u(x)|}{1+|x|^{n+\alpha}}dx<\infty\Big\}.
\end{equation}
For $0<\alpha<2$, we assume the solution $u\in C_{loc}^{1,1}(\mathbb{R}^{n})\cap\mathcal{L}_{\alpha}(\mathbb{R}^{n})$ if $0\leq a<+\infty$, $u\in C_{loc}^{1,1}(\mathbb{R}^{n}\setminus\{0\})\cap C(\mathbb{R}^{n})\cap\mathcal{L}_{\alpha}(\mathbb{R}^{n})$ if $-\alpha<a<0$. For $\alpha=2m$ with $1\leq m<\frac{n}{2}$, we assume the solution $u\in C^{2m}(\mathbb{R}^{n})$ if $0\leq a<+\infty$, $u\in C^{2m}(\mathbb{R}^{n}\setminus\{0\})\cap C(\mathbb{R}^{n})$ if $-2m<a<0$.

For $0<\alpha\leq n$, PDEs of the form
\begin{equation}\label{GPDE}
  (-\Delta)^{\frac{\alpha}{2}}u(x)=|x|^{a}u^{p}(x)
\end{equation}
are called the fractional order or higher order H\'{e}non, Lane-Emden, Hardy equations for $a>0$, $a=0$, $a<0$, respectively. These equations have numerous important applications in conformal geometry and Sobolev inequalities. In particular, in the case $a=0$, \eqref{GPDE} becomes the well-known Lane-Emden equation, which models many phenomena in mathematical physics and in astrophysics.

We say equations \eqref{GPDE} have critical order if $\alpha=n$ and non-critical order if $0<\alpha<n$. The nonlinear terms in \eqref{GPDE} is called critical if $p=p_{c}(a):=\frac{n+\alpha+2a}{n-\alpha}$ ($:=\infty$ if $n=\alpha$) and subcritical if $0<p<p_{c}(a)$. Liouville type theorems for equations \eqref{GPDE} (i.e., nonexistence of nontrivial nonnegative solutions) in the whole space $\mathbb{R}^n$ and in the half space $\mathbb{R}^n_+$ have been extensively studied (see \cite{BG,CDQ,CFL,CFY,CL,CLiu,CLL,CLZ,Cowan,DPQ,DQ1,DQ,DZ,FW,GS,Lei,Lin,MP,P,PS,RW,SX,WX} and the references therein). For Liouville type theorems and related properties on systems of PDEs of type \eqref{GPDE} with respect to various types of solutions (e.g., stable, radial, nonnegative, sign-changing, $\cdots$), please refer to \cite{BG,DDGW,DQ,FG,LB,M,P,PQS,S,SZ} and the references therein. These Liouville theorems, in conjunction with the blowing up and re-scaling arguments, are crucial in establishing a priori estimates and hence existence of positive solutions to non-variational boundary value problems for a class of elliptic equations on bounded domains or on Riemannian manifolds with boundaries (see \cite{CDQ,CL4,GS1,PQS}).

\smallskip
{\em (i)  The cases $0<\alpha\leq2$.}
\smallskip

We first consider the case $\alpha=2$. For $a=0$, $\alpha=2$ and $1<p<p_{s}:=\frac{n+2}{n-2}$ ($:=\infty$ if $n=2$), Liouville type theorem was established by Gidas and Spruck in their celebrated article \cite{GS}. Later, the proof was remarkably simplified by Chen and Li in \cite{CL} using the Kelvin transform and the method of moving planes. The cases $a\neq0$ have not been fully understood. If $a\leq-2$, \cite{GS} also proved that that \eqref{GPDE} possesses no positive solution in any domain $\Omega$ containing the origin. For $a>-2$, in the case of radial solutions, we have known that if $1<p<p_{c}(a):=\frac{n+2+2a}{n-2}$ ($:=\infty$ if $n=2$), then \eqref{PDE} has no positive radial solution in $\mathbb{R}^{n}$; if $p\geq p_{c}(a)$, then \eqref{PDE} possesses bounded positive radial solution in $\mathbb{R}^{n}$ (see \cite{BG,GS}). It was naturally conjectured that \eqref{PDE} has no positive solution in $\mathbb{R}^{n}$ if $a>-2$ and $1<p<p_{c}(a)$. The condition $p<p_{c}(a)$ is the best possible due to the results on radial solutions in \cite{BG,GS}.
\begin{conj}\label{conjecture}(Phan and Souplet \cite{PS}, 2012)
If $n\geq2$, $\alpha=2$, $a>-2$ and $1<p<p_{c}(a)$, then \eqref{PDE} has no positive solutions in $\mathbb{R}^{n}$.
\end{conj}

For $a>-2$, the nonexistence of positive entire solution was first established by Bidaut-V\'{e}ron and Giacomini \cite{BG} and Bidaut-V\'{e}ron and Pohozaev \cite{BP} for $0<p<\min\big(p_{s},p_{c}(a)\big)$ and Mitidieri and Pohozaev \cite{MP} for $1<p<\frac{n+a}{n-2}$. Subsequently, Phan and Souplet \cite{PS} proved the conjecture for dimension $n=3$ in the class of \emph{bounded} solutions, using the integral estimates and feedback estimates arguments based on Pohozaev identities (\emph{special spatial dimension $n$} and \emph{global boundedness assumptions on $u$} are necessary if $a\neq0$) which were introduced by Serrin and Zou \cite{SZ} and further developed by Souplet \cite{S}. Dai and Qin \cite{DQ} derived Liouville theorem for \eqref{GPDE} for all $a\in\mathbb{R}$ and $0<p\leq1$. Cheng and Liu \cite{CLiu} proved the nonexistence of positive solutions for $a>0$ and $1<p<\frac{n+2+a}{n-2}$ using the method of moving planes in integral forms. One should observe that, if $a>0$ and $n\geq3$, there is still a gap between $\frac{n+2+a}{n-2}$ and the critical exponent $p_{c}(a)$. It is well known that this gap $\big[\frac{n+2+a}{n-2},p_{c}(a)\big)$ can not be covered by the method of moving planes in conjunction with Kelvin transforms, since Kelvin transforms will yield a equation with non-decreasing power $|x|^{\mu}$ ($\mu\geq0$) in the nonlinearity.

As to the cases $0<\alpha<2$, \eqref{PDE} is a fractional equation of nonlocal nature. For $a=0$ and $1<p<\frac{n+\alpha}{n-\alpha}$, Liouville theorem for fractional order Lane-Emden equation \eqref{PDE} was established by Chen, Li and Li \cite{CLL} using a direct method of moving planes (see also \cite{CLZ,DFQ}). There are only a few known results for $a\neq0$. For $a>0$ and $\frac{n+a}{n-\alpha}<p<\frac{n+\alpha+a}{n-\alpha}$, Dou and Zhou \cite{DZ} proved Liouville theorem for fractional order H\'{e}non equation \eqref{PDE} using the method of moving planes in integral forms.

\medskip

By developing and applying the \emph{(direct) method of scaling spheres}, we will first establish Liouville theorem for nonnegative solutions of \eqref{PDE} in cases $0<\alpha\leq2$, $-\alpha<a<+\infty$ and $0<p<p_{c}(a):=\frac{n+\alpha+2a}{n-\alpha}$. One should note that, our results extend the range $p<\frac{n+\alpha+a}{n-\alpha}$ in the literatures mentioned above to the full range $p\in(0,p_{c}(a))$.

\medskip

\emph{The method of scaling spheres} is essentially a frozen variant of the method of moving spheres, that is, we only \emph{dilate or shrink} the spheres with respect to \emph{one fixed} center. The method of moving spheres was initially used by Padilla \cite{Pa} and Chen and Li \cite{CL1}, then developed by Li and Zhu \cite{LZ} (classification results and key calculus lemmas were established therein), which means moving spheres centered at \emph{every points} in $\mathbb{R}^{n}$ or $\partial\mathbb{R}^{n}_{+}$ \emph{simultaneously} (either \emph{all stop at finite radius} or \emph{all stop at $\infty$}) in conjunction with \emph{calculus lemmas} and \emph{ODE analysis}. Later, it was further developed by Li \cite{Li}, Chen and Li \cite{CL0}, Li and Zhang \cite{LZ1} and Jin, Li and Xu \cite{JLX}. Recently, Chen, Li and Zhang developed a direct method of moving spheres on fractional order equations in \cite{CLZ}. One should note that, being different from the method of moving spheres, \emph{the method of scaling spheres} takes full advantage of the \emph{integral representation formulae} of solutions and the \emph{``Bootstrap" technique} to derive \emph{lower bound estimates on the asymptotic behaviour} of positive solutions w.r.t. \emph{singularities} or \emph{$\infty$}, which can finally lead to a \emph{Liouville type result}. Besides improving most of known Liouville type results to the \emph{best exponents} in a \emph{unified and simple} way, we believe the \emph{method of scaling spheres} can be applied conveniently to various problems \emph{with singularities} or \emph{without translation invariance} in general domains.

\medskip

\emph{The method of scaling spheres} developed in this paper consists of \emph{three} main steps. First, the integral representation formulae of solutions imply positive solution $u$ satisfies the same lower bound as fundamental solutions as $|x|\rightarrow\infty$. Second, by combining this lower bound with the integral representation formulae and the scaling spheres procedure, we will get a better lower bound for positive solution than fundamental solutions. Third, based on the second step, a ``Bootstrap" iteration process in conjunction with the integral representation formulae will provide better and better lower bound estimates on the asymptotic behaviour of positive solution $u$ (see Theorem \ref{lower0}, \ref{lower1}, \ref{lower2} and \ref{lower3}), which finally lead to a contradiction with the integrability of solutions indicated by the integral representation formulae unless the solution $u\equiv0$ (see Sections 2-5).

\medskip

Our Liouville type result for \eqref{PDE} with $0<\alpha\leq2$ is the following theorem.
\begin{thm}\label{Thm0}
Assume $n>\alpha$, $0<\alpha\leq2$, $-\alpha<a<+\infty$, $0<p<\frac{n+\alpha+2a}{n-\alpha}$. Suppose $u$ is a nonnegative solution of \eqref{PDE}, then $u\equiv0$ in $\mathbb{R}^{n}$.
\end{thm}

Combining Theorem \ref{Thm0} for $\alpha=2$ and $a>0$ with the results in Bidaut-V\'{e}ron and Giacomini \cite{BG}, Bidaut-V\'{e}ron and Pohozaev \cite{BP}, and Souto \cite{Souto} for $n=2$, we conclude immediately the following corollary.
\begin{cor}\label{corollary}
Conjecture \ref{conjecture} is true for all $0<p<p_{c}(a)$.
\end{cor}

\begin{rem}
Until two months after this work has been completed and submitted, we were aware that Conjecture \ref{conjecture} has been proved before in Guo and Wan \cite{GW} and Reichel and Zou \cite{RZ} for $n\geq3$ and $1<p<p_{c}(a)$. Our method of scaling spheres gives another unified approach to Conjecture \ref{conjecture} for both $0<p\leq1$ and $1<p<p_{c}(a)$, in addition, it can also be applied to general fractional order and higher order problems.
\end{rem}

\smallskip
{\em (ii)  The higher order cases $\alpha=2m$ with $2\leq m<\frac{n}{2}$.}
\smallskip

The case $a\leq0$ has been widely studied by many authors. For $a=0$, $\alpha=4<n$ and $1<p<\frac{n+4}{n-4}$, Lin established the Liouville type theorem for all the nonnegative $C^{4}(\mathbb{R}^{n})$ smooth solutions of \eqref{PDE} in \cite{Lin}. When $a=0$, $\alpha=2m<n$ and $1<p<\frac{n+2m}{n-2m}$, Wei and Xu \cite{WX} proved Liouville type theorem for all the nonnegative $C^{\alpha}(\mathbb{R}^{n})$ smooth solutions of \eqref{PDE}. For $\alpha=2m<n$, $-2m<a\leq0$, $1<p<\frac{n+2m+2a}{n-2m}$ if $-2<a\leq0$, $1<p<+\infty$ if $-2m<a\leq-2$, Dai, Peng and Qin \cite{DPQ} derived Liouville type theorem for all the nonnegative $C^{2m}(\mathbb{R}^{n}\setminus\{0\})\cap C^{2m-2}(\mathbb{R}^{n})$ solutions of \eqref{PDE} under assumptions that either $-2p-2\leq a\leq0$ or $u(x)=o(|x|^{2})$ at $\infty$ holds. Under the same assumptions, they also obtained Liouville theorem for all the nonnegative $C^{2m}(\mathbb{R}^{n}\setminus\{0\})\cap C(\mathbb{R}^{n})$ solutions in the range $1<p<\frac{n+2m+2a}{n-2m}$ and $-2m<a\leq0$. Similar results for critical order H\'{e}non-Hardy equations have been established by Chen, Dai and Qin in \cite{CDQ}.

The cases $a>0$ have not been fully understood. For $\alpha=4$, $a>0$ and $n=5$, Cowan \cite{Cowan} has proved that there are no positive \emph{bounded} classical solutions to \eqref{PDE} provided $1<p<\frac{n+4+2a}{n-4}$ using the integral estimates and feedback estimates arguments based on Pohozaev identities (global boundedness assumptions on $u$ are necessary if $a\neq0$) which were introduced by Serrin and Zou \cite{SZ} and further developed by Souplet \cite{S}. Under some assumptions, Fazly \cite{F} and Phan \cite{P} derived nonexistence of \emph{bounded} positive solutions for general poly-harmonic H\'{e}non equations of type \eqref{PDE} with $a\geq0$ and $p>1$, they also proved Liouville type theorems for systems of poly-harmonic H\'{e}non equations. For $\alpha=2m<n$, Cheng and Liu \cite{CLiu} proved Liouville type theorem for \eqref{PDE} in the cases $a>0$ and $\frac{n}{n-2m}<p<\frac{n+2m+a}{n-2m}$. One should observe that, there is still a gap between $\frac{n+2m+a}{n-2m}$ and the critical exponent $p_{c}(a)$ since $a>0$, which can not be covered by the method of moving planes in conjunction with Kelvin transforms.

By applying the \emph{method of scaling spheres in integral forms}, we will establish Liouville theorem for nonnegative solutions of \eqref{PDE} in cases $\alpha=2m$ with $2\leq m<\frac{n}{2}$, $0\leq a<+\infty$ and $1<p<\frac{n+2m+2a}{n-2m}$. One should note that our results extend the range $\frac{n}{n-2m}<p<\frac{n+2m+a}{n-2m}$ in \cite{CLiu} to the full range $p\in(1,p_{c}(a))$.

First, by Theorem 2 in \cite{CLiu} and Theorem 2.1 in \cite{DPQ}, we have the super poly-harmonic properties of nonnegative solutions $u$, that is, $(-\Delta)^{i}u\geq0$ for every $i=1,2,\cdots,m-1$. Furthermore, from Theorem 1 in \cite{CLiu} and Theorem 2.3 in \cite{DPQ}, we have also known the equivalence between PDEs \eqref{PDE} and the following integral equations
\begin{equation}\label{IE1}
  u(x)=\int_{\mathbb{R}^{n}}\frac{C_{n,m}}{|x-y|^{n-2m}}|y|^{a}u^{p}(y)dy.
\end{equation}
That is, we have the following theorem from \cite{CLiu,DPQ}.
\begin{thm}\label{equivalence1}
Assume $\alpha=2m$ with $1\leq m<\frac{n}{2}$, $0\leq a<+\infty$, $1<p<+\infty$. If u is a nonnegative solution of PDEs \eqref{PDE}, then $u$ is also a nonnegative solution of integral equations \eqref{IE1}, and vice versa.
\end{thm}

Next, we consider the integral equations \eqref{IE1} instead of PDEs \eqref{PDE}. By applying \emph{the method of scaling spheres in integral forms} (see Section 3), we can prove the following Liouville type result for IEs \eqref{IE1}.
\begin{thm}\label{Thm1}
Assume $\alpha=2m$ with $1\leq m<\frac{n}{2}$, $-2m<a<+\infty$, $0<p<\frac{n+2m+2a}{n-2m}$. If $u\in C(\mathbb{R}^{n})$ is a nonnegative solution to IEs \eqref{IE1}, then $u\equiv 0$ in $\mathbb{R}^{n}$.
\end{thm}

\begin{rem}\label{rem3}
One can see clearly from the proof that the assumption $u\in C(\mathbb{R}^{n})$ in Theorem \ref{Thm1} can be weaken into $|x|^{a}u^{p}\in L^{1}_{loc}(\mathbb{R}^{n})$ and $|x|^{a}u^{p-1}\in L^{\frac{n}{2m}}_{loc}(\mathbb{R}^{n})$ if $p>1$.
\end{rem}

As a consequence of Theorem \ref{equivalence1} and \ref{Thm1}, we obtain immediately the following Liouville type theorem on PDEs \eqref{PDE}.
\begin{thm}\label{Thm1'}
Assume $n\geq5$, $\alpha=2m$ with $2\leq m<\frac{n}{2}$, $0\leq a<+\infty$, $1<p<\frac{n+2m+2a}{n-2m}$. Suppose $u$ is a nonnegative solution of \eqref{PDE}, then $u\equiv0$ in $\mathbb{R}^{n}$.
\end{thm}

\begin{rem}\label{remark0}
In Theorem \ref{Thm0} and \ref{Thm1'}, the smoothness assumption on $u$ at $x=0$ is necessary. For $p>1$, equation \eqref{PDE} admits a distributional solution of the form $u(x)=C|x|^{-\sigma}$ with $\sigma=\frac{\alpha+a}{p-1}>0$.
\end{rem}

\begin{rem}\label{rem2}
Once we have established the equivalence between PDEs \eqref{PDE} and IEs \eqref{IE1} for $0<p\leq1$, then Liouville theorem for PDEs \eqref{PDE} follows immediately from Theorem \ref{Thm1}.
\end{rem}

\begin{rem}\label{remark1}
In the cases $a>0$ and $\frac{n+\alpha+a}{n-\alpha}\leq p<\frac{n+\alpha+2a}{n-\alpha}$, since equations \eqref{PDE} are not \emph{translation invariant}, Liouville theorems \emph{without global boundedness assumptions} (i.e., Theorem \ref{Thm0} and \ref{Thm1'}) \emph{can not} be deduced from the method of moving planes combining with Kelvin transforms, the method of moving spheres, and the integral estimates and feedback estimates arguments based on Pohozaev identities. We will use \emph{the method of scaling spheres} to overcome all these difficulties. \emph{The method of scaling spheres} developed here, in conjunction with the integral representation formulae of solutions and a ``Bootstrap" iteration process, will provide useful lower bound estimates for positive solutions as $|x|\rightarrow\infty$ (see Theorem \ref{lower0}, \ref{lower1}, \ref{lower2} and \ref{lower3}), which will lead to a contradiction on integrability of solutions indicated by the integral representation formulae unless the solution $u\equiv0$. As a consequence, Theorem \ref{Thm0} and \ref{Thm1} are proved. In addition to improving \emph{classical known results} to the \emph{critical exponents} in a \emph{unified and simple} way, we believe \emph{the method of scaling spheres} developed here can be applied conveniently to various fractional or higher order problems in $\mathbb{R}^{n}$, $\mathbb{R}^{n}_{+}$ or $\Omega$ \emph{with singularities} or \emph{without translation invariance} or in the cases that \emph{the method of moving planes in conjunction with Kelvin transforms do not work}. Possible applications also include problems on \emph{angular domains}, \emph{fan-shaped domains} and \emph{exterior domains}.
\end{rem}

\subsection{Liouville theorems on a half space $\mathbb{R}^{n}_{+}$}

Applying \emph{the method of scaling spheres}, we can also study the Liouville properties for H\'{e}non-Hardy type equations on a half space $\mathbb{R}^{n}_{+}$.

\smallskip
{\em (i)  The cases $0<\alpha\leq2$.}
\smallskip

We will first prove the Liouville theorem for the following H\'{e}non-Hardy equations with Dirichlet boundary conditions on a half space $\mathbb{R}^{n}_{+}$:
\begin{equation}\label{PDE+}\\\begin{cases}
(-\Delta)^{\frac{\alpha}{2}}u(x)=|x|^{a}u^{p}(x), \,\,\,\,\,\,\,\,\, u(x)\geq0, \,\,\,\,\,\,\,\,\, x\in\mathbb{R}^{n}_{+}, \\
u(x)=0, \,\,\,\,\,\,\,\, x\in\mathbb{R}^{n}\setminus\mathbb{R}^{n}_{+},
\end{cases}\end{equation}
where $n>\alpha$, $0<\alpha\leq2$, $-\alpha<a<+\infty$, $1\leq p<\frac{n+\alpha+2a}{n-\alpha}$ and $\mathbb{R}^{n}_{+}=\{x=(x_{1},\cdots,x_{n})\in\mathbb{R}^{n}\,|\,x_{n}>0\}$ be the upper half Euclidean space. We assume the solution $u\in \mathcal{L}_{\alpha}(\mathbb{R}^{n})\cap C^{1,1}_{loc}(\mathbb{R}^{n}_{+})\cap C(\overline{\mathbb{R}^{n}_{+}})$ if $0<\alpha<2$, $u\in C^{2}(\mathbb{R}^{n}_{+})\cap C(\overline{\mathbb{R}^{n}_{+}})$ if $\alpha=2$.

For $a=0$, there are many works on the Liouville type theorems for Lane-Emden equations on half space $\mathbb{R}^{n}_{+}$, for instance, see \cite{CFL,CFY,CLM,CLZ,CLZC,LZ,RW} and the references therein. In \cite{CFL}, Chen, Fang and Li established Liouville theorem for Navier problem of Lane-Emden equation \eqref{PDE+} on $\mathbb{R}^{n}_{+}$ in the higher order cases $\alpha=2m$ with $1\leq m<\frac{n}{2}$ and $\frac{n}{n-2m}<p\leq\frac{n+2m}{n-2m}$. For the fractional cases $0<\alpha<2$, Chen, Fang and Yang \cite{CFY} established Liouville theorem for Dirichlet problem of Lane-Emden equation \eqref{PDE+} on $\mathbb{R}^{n}_{+}$ in the cases $1<p\leq\frac{n+\alpha}{n-\alpha}$ using the method of moving planes in integral forms. Subsequently, Chen, Li and Zhang \cite{CLZ} reproved the Liouville theorem in \cite{CFY} via a quite different and much simpler approach - a combination of direct methods of moving spheres and moving planes.

In this paper, we will extend these known results for $a=0$ to general cases $-\alpha<a<+\infty$ and $1\leq p<\frac{n+\alpha+2a}{n-\alpha}$ by applying the \emph{(direct) method of scaling spheres} (see Section 4).

Our Liouville type result for \eqref{PDE+} is the following theorem.
\begin{thm}\label{Thm2}
Assume $n>\alpha$, $0<\alpha\leq2$, $-\alpha<a<+\infty$ and $1\leq p<\frac{n+\alpha+2a}{n-\alpha}$. Assume further $-1<a<+\infty$ if $\alpha=2$. Suppose $u$ is a nonnegative classical solution of PDE \eqref{PDE+} on $\mathbb{R}^{n}_{+}$, then $u\equiv0$ in $\overline{\mathbb{R}^{n}_{+}}$.
\end{thm}

\begin{rem}\label{rem1}
When $a\neq0$, it remains an open problem to prove the Liouville theorem for nonnegative classical solutions of \eqref{PDE+} with critical exponent $p=\frac{n+\alpha+2a}{n-\alpha}$.
\end{rem}

\smallskip
{\em (ii)  The higher order cases $\alpha=2m$ with $1\leq m<\frac{n}{2}$.}
\smallskip

As another application of \emph{the method of scaling spheres}, we will also study the Liouville theorem for the following integral equations:
\begin{equation}\label{IE++}
  u(x)=\int_{\mathbb{R}^{n}_{+}}G^{+}_{m}(x,y)(y_{n})^{a}u^{p}(y)dy,
\end{equation}
where $1\leq m<\frac{n}{2}$, $-\frac{2m}{n}<a<+\infty$, $1\leq p<\frac{n+2m+2a}{n-2m}$ and Green's function $G^{+}_{m}(x,y)$ for $(-\Delta)^{m}$ with Navier boundary conditions on the half space $\mathbb{R}^{n}_{+}$ is given by
\begin{equation}\label{GREEN}
  G^{+}_{m}(x,y):=C_{n,m}\left(\frac{1}{|x-y|^{n-2m}}-\frac{1}{|\bar{x}-y|^{n-2m}}\right)
\end{equation}
and $\bar{x}:=(x_{1},x_{2},\cdots,-x_{n})$ is the reflection of $x$ with respect to the boundary $\partial\mathbb{R}^{n}_{+}$. Integral equations \eqref{IE++} are closely related to the following higher order H\'{e}non-Lane-Emden type equations with Navier boundary conditions on a half space $\mathbb{R}^{n}_{+}$:
\begin{equation}\label{PDE++}\\\begin{cases}
(-\Delta)^{m}u(x)=(x_{n})^{a}u^{p}(x), \,\,\,\,\,\,\,\,\, u(x)\geq0, \,\,\,\,\,\,\,\,\, x\in\mathbb{R}^{n}_{+}, \\
u(x)=-\Delta u(x)=\cdots=(-\Delta)^{m-1}u(x)=0, \,\,\,\,\,\,\,\, x\in\partial\mathbb{R}^{n}_{+},
\end{cases}\end{equation}
where $1\leq m<\frac{n}{2}$, $-\frac{2m}{n}<a<+\infty$, $1\leq p<\frac{n+2m+2a}{n-2m}$ and $u\in C^{2m}(\mathbb{R}^{n}_{+})\cap C^{2m-2}(\overline{\mathbb{R}^{n}_{+}})$.

In \cite{CC}, Cao and Chen established Liouville theorem (see also Proposition 2 in \cite{CFL}) for IEs \eqref{IE++} in the cases $a=0$ and $\frac{n}{n-2m}<p\leq\frac{n+2m}{n-2m}$. We will extend their results to general $a>-\frac{2m}{n}$ and $1\leq p<\frac{n+2m+2a}{n-2m}$ by applying \emph{the method of scaling spheres in integral forms} (see Section 5).

Our Liouville type result for IEs \eqref{IE++} is the following theorem.
\begin{thm}\label{Thm3}
Suppose $1\leq m<\frac{n}{2}$, $-\frac{2m}{n}<a<+\infty$, $1\leq p<\frac{n+2m+2a}{n-2m}$. If $u\in C(\overline{\mathbb{R}^{n}_{+}})$ is a nonnegative solution of IEs \eqref{IE++} on $\mathbb{R}^{n}_{+}$, then $u\equiv0$ in $\overline{\mathbb{R}^{n}_{+}}$.
\end{thm}

\begin{rem}\label{rem4}
One can see clearly from the proof that the assumption $u\in C(\overline{\mathbb{R}^{n}_{+}})$ in Theorem \ref{Thm3} can be weaken into $(x_{n})^{a}u^{p-1}\in L^{\frac{n}{2m}}_{loc}(\overline{\mathbb{R}^{n}_{+}})$ if $p>1$.
\end{rem}

\begin{rem}\label{rem1'}
When $a\neq0$, it remains an open problem to prove the Liouville theorem for nonnegative solutions of \eqref{IE++} with critical exponent $p=\frac{n+2m+2a}{n-2m}$.
\end{rem}

In the special case $a=0$, Chen, Fang and Li proved the super poly-harmonic properties of solution $u$ to PDEs \eqref{PDE++} in Theorem 3 in \cite{CFL}, that is, $(-\Delta)^{i}u\geq0$ for every $i=1,2,\cdots,m-1$. Furthermore, in Theorem 1 and Theorem 2 in \cite{CFL}, they also established the equivalence between PDEs \eqref{PDE++} and IEs \eqref{IE++} for $a=0$ and any $1<p<+\infty$. This means, the nonnegative classical solution $u$ to PDEs \eqref{PDE++} also satisfies the equivalent IEs \eqref{IE++}, and vice versa. As a consequence, Chen, Fang and Li \cite{CFL} derived Liouville theorem (Theorem 5 in \cite{CFL}) for PDEs \eqref{PDE++} in the cases $a=0$ and $\frac{n}{n-2m}<p\leq\frac{n+2m}{n-2m}$. Combining our Theorem \ref{Thm3} and Theorem 1 and Theorem 2 in \cite{CFL}, we can obtain Liouville result for PDEs \eqref{PDE++} in the cases $a=0$ and $1<p<\frac{n+2m}{n-2m}$, which extends the results in \cite{CFL}.

Our Liouville type result for PDEs \eqref{PDE++} with $a=0$ is the following corollary.
\begin{cor}\label{Thm3''}
Suppose $a=0$, $1\leq m<\frac{n}{2}$ and $1<p<\frac{n+2m}{n-2m}$. If u is a nonnegative classical solution of Navier problem \eqref{PDE++} on $\mathbb{R}^{n}_{+}$, then $u\equiv0$ in $\overline{\mathbb{R}^{n}_{+}}$.
\end{cor}

\begin{rem}\label{rem5}
Once we have established the equivalence between PDEs \eqref{PDE++} and IEs \eqref{IE++} for general $a\geq0$, then Liouville theorem for PDEs \eqref{PDE++} follows immediately from Theorem \ref{Thm3}.
\end{rem}

Combining Corollary \ref{Thm3''} with Theorem 5 in \cite{CFL}, we get Liouville theorem for PDEs \eqref{PDE++} with $a=0$ in the full range $1<p\leq\frac{n+2m}{n-2m}$.
\begin{thm}\label{Thm3'}
Suppose $a=0$, $1\leq m<\frac{n}{2}$ and $1<p\leq\frac{n+2m}{n-2m}$. If u is a nonnegative classical solution of Navier problem \eqref{PDE++} on $\mathbb{R}^{n}_{+}$, then $u\equiv0$ in $\overline{\mathbb{R}^{n}_{+}}$.
\end{thm}

\begin{rem}\label{remark6}
It follows from Theorem \ref{Thm2} for $\alpha=2$ and $a=0$ that, when $\alpha=2$, Theorem \ref{Thm3'} also holds for $p=1$.
\end{rem}

\begin{rem}\label{remark2}
We would like to mention other applications of \emph{the method of scaling spheres}. For instance, it is an interesting open problem to extend the Liouville theorem for higher order Lane-Emden equations (Theorem \ref{Thm3'}) to general higher order H\'{e}non-Hardy type equations with Navier or Dirichlet boundary conditions on a half space $\mathbb{R}^{n}_{+}$.
\end{rem}

\subsection{A priori estimates and existence of positive solutions in bounded domains}

As an immediate application of the Liouville theorem (Theorem \ref{Thm3'}), we derive a priori estimates and existence of positive solutions to higher order Lane-Emden equations in bounded domains for all $1<p<\frac{n+2m}{n-2m}$.

In general, let the higher order uniformly elliptic operator $L$ be defined by
\begin{eqnarray}\label{0-0}
  L &:=& \left(\sum_{i,j=1}^{n}a_{ij}(x)\frac{\partial^{2}}{\partial x_{i}\partial x_{j}}\right)^{m}+\sum_{|\beta|\leq2m-1}b_{\beta}(x)D^{\beta} \\
 \nonumber &=:& A^{m}+\sum_{|\beta|\leq2m-1}b_{\beta}(x)D^{\beta},
\end{eqnarray}
where the coefficients $b_{\beta}\in L^{\infty}(\Omega)$ and $a_{ij}\in C^{2m-2}(\overline{\Omega})$ such that there exists constant $\tau>0$ with
\begin{equation}\label{0-2}
  \tau|\xi|^{2}\leq\sum_{i,j=1}^{n}a_{ij}(x)\xi_{i}\xi_{j}\leq\tau^{-1}|\xi|^{2}, \,\,\,\,\,\,\,\,\,\, \forall \,\, \xi\in\mathbb{R}^{n}, \,\, x\in\Omega.
\end{equation}
Consider the Navier boundary value problem:
\begin{equation}\label{PDE-N}\\\begin{cases}
Lu(x)=f(x,u), \,\,\,\,\,\,\,\,\,\,\,\,\,\,\,\, x\in\Omega, \\
u(x)=Au(x)=\cdots=A^{m-1}u(x)=0, \,\,\,\,\,\,\,\, x\in\partial\Omega,
\end{cases}\end{equation}
where $n\geq3$, $1\leq m<\frac{n}{2}$, $u\in C^{2m}(\Omega)\cap C^{2m-2}(\overline{\Omega})$ and $\Omega$ is a bounded domain with boundary $\partial\Omega\in C^{2m-2}$.

By virtue of the Liouville theorem in $\mathbb{R}^{n}_{+}$ (Theorem 5 in \cite{CFL}) and Liouville theorem in $\mathbb{R}^{n}$ in \cite{WX}, using the blowing-up and re-scaling methods, Chen, Fang and Li \cite{CFL} obtained a priori estimates for the Navier problem \eqref{PDE-N} in the cases $\frac{n}{n-2m}<p<\frac{n+2m}{n-2m}$ (Theorem 6 in \cite{CFL}). Since Theorem \ref{Thm3'} extends Theorem 5 in \cite{CFL} from $\frac{n}{n-2m}<p\leq\frac{n+2m}{n-2m}$ to the full range $1<p\leq\frac{n+2m}{n-2m}$, through entirely similar blowing-up techniques, we can derive the following a priori estimate for classical solutions (possibly sign-changing solutions) to the Navier problem \eqref{PDE-N} in the full range $1<p<\frac{n+2m}{n-2m}$.
\begin{thm}\label{Thm4}
Assume $1<p<\frac{n+2m}{n-2m}$ and there exist positive, continuous functions $h(x)$ and $k(x)$: $\overline{\Omega}\rightarrow(0,+\infty)$ such that
\begin{equation}\label{0-3}
  \lim_{s\rightarrow+\infty}\frac{f(x,s)}{s^{p}}=h(x), \,\,\,\,\,\,\,\,\,\,\,\, \lim_{s\rightarrow-\infty}\frac{f(x,s)}{|s|^{p}}=k(x)
\end{equation}
uniformly with respect to $x\in\overline{\Omega}$. Then there exists a constant $C>0$ depending only on $\Omega$, $n$, $m$, $p$, $h(x)$, $k(x)$, such that
\begin{equation}\label{0-4}
  \|u\|_{L^{\infty}(\overline{\Omega})}\leq C
\end{equation}
for every classical solution $u$ of problem \eqref{PDE-N}.
\end{thm}

\begin{rem}\label{remark3}
The proof of Theorem \ref{Thm4} is entirely similar to that of Theorem 6 in \cite{CFL}. We only need to replace Theorem 5 in \cite{CFL} by Theorem \ref{Thm3'} in the proof. Thus we omit the details of the proof.
\end{rem}

One can immediately apply Theorem \ref{Thm4} to the following higher order Navier problem
\begin{equation}\label{tNavier}\\\begin{cases}
(-\Delta)^{m}u(x)=u^{p}(x)+t \,\,\,\,\,\,\,\,\,\, \text{in} \,\,\, \Omega, \\
u(x)=-\Delta u(x)=\cdots=(-\Delta)^{m-1}u(x)=0 \,\,\,\,\,\,\,\, \text{on} \,\,\, \partial\Omega,
\end{cases}\end{equation}
where $n\geq3$, $1\leq m<\frac{n}{2}$, $\Omega\subset\mathbb{R}^{n}$ is a bounded domain with $C^{2m-2}$ boundary $\partial\Omega$ and $t$ is an arbitrary nonnegative real number.

We can deduce the following corollary from Theorem \ref{Thm4}.
\begin{cor}\label{cor1}
Assume $1<p<\frac{n+2m}{n-2m}$. Then, for any nonnegative solution $u\in C^{2m}(\Omega)\cap C^{2m-2}(\overline{\Omega})$ to the higher order Navier problem \eqref{tNavier}, we have
\begin{equation}\label{0-5}
  \|u\|_{L^{\infty}(\overline{\Omega})}\leq C(n,m,p,\Omega).
\end{equation}
\end{cor}

\begin{rem}\label{remark5}
In Theorem 1.6 in Dai, Peng and Qin \cite{DPQ}, using the method of moving planes in local way and blow-up arguments, the authors have derived the a priori estimates for the higher order Navier problem \eqref{tNavier} under the assumptions that either $1<p<\frac{n+2m}{n-2m}$ and $\Omega$ is strictly convex, or $p\in\big(1,\frac{n+2}{n-2}\big]\bigcap\left(1,\frac{n+2m}{n-2m}\right)$. One can easily observe that Corollary \ref{cor1} also extends Theorem 1.6 in \cite{DPQ}.
\end{rem}

As a consequence of the a priori estimates (Corollary \ref{cor1}), by applying the Leray-Schauder fixed point theorem (see Theorem 4.1 in \cite{DPQ}), we can derive existence result for positive solution to the following Navier problem for higher order Lane-Emden equations
\begin{equation}\label{Navier}\\\begin{cases}
(-\Delta)^{m}u(x)=u^{p}(x) \,\,\,\,\,\,\,\,\,\, \text{in} \,\,\, \Omega, \\
u(x)=-\Delta u(x)=\cdots=(-\Delta)^{m-1}u(x)=0 \,\,\,\,\,\,\,\, \text{on} \,\,\, \partial\Omega,
\end{cases}\end{equation}
where $n\geq3$, $1\leq m<\frac{n}{2}$ and $\Omega\subset\mathbb{R}^{n}$ is a bounded domain with $C^{2m-2}$ boundary $\partial\Omega$.

By virtue of the a priori estimate (Theorem 6 in \cite{CFL} and Theorem 1.6 in \cite{DPQ}), using the Leray-Schauder fixed point theorem, Dai, Peng and Qin \cite{DPQ} obtained existence of positive solution for the Navier problem \eqref{Navier} in the cases that $p\in(1,\frac{n+2m}{n-2m})$ and $\Omega$ is strictly convex, or $p\in\left(1,\frac{n+2}{n-2}\right]\bigcap\left(1,\frac{n+2m}{n-2m}\right)\bigcup\left(\frac{n}{n-2m},\frac{n+2m}{n-2m}\right)$ (Theorem 1.7 in \cite{DPQ}). For other existence results on second order or critical order H\'{e}non-Hardy equations on bounded domains, please also see \cite{CDQ,CPY,GGN,N} and the references therein. Since Theorem \ref{Thm4} and Corollary \ref{cor1} extend Theorem 6 in \cite{CFL} and Theorem 1.6 in \cite{DPQ} from $p\in\left(1,\frac{n+2}{n-2}\right]\bigcap\left(1,\frac{n+2m}{n-2m}\right)\bigcup\left(\frac{n}{n-2m},\frac{n+2m}{n-2m}\right)$ to the full range $1<p<\frac{n+2m}{n-2m}$, through entirely similar arguments, we can improve Theorem 1.7 in \cite{DPQ} remarkably and derive the following existence result for positive solution to the Navier problem \eqref{Navier} in the full range $1<p<\frac{n+2m}{n-2m}$.
\begin{thm}\label{Thm5}
Assume $1<p<\frac{n+2m}{n-2m}$. Then, the higher order Navier problem \eqref{Navier} possesses at least one positive solution $u\in C^{2m}(\Omega)\cap C^{2m-2}(\overline{\Omega})$. Moreover, the positive solution $u$ satisfies
\begin{equation}\label{lower-bound}
  \|u\|_{L^{\infty}(\overline{\Omega})}\geq\left(\frac{\sqrt{2n}}{diam\,\Omega}\right)^{\frac{2m}{p-1}}.
\end{equation}
\end{thm}

\begin{rem}\label{remark4}
The proof of Theorem \ref{Thm5} is entirely similar to that of Theorem 1.7 in \cite{DPQ}. We only need to replace Theorem 6 in \cite{CFL} and Theorem 1.6 in \cite{DPQ} by Corollary \ref{cor1} in the proof. Thus we omit the details of the proof.
\end{rem}

\begin{rem}\label{remark7}
The lower bounds \eqref{lower-bound} on the $L^{\infty}$ norm of positive solutions $u$ indicate that, if $diam\,\Omega<\sqrt{2n}$, then a uniform priori estimate does not exist and blow-up may occur when $p\rightarrow1+$.
\end{rem}

\begin{rem}\label{rem11}
For literature on Dirichlet or Neumann problems involving higher order elliptic operators, poly-harmonic boundary value problems and higher order equations on compact Riemannian manifolds arising from conformal geometry, please refer to Barton, Hofmann and Mayboroda \cite{BHM1,BHM2}, Felli, Hebey and Robert \cite{FHR}, Hebey, Robert and Wen \cite{HRW}, Mayboroda and Maz$'$ya \cite{MM1,MM2} and the references therein.
\end{rem}

\subsection{Liouville theorems in bounded domains}

By applying the \emph{method of scaling spheres (in local way)}, we also study the following fractional and higher order super-critical problems in balls $B_{R}(0)$ with arbitrary $R>0$, that is, the Dirichlet problems
\begin{equation}\label{Dball}\\\begin{cases}
(-\Delta)^{\frac{\alpha}{2}}u(x)=|x|^{a}u^{p}(x), \,\,\,\,\,\,\, u(x)\geq0 \,\,\,\,\,\,\, \text{in} \,\,\, B_{R}(0), \\
u(x)=0 \,\,\,\,\,\,\,\, \text{in} \,\,\, \mathbb{R}^{n}\setminus B_{R}(0)
\end{cases}\end{equation}
with $0<\alpha<2$, and the Navier problems
\begin{equation}\label{Nball}\\\begin{cases}
(-\Delta)^{\frac{\alpha}{2}}u(x)=|x|^{a}u^{p}(x), \,\,\,\,\,\,\, u(x)\geq0 \,\,\,\,\,\,\, \text{in} \,\,\, B_{R}(0), \\
u(x)=(-\Delta)u(x)=\cdots=(-\Delta)^{\frac{\alpha}{2}-1}u(x)=0 \,\,\,\,\,\,\,\, \text{on} \,\,\, \partial B_{R}(0)
\end{cases}\end{equation}
with $\alpha=2m$ and $1\leq m<\frac{n}{2}$, where $n>\alpha$, $-\alpha<a<+\infty$ and $p_{c}(a):=\frac{n+\alpha+2a}{n-\alpha}<p<+\infty$. For $0<\alpha<2$, we assume $u\in C^{1,1}_{loc}(B_{R})\cap C(\overline{B_{R}})$ if $0\leq a<+\infty$, and $u\in C^{1,1}_{loc}(B_{R}\setminus\{0\})\cap C(\overline{B_{R}})$ if $-\alpha<a<0$. For $\alpha=2m$ with $1\leq m<\frac{n}{2}$, we assume $u\in C^{2m}(B_{R})\cap C^{2m-2}(\overline{B_{R}})$ if $0\leq a<+\infty$, and $u\in C^{2m}(B_{R}\setminus\{0\})\cap C^{2m-2}(\overline{B_{R}}\setminus\{0\})\cap C(\overline{B_{R}})$ if $-\alpha<a<0$.

Our Liouville type result for super-critical problems \eqref{Dball} and \eqref{Nball} is the following theorem.

\begin{thm}\label{ball}
Assume $n>\alpha$, $0<\alpha<2$ or $\alpha=2m$ with $1\leq m<\frac{n}{2}$, $-\alpha<a<+\infty$ and $\frac{n+\alpha+2a}{n-\alpha}<p<+\infty$. Suppose $u$ is a nonnegative solution to super-critical problems \eqref{Dball} or \eqref{Nball}, then $u\equiv0$ in $\overline{B_{R}(0)}$.
\end{thm}

Our theorem seems to be the first result on Liouville properties for fractional and higher order super-critical problems in bounded domains.

\begin{rem}\label{remball-1}
Theorem \ref{Thm5} and Theorem \ref{ball} indicate that the exponent $p_{c}(a):=\frac{n+\alpha+2a}{n-\alpha}$ is optimal for existence of positive solutions to H\'{e}non-Hardy equations in bounded domains.
\end{rem}

\begin{rem}\label{remball-2}
We will prove Theorem \ref{ball} in fractional and higher order cases via a unified approach - \emph{the method of scaling spheres in local way}, that is, shrinking the spheres up to the origin $0\in\mathbb{R}^{n}$. \emph{The method of scaling spheres in local way}, in conjunction with the integral representation formulae of solutions and a ``Bootstrap" iteration process, will provide useful lower bound estimates for the asymptotic behaviour of positive solutions as $|x|\rightarrow0$ (see Theorem \ref{lower4}), which will lead to a contradiction on integrability of solutions indicated by the integral representation formulae near the origin unless the solution $u\equiv0$ (see Section 6).
\end{rem}

\begin{rem}\label{remball-3}
We would also like to mention other applications of \emph{the method of scaling spheres}. For instance, it is an interesting open problem to extend the Liouville theorems, a priori estimates and existence of solutions for non-critical higher order H\'{e}non-Hardy equations (Theorems \ref{Thm3}, \ref{Thm3'}, \ref{Thm4}, \ref{Thm5}, \ref{ball} and Corollary \ref{cor1}) to critical order H\'{e}non-Hardy equations with Navier or Dirichlet boundary conditions on $\mathbb{R}^{n}_{+}$ or bounded domains $\Omega$. \emph{The method of scaling spheres} can also be applied to solve various problems on systems of PDEs or IEs.
\end{rem}

\subsection{Extensions to general nonlinearities}

Consider general fractional or second order elliptic equations of the following forms:
\begin{equation}\label{ePDE}
  (-\Delta)^{\frac{\alpha}{2}}u(x)=f(x,u), \quad\, u(x)\geq0, \quad\, x\in\Omega:=\mathbb{R}^{n}, \, \mathbb{R}^{n}_{+} \,\, \text{or} \,\, B_{R}(0),
\end{equation}
where $n>\alpha$, $0<\alpha\leq2$ or $\alpha=2m$ with $2\leq m<\frac{n}{2}$, and the nonlinear terms $f:\, (\Omega\setminus\{0\})\times\overline{\mathbb{R}_{+}}\rightarrow \overline{\mathbb{R}_{+}}$.

\begin{defn}\label{defn1}
We say that the nonlinear term $f$ has subcritical (supercritical) growth provided that
\begin{equation}\label{e1}
  \mu^{\frac{n+\alpha}{n-\alpha}}f(\mu^{\frac{2}{n-\alpha}}x,\mu^{-1}u)
\end{equation}
is strictly increasing (decreasing) with respect to $\mu\geq1$ or $\mu\leq1$ for all $(x,u)\in(\Omega\setminus\{0\})\times\mathbb{R}_{+}$.
\end{defn}
\begin{defn}\label{defn2}
A function $g(x,u)$ is called locally Lipschitz on $u$ in $\Omega\times\overline{\mathbb{R}_{+}}$ ($\Omega\times\mathbb{R}_{+}$), provided that for any $u_{0}\in\overline{\mathbb{R}_{+}}$ ($u_{0}\in\mathbb{R}_{+}$) and $\omega\subseteq\Omega$ bounded, there exists a (relatively) open neighborhood $U(u_{0})\subset\overline{\mathbb{R}_{+}}$ ($U(u_{0})\subset\mathbb{R}_{+}$) such that $g$ is Lipschitz continuous on $u$ in $\omega\times U(u_{0})$.
\end{defn}

We need the following three assumptions on the nonlinear term $f(x,u)$. \\
$(\mathbf{f_{1}})$ The nonlinear term $f$ is non-decreasing about $u$ in $(\Omega\setminus\{0\})\times\overline{\mathbb{R}_{+}}$, namely,
\begin{equation}\label{e2}
  (x,u), \, (x,v)\in(\Omega\setminus\{0\})\times\overline{\mathbb{R}_{+}} \,\,\, \text{with} \,\,\, u\leq v \,\,\, \text{implies} \,\,\, f(x,u)\leq f(x,v).
\end{equation}
$(\mathbf{f_{2}})$ There exists a $\sigma<\alpha$ such that, $|x|^{\sigma}f(x,u)$ is locally Lipschitz on $u$ in $\Omega\times\mathbb{R}_{+}$. If $\Omega=\mathbb{R}^{n}_{+}$ or $B_{R}(0)$, we require $|x|^{\sigma}f(x,u)$ is locally Lipschitz on $u$ in $\Omega\times\overline{\mathbb{R}_{+}}$. \\
$(\mathbf{f_{3}})$ There exist a cone $\mathcal{C}\subset\Omega$ with vertex at $0$, constants $C>0$, $-\alpha<a<+\infty$, $0<p<p_{c}(a)$ if $f$ is subcritical or $p_{c}(a)<p<+\infty$ if $f$ is supercritical such that, the nonlinear term
\begin{equation}\label{e3}
  f(x,u)\geq C|x|^{a}u^{p}
\end{equation}
in $\mathcal{C}\times\overline{\mathbb{R}_{+}}$. If $\Omega=\mathbb{R}^{n}_{+}$, we require the cone $\mathcal{C}$ contains the positive $x_{n}$-axis, say, $\mathcal{C}=\{x\in\mathbb{R}^{n}_{+}\,|\,x_{n}>\frac{|x|}{\sqrt{n}}\}$. In addition, if $\alpha=2m$ with $2\leq m<\frac{n}{2}$ and $\Omega=\mathbb{R}^{n}$, we assume there exist constants $C>0$, $-2m<a<+\infty$ and $0<p<p_{c}(a)$ such that \eqref{e3} holds on $(\Omega\setminus\{0\})\times\overline{\mathbb{R}_{+}}$.

\smallskip

By applying \emph{the method of scaling spheres} to the generalized equations \eqref{ePDE}, we can derive the following Liouville theorem.
\begin{thm}\label{general}
Assume $\Omega=\mathbb{R}^{n},\mathbb{R}^{n}_{+}$ or $B_{R}(0)$ if $0<\alpha\leq2$ and $\alpha<n$, and $\Omega=\mathbb{R}^{n}$ or $B_{R}(0)$ if $\alpha=2m$ with $2\leq m<\frac{n}{2}$. Assume further $\sigma<1$ if $\Omega=\mathbb{R}^{n}_{+}$ and $\alpha=2$. Suppose $f$ is subcritical if $\Omega=\mathbb{R}^{n}$ or $\mathbb{R}^{n}_{+}$, $f$ is supercritical if $\Omega=B_{R}(0)$, and $f$ satisfies the assumptions $(\mathbf{f_{1}})$, $(\mathbf{f_{2}})$ and $(\mathbf{f_{3}})$, then the Liouville type results in Theorems \ref{Thm0}, \ref{Thm1'}, \ref{Thm2} and \ref{ball} are valid for equations \eqref{ePDE}.
\end{thm}

\begin{rem}\label{rem6}
By using the assumptions $(\mathbf{f_{1}})$, $(\mathbf{f_{2}})$, $(\mathbf{f_{3}})$ and \emph{the method of scaling spheres}, Theorem \ref{general} can be proved through a quite similar way as in the proofs of Theorems \ref{Thm0}, \ref{Thm1'}, \ref{Thm2} and \ref{ball}, so we leave the details to readers. For references, see our later paper (JDE, 269 (2020), 7231-7252). We would like to mention that, if the nonlinear term $f(x,u)$ satisfies subcritical (supercritical) conditions for $\mu\leq1$ (see Definition \ref{defn1}), we only need to carry out calculations and estimates outside the ball $B_{\lambda}(0)$ during the scaling spheres procedure.
\end{rem}

\begin{rem}\label{rem7}
In particular, $f(x,u)=|x|^{a}u^{p}$ with $a>-\alpha$ satisfies all the assumptions in Theorem \ref{general}, thus Theorems \ref{Thm0}, \ref{Thm1'}, \ref{Thm2} and \ref{ball} can also be regarded as corollaries of Theorem \ref{general}. In addition, $f(x,u)=|x|^{a}|x_{n}|^{b}u^{p}$ with $a+b>-\alpha$ and $b\geq0$ also satisfies all the assumptions in Theorem \ref{general} except when $\alpha=2m$ with $2\leq m<\frac{n}{2}$ and $\Omega=\mathbb{R}^{n}$.
\end{rem}

Instead of assumption $(\mathbf{f_{2}})$, we may also make the following assumption on $f(x,u)$.

\smallskip

\noindent$(\mathbf{f_{2}}')$ There exist $\sigma$ and $\nu$ satisfying $\sigma+\nu<\alpha$ and $\nu<\frac{\alpha}{n}$ such that, $|x|^{\sigma}|x_{n}|^{\nu}f(x,u)$ is locally Lipschitz on $u$ in $\Omega\times\mathbb{R}_{+}$. If $\Omega=\mathbb{R}^{n}_{+}$ or $B_{R}(0)$, we require $|x|^{\sigma}|x_{n}|^{\nu}f(x,u)$ is locally Lipschitz on $u$ in $\Omega\times\overline{\mathbb{R}_{+}}$.

By applying \emph{the method of scaling spheres}, we can derive the following Liouville theorem for nonlinear term $f(x,u)$ satisfying $(\mathbf{f_{1}})$, $(\mathbf{f_{2}}')$ and $(\mathbf{f_{3}})$.
\begin{thm}\label{general'}
Assume $\sigma+\nu<1$ and $\nu<\frac{1}{n}$ if $\Omega=\mathbb{R}^{n}_{+}$ and $\alpha=2$. Suppose $f$ satisfies the assumptions $(\mathbf{f_{1}})$, $(\mathbf{f_{2}}')$ and $(\mathbf{f_{3}})$, then the Liouville type results in Theorem \ref{general} are valid for equations \eqref{ePDE} under the same conditions.
\end{thm}

\begin{rem}\label{rem10}
By using the assumptions $(\mathbf{f_{1}})$, $(\mathbf{f_{2}}')$, $(\mathbf{f_{3}})$ and \emph{the method of scaling spheres}, Theorem \ref{general'} can be proved through a quite similar way as in the proofs of Theorems \ref{Thm0}, \ref{Thm1'}, \ref{Thm2} and \ref{ball}, so we leave the details to interested readers. We would like to mention that, if $\Omega=\mathbb{R}^{n}$, $0<\alpha\leq n$ and $\alpha<n$, it is more convenient to prove the Liouville type results in Theorem \ref{general'} via applying the \emph{method of scaling spheres in integral forms} (see proofs of Theorems \ref{Thm1} and \ref{Thm3}) instead of the \emph{direct method of scaling spheres} used in the proof of Theorem \ref{Thm0}.
\end{rem}

\begin{rem}\label{rem9}
When $\alpha=2$ and $\Omega=\mathbb{R}^{n}_{+}$, the additional restrictions $a>-1$ in Theorem \ref{Thm2}, $\sigma<1$ in Theorem \ref{general}, $\sigma+\nu<1$ and $\nu<\frac{1}{n}$ in Theorem \ref{general'} come from the \emph{Narrow region principle} Theorem \ref{NRP1+}. By using the equivalence between PDE \eqref{PDE+} and IE \eqref{IE+} in Theorem \ref{equivalent+} and the \emph{method of scaling spheres in integral forms} in our later paper (Theorem 1.3, Nonlinear Analysis, 183 (2019), 284-302), these additional assumptions can be removed when $\alpha=2$ and $\Omega=\mathbb{R}^{n}_{+}$, and the Liouville type results in Theorems \ref{Thm2}, \ref{general} and \ref{general'} actually hold for any $a>-2$, $\sigma<2$, $\sigma+\nu<2$ and $\nu<\frac{2}{n}$ respectively. This also indicates an evident difference between the \emph{direct method of scaling spheres} and the \emph{method of scaling spheres in integral forms}.
\end{rem}

\begin{rem}\label{rem7'}
In particular, $f(x,u)=|x|^{a}|x_{n}|^{b}u^{p}$ with $a+b>-\alpha$ and $b>-\frac{\alpha}{n}$ satisfies all the assumptions in Theorem \ref{general'} except when $\alpha=2m$ with $2\leq m<\frac{n}{2}$ and $\Omega=\mathbb{R}^{n}$.
\end{rem}

\begin{rem}\label{rem8}
Assume $\alpha=2m$ with $1\leq m<\frac{n}{2}$. Suppose $f(x,u)$ is subcritical and satisfies assumptions $(\mathbf{f_{1}})$, $(\mathbf{f_{2}})$ (or $(\mathbf{f_{2}}')$). If there exists a cone $\mathcal{C}\subset\mathbb{R}^{n}$ with vertex at $0$ such that, $f(x,u)$ satisfies the inequality \eqref{e3} in assumption $(\mathbf{f_{3}})$ in $\mathcal{C}\times\overline{\mathbb{R}_{+}}$, then one can prove in a similar way that the Liouville type results in Theorem \ref{Thm1} are valid for the following generalized integral equations:
\begin{equation}\label{IE1-g}
  u(x)=\int_{\mathbb{R}^{n}}\frac{C_{n,m}}{|x-y|^{n-2m}}f(y,u(y))dy.
\end{equation}
\end{rem}

\begin{rem}\label{rem12}
We would like to mention that the method of scaling spheres developed here can hardly be applied directly to the negative exponent cases $p<0$. In order to derive a complete picture on the classification of nonnegative solutions to H\'{e}non-Hardy type equations for all $p\in\mathbb{R}$, it is equally interesting to investigate the cases $p\leq0$. For example, Duong, Nguyen and Vu (J. Pseudo-Differ. Oper. Appl., 11 (2020), 1719-1730) proved the nonexistence of positive super-solutions for the fractional Lane-Emden system with non-positive exponents $p\leq0$ and $q\leq0$. It seems that similar results for the fractional H\'{e}non-Hardy type equations are still open. It certainly deserves to investigate the fractional or higher order H\'{e}non-Hardy type equations with non-positive exponents $p\leq0$ and could lead to some interesting works.
\end{rem}

The rest of this paper is organized as follows. In section 2, we develop the \emph{(direct) method of scaling spheres} and prove Theorem \ref{Thm0}. In section 3, we prove Theorem \ref{Thm1} via the \emph{method of scaling spheres in integral forms}. Section 4 and 5 are devoted to the proof of Theorem \ref{Thm2} and Theorem \ref{Thm3} respectively. In section 6, we prove Theorem \ref{ball} using \emph{the method of scaling spheres in local way}.

In the following, we will use $C$ to denote a general positive constant that may depend on $n$, $\alpha$, $a$, $p$ and $u$, and whose value may differ from line to line.

\section{Proof of Theorem \ref{Thm0}}
In this section, we will prove Theorem \ref{Thm0} by using contradiction arguments and \emph{the (direct) method of scaling spheres}. Now suppose on the contrary that $u\geq0$ satisfies equation \eqref{PDE} but $u$ is not identically zero, then there exists some $\bar{x}\in\mathbb{R}^{n}$ such that $u(\bar{x})>0$.

\subsection{Equivalence between PDE and IE}
In order to get a contradiction, we first need to show that the solution $u$ to \eqref{PDE} also satisfies the equivalent integral equation
\begin{equation}\label{IE}
  u(x)=C_{n,\alpha}\int_{\mathbb{R}^{n}}\frac{|y|^{a}}{|x-y|^{n-\alpha}}u^{p}(y)dy.
\end{equation}
\begin{thm}\label{equivalent}
Assume $n>\alpha$, $0<\alpha\leq2$, $-\alpha<a<+\infty$ and $0<p<\infty$. Suppose $u$ is nonnegative solution to \eqref{PDE}, then it also solves the integral equation \eqref{IE}, and vice versa.
\end{thm}
\begin{proof}
The following lemma concerning the removable singularity is necessary for our proof.
\begin{lem}\label{lemma1}
Assume $n>\alpha$ and $0<\alpha\leq2$. Suppose $u$ is $\alpha$-harmonic in $B_{R}(0)\setminus\{0\}$ and satisfies
\begin{equation*}
  u(x)=o(|x|^{\alpha-n}), \,\,\,\,\,\,\,\,\,\, \text{as} \,\,\, |x|\rightarrow0.
\end{equation*}
Then $u$ can be defined at $0$ so that it is $\alpha$-harmonic in $B_{R}(0)$.
\end{lem}
Lemma \ref{lemma1} can be proved directly by using the Poisson integral formula and maximum principles (Lemma \ref{max}), so we omit the details.

For arbitrary $R>0$, let
\begin{equation}\label{2-1-29}
v_R(x)=\int_{B_R(0)}G^\alpha_R(x,y)|y|^{a}u^{p}(y)dy,
\end{equation}
where Green's function for $(-\Delta)^{\frac{\alpha}{2}}$ with $0<\alpha\leq2$ on $B_R(0)$ is given by
\begin{equation}\label{2-1-30}
G^\alpha_R(x,y):=\frac{C_{n,\alpha}}{|x-y|^{n-\alpha}}\int_{0}^{\frac{t_{R}}{s_{R}}}\frac{b^{\frac{\alpha}{2}-1}}{(1+b)^{\frac{n}{2}}}db
\,\,\,\,\,\,\,\,\, \text{if} \,\, x,y\in B_{R}(0)
\end{equation}
with $s_{R}=\frac{|x-y|^{2}}{R^{2}}$, $t_{R}=\left(1-\frac{|x|^{2}}{R^{2}}\right)\left(1-\frac{|y|^{2}}{R^{2}}\right)$, and $G^{\alpha}_{R}(x,y)=0$ if $x$ or $y\in\mathbb{R}^{n}\setminus B_{R}(0)$ (see \cite{K}). Then, we can derive from $-\alpha<a<+\infty$ that $v_{R}\in C^{1,1}_{loc}\left(B_{R}(0)\setminus\{0\}\right)\cap C(\mathbb{R}^{n})\cap\mathcal{L}_{\alpha}(\mathbb{R}^{n})$ ($v_{R}\in C^{2}\left(B_{R}(0)\setminus\{0\}\right)\cap C(\mathbb{R}^{n})$ if $\alpha=2$) and satisfies
\begin{equation}\label{2-1-31}\\\begin{cases}
(-\Delta)^{\frac{\alpha}{2}}v_R(x)=|x|^{a}u^{p}(x), \qquad x\in B_R(0)\setminus\{0\},\\
v_R(x)=0,\ \ \ \ \ \ \ x\in \mathbb{R}^{n}\setminus B_R(0).
\end{cases}\end{equation}
Let $w_R(x)=u(x)-v_R(x)\in C^{1,1}_{loc}\left(B_{R}(0)\setminus\{0\}\right)\cap C(\mathbb{R}^{n})\cap\mathcal{L}_{\alpha}(\mathbb{R}^{n})$ ($w_R\in C^{2}\left(B_{R}(0)\setminus\{0\}\right)\cap C(\mathbb{R}^{n})$ if $\alpha=2$). By Lemma \ref{lemma1}, \eqref{PDE} and \eqref{2-1-31}, we have $w_R\in C^{1,1}_{loc}\left(B_{R}(0)\right)\cap C(\mathbb{R}^{n})\cap\mathcal{L}_{\alpha}(\mathbb{R}^{n})$ ($w_R\in C^{2}\left(B_{R}(0)\right)\cap C(\mathbb{R}^{n})$ if $\alpha=2$) and satisfies
\begin{equation}\label{2-1-32}\\\begin{cases}
(-\Delta)^{\frac{\alpha}{2}}w_R(x)=0, \qquad x\in B_R(0),\\
w_R(x)\geq0, \qquad x\in \mathbb{R}^{n}\setminus B_R(0).
\end{cases}\end{equation}

Now we need the following maximum principle for fractional Laplacians.
\begin{lem}\label{max}(Maximum principle, \cite{CLL,Si})
Let $\Omega$ be a bounded domain in $\mathbb{R}^{n}$ and $0<\alpha<2$. Assume that $u\in\mathcal{L}_{\alpha}\cap C^{1,1}_{loc}(\Omega)$ and is l.s.c. on $\overline{\Omega}$. If $(-\Delta)^{\frac{\alpha}{2}}u\geq 0$ in $\Omega$ and $u\geq 0$ in $\mathbb{R}^n\setminus\Omega$, then $u\geq 0$ in $\mathbb{R}^n$. Moreover, if $u=0$ at some point in $\Omega$, then $u=0$ a.e. in $\mathbb{R}^{n}$. These conclusions also hold for unbounded domain $\Omega$ if we assume further that
\[\liminf_{|x|\rightarrow\infty}u(x)\geq0.\]
\end{lem}
By Lemma \ref{max} and maximal principles, we deduce that for any $R>0$,
\begin{equation}\label{2-1-33}
  w_R(x)=u(x)-v_{R}(x)\geq0, \,\,\,\, \forall \,\, x\in\mathbb{R}^{n}.
\end{equation}
Now, for each fixed $x\in\mathbb{R}^{n}$, letting $R\rightarrow\infty$ in \eqref{2-1-33}, we have
\begin{equation}\label{2-1-34}
u(x)\geq C_{n,\alpha}\int_{\mathbb{R}^{n}}\frac{|y|^{a}}{|x-y|^{n-\alpha}}u^{p}(y)dy=:v(x)\geq0.
\end{equation}
Take $x=0$ in \eqref{2-1-34}, we get that $u$ satisfies the following integrability
\begin{equation}\label{2-1-35}
  \int_{\mathbb{R}^{n}}\frac{u^{p}(y)}{|y|^{n-\alpha-a}}dy\leq C_{n,\alpha}^{-1}\,u(0)<\infty.
\end{equation}
One can observe that $v\in C^{1,1}_{loc}(\mathbb{R}^{n}\setminus\{0\})\cap C(\mathbb{R}^{n})\cap\mathcal{L}_{\alpha}(\mathbb{R}^{n})$ ($v\in C^{2}(\mathbb{R}^{n}\setminus\{0\})\cap C(\mathbb{R}^{n})$ if $\alpha=2$) is a solution of
\begin{equation}\label{2-1-36}
(-\Delta)^{\frac{\alpha}{2}}v(x)=|x|^{a}u^{p}(x),  \,\,\,\,\,\forall \, x\in \mathbb{R}^n\setminus\{0\}.
\end{equation}
Define $w(x)=u(x)-v(x)$, then by Lemma \ref{lemma1}, \eqref{PDE} and \eqref{2-1-36}, we have $w\in C^{1,1}_{loc}(\mathbb{R}^{n})\cap\mathcal{L}_{\alpha}(\mathbb{R}^{n})$ ($w\in C^{2}(\mathbb{R}^{n})$ if $\alpha=2$) and satisfies
\begin{equation}\label{2-1-37}\\\begin{cases}
(-\Delta)^{\frac{\alpha}{2}}w(x)=0, \,\,\,\,\,  x\in \mathbb{R}^n,\\
w(x)\geq0 \,\,\,\,\,\,  x\in \mathbb{R}^n.
\end{cases}\end{equation}

Now we need the following Liouville type theorem for $\alpha$-harmonic functions in $\mathbb{R}^{n}$.
\begin{lem}\label{Liouville}(Liouville theorem, \cite{BKN,Fall})
Assume $n\geq1$ and $0<\alpha<2$. Let $u$ be a strong solution of
\begin{equation*}\\\begin{cases}
(-\Delta)^{\frac{\alpha}{2}}u(x)=0, \,\,\,\,\,\,\,\, x\in\mathbb{R}^{n}, \\
u(x)\geq0, \,\,\,\,\,\,\, x\in\mathbb{R}^{n},
\end{cases}\end{equation*}
then $u\equiv C\geq0$.
\end{lem}
For the proof of Lemma \ref{Liouville}, please refer to \cite{BKN,Fall}, see also \cite{CS,ZCCY}.

From Lemma \ref{Liouville} and Liouville theorem for harmonic functions, we get $w(x)=u(x)-v(x)\equiv C\geq0$. Thus, we have proved that
\begin{equation}\label{2-1-38}
  u(x)=C_{n,\alpha}\int_{\mathbb{R}^{n}}\frac{|y|^{a}}{|x-y|^{n-\alpha}}u^{p}(y)dy+C\geq C\geq0.
\end{equation}
Now, by combining \eqref{2-1-35} with \eqref{2-1-38}, we get
\begin{equation}\label{2-1-39}
  C^{p}\int_{\mathbb{R}^{n}}\frac{1}{|y|^{n-\alpha-a}}dy
  \leq\int_{\mathbb{R}^{n}}\frac{u^{p}(y)}{|y|^{n-\alpha-a}}dy<\infty,
\end{equation}
from which we can infer immediately that $C=0$. Therefore, we arrived at
\begin{equation}\label{2-1-40}
  u(x)=C_{n,\alpha}\int_{\mathbb{R}^{n}}\frac{|y|^{a}}{|x-y|^{n-\alpha}}u^{p}(y)dy,
\end{equation}
that is, $u$ satisfies the integral equation \eqref{IE}.

Conversely, assume that $u$ is a nonnegative classical solution of integral equation \eqref{IE}, then
\begin{eqnarray}\label{2e25}
(-\Delta)^{\frac{\alpha}{2}}u(x)
\nonumber &=& \int_{\mathbb{R}^{n}}{\left[(-\Delta)^{\frac{\alpha}{2}}\left(\frac{C}{|x-\xi|^{n-\alpha}}\right)\right]}|\xi|^{a}u^{p}(\xi)d{\xi}
\\
\nonumber &=& \int_{\mathbb{R}^{n}}\delta(x-\xi)|\xi|^{a}u^{p}(\xi)d{\xi}=|x|^{a}u^{p}(x),
\end{eqnarray}
that is, $u$ also solves the PDE \eqref{PDE}. This completes the proof of equivalence between PDE \eqref{PDE} and IE \eqref{IE}.
\end{proof}

\begin{rem}\label{rem0}
Lemma \ref{max} has been established first by Silvestre \cite{Si} without the assumption $u\in C^{1,1}_{loc}(\Omega)$. In \cite{CLL}, Chen, Li and Li provided a much more elementary and simpler proof for Lemma \ref{max} under the assumption $u\in C^{1,1}_{loc}(\Omega)$.
\end{rem}

Since the nonnegative solution $u$ with $u(\bar{x})>0$ also satisfy the integral equation \eqref{IE}, it is actually a positive solution in $\mathbb{R}^{n}$, that is,
\begin{equation}\label{2-0}
  u(x)>0 \quad\quad\quad \text{in} \,\, \mathbb{R}^{n}.
\end{equation}
Moreover, there exist a constant $C>0$, such that the solution $u$ satisfies the following lower bound:
\begin{equation}\label{2-1}
  u(x)\geq\frac{C}{|x|^{n-\alpha}} \,\,\,\,\,\,\,\,\,\,\,\,\, \text{for} \,\,\, |x|\geq1.
\end{equation}
Indeed, since $u>0$ also satisfy the integral equation \eqref{IE}, we can deduce that
\begin{eqnarray}\label{2-2}
  u(x)&\geq& C_{n,\alpha}\int_{|y|\leq\frac{1}{2}}\frac{|y|^{a}}{|x-y|^{n-\alpha}}u^{p}(y)dy\\
 \nonumber &\geq& \frac{C}{|x|^{n-\alpha}}\int_{|y|\leq\frac{1}{2}}|y|^{a}u^{p}(y)dy=:\frac{C}{|x|^{n-\alpha}}
\end{eqnarray}
for all $|x|\geq1$.

\subsection{The method of scaling spheres}

In this subsection, we will apply the \emph{(direct) method of scaling spheres} to show the following lower bound estimates for positive solution $u$, which contradict with the integral equation \eqref{IE} for $0<p<\frac{n+\alpha+2a}{n-\alpha}$.
\begin{thm}\label{lower0}
Assume $n>\alpha$, $0<\alpha\leq2$, $-\alpha<a<+\infty$ and $0<p<\frac{n+\alpha+2a}{n-\alpha}$. Suppose $u$ is a positive solution to \eqref{PDE}, then it satisfies the following lower bound estimates: for $|x|\geq1$,
\begin{equation}\label{lb1}
  u(x)\geq C_{\kappa}|x|^{\kappa} \quad\quad \forall \, \kappa<\frac{\alpha+a}{1-p}, \quad\quad \text{if} \,\,\,\, 0<p<1;
\end{equation}
\begin{equation}\label{lb2}
  u(x)\geq C_{\kappa}|x|^{\kappa} \quad\quad \forall \, \kappa<+\infty, \quad\quad \text{if} \,\,\,\, 1\leq p<\frac{n+\alpha+2a}{n-\alpha}.
\end{equation}
\end{thm}
\begin{proof}
Given any $\lambda>0$, we first define the Kelvin transform of a function $u:\,\mathbb{R}^{n}\rightarrow\mathbb{R}$ centered at $0$ by
\begin{equation}\label{Kelvin}
  u_{\lambda}(x)=\left(\frac{\lambda}{|x|}\right)^{n-\alpha}u\left(\frac{\lambda^{2}x}{|x|^{2}}\right)
\end{equation}
for arbitrary $x\in\mathbb{R}^{n}\setminus\{0\}$. It's obvious that the Kelvin transform $u_{\lambda}$ may have singularity at $0$ and $\lim_{|x|\rightarrow\infty}|x|^{n-\alpha}u_{\lambda}(x)=\lambda^{n-\alpha}u(0)>0$. By \eqref{Kelvin}, one can infer from the regularity assumptions on $u$ that $u_{\lambda}\in\mathcal{L}_{\alpha}(\mathbb{R}^{n})\cap C^{1,1}_{loc}(\mathbb{R}^{n}\setminus\{0\})$ if $0<\alpha<2$ and $u_{\lambda}\in C^{2}(\mathbb{R}^{n}\setminus\{0\})$ if $\alpha=2$. Furthermore, we can deduce from \eqref{PDE} and \eqref{Kelvin} that (for the invariance properties of fractional or higher order Laplacians under the Kelvin type transforms, please refer to \cite{CGS,CL,CLL,CLO,Lin,WX})
\begin{equation}\label{2-3}
   (-\Delta)^{\frac{\alpha}{2}}u_{\lambda}(x)=\left(\frac{\lambda}{|x|}\right)^{n+\alpha}\frac{\lambda^{2a}}{|x|^{a}}u^{p}\left(\frac{\lambda^{2}x}{|x|^{2}}\right)
   =\left(\frac{\lambda}{|x|}\right)^{\tau}|x|^{a}u_{\lambda}^{p}(x), \,\,\,\,\,\,\,\,\, x\in\mathbb{R}^{n}\setminus\{0\},
\end{equation}
where $\tau:=n+\alpha+2a-p(n-\alpha)>0$.

Next, we will carry out the process of scaling spheres with respect to the origin $0\in\mathbb{R}^{n}$. For this purpose, we need some definitions.

Let $\lambda>0$ be an arbitrary positive real number and let the scaling sphere be
\begin{equation}\label{2-4}
  S_{\lambda}:=\partial B_{\lambda}(0)=\{x\in\mathbb{R}^{n}:\, |x|=\lambda\}.
\end{equation}
We define the reflection of $x$ about the sphere $S_{\lambda}$ by $x^{\lambda}:=\frac{\lambda^{2}x}{|x|^{2}}$ and define
\begin{equation}\label{2-5}
  \widetilde{B_{\lambda}}(0):=\{x\in\mathbb{R}^{n}: \, x^{\lambda}\in B_{\lambda}(0)\setminus\{0\}\}.
\end{equation}

Let $\omega^{\lambda}(x):=u_{\lambda}(x)-u(x)$ for any $x\in B_{\lambda}(0)\setminus\{0\}$. By the definition of $u_{\lambda}$ and $\omega^{\lambda}$, we have
\begin{eqnarray}\label{2-6}
  \omega^{\lambda}(x)&=&u_{\lambda}(x)-u(x)=\left(\frac{\lambda}{|x|}\right)^{n-\alpha}u(x^{\lambda})-u(x) \\
 \nonumber &=&\left(\frac{\lambda}{|x|}\right)^{n-\alpha}\left(u(x^{\lambda})-\left(\frac{\lambda}{|x^{\lambda}|}\right)^{n-\alpha}u\big((x^{\lambda})^{\lambda}\big)\right) \\
 \nonumber &=&-\left(\frac{\lambda}{|x|}\right)^{n-\alpha}\omega^{\lambda}(x^{\lambda})=-\big(\omega^{\lambda}\big)_{\lambda}(x)
\end{eqnarray}
for every $x\in B_{\lambda}(0)\setminus\{0\}$.

We will first show that, for $\lambda>0$ sufficiently small,
\begin{equation}\label{2-7}
  \omega^{\lambda}(x)\geq0, \,\,\,\,\,\, \forall \,\, x\in B_{\lambda}(0)\setminus\{0\}.
\end{equation}
Then, we start dilating the sphere $S_{\lambda}$ from near the origin $0$ outward as long as \eqref{2-7} holds, until its limiting position $\lambda=+\infty$ and derive lower bound estimates on $u$. Therefore, the scaling sphere process can be divided into two steps.

\emph{Step 1. Start dilating the sphere from near $\lambda=0$.} Define
\begin{equation}\label{2-8}
  B^{-}_{\lambda}:=\{x\in B_{\lambda}(0)\setminus\{0\} \, | \, \omega^{\lambda}(x)<0\}.
\end{equation}
We will show through contradiction arguments that, for $\lambda>0$ sufficiently small,
\begin{equation}\label{2-9}
  B^{-}_{\lambda}=\emptyset.
\end{equation}

Suppose \eqref{2-9} does not hold, that is, $B^{-}_{\lambda}\neq\emptyset$ and hence $\omega^{\lambda}$ is negative somewhere in $B_{\lambda}(0)\setminus\{0\}$. For arbitrary $x\in B_{\lambda}(0)\setminus\{0\}$, we get from \eqref{2-3} that
\begin{eqnarray}\label{2-10}
  (-\Delta)^{\frac{\alpha}{2}}\omega^{\lambda}(x)&=&\left(\frac{\lambda}{|x|}\right)^{\tau}|x|^{a}u_{\lambda}^{p}(x)-|x|^{a}u^{p}(x) \\
 \nonumber &\geq& |x|^{a}\left(u_{\lambda}^{p}(x)-u^{p}(x)\right)=p|x|^{a}\psi_{\lambda}^{p-1}(x)\omega^{\lambda}(x),
\end{eqnarray}
where $\psi_{\lambda}(x)$ is valued between $u(x)$ and $u_{\lambda}(x)$ by mean value theorem. Therefore, for all $x\in B^{-}_{\lambda}$ with $\theta\lambda<|x|<\lambda$,
\begin{equation}\label{2-11}
  (-\Delta)^{\frac{\alpha}{2}}\omega^{\lambda}(x)\geq p\max\{1,\theta^{a}\}\lambda^{a}\max\left\{u^{p-1}(x),u_{\lambda}^{p-1}(x)\right\}\omega^{\lambda}(x)=:c_{\lambda}(x)\omega^{\lambda}(x),
\end{equation}
where $\theta:=\left[\frac{\min_{|x|\leq1}u}{1+\max_{|x|\leq1}u}\right]^{\frac{1}{n-\alpha}}\in(0,1)$.

\begin{rem}\label{remark8}
For the convenience of readers and in order to avoid misunderstanding, comparing with the original version of our paper on arXiv (arXiv: 1810.02752v1), we simply emphasize here the condition $\theta\lambda<|x|<\lambda$ for \eqref{2-11}. It is quite clear from the \emph{Narrow region principle} Theorem \ref{NRP1} and the proof that we only need to take the case $|x|\sim\lambda$ into account.
\end{rem}

Now we need the following Theorem, which is a variant and generalization of the \emph{Narrow region principle} (Theorem 2.2 in \cite{CLZ}).
\begin{thm}\label{NRP1}(Narrow region principle)
Assume $n>\alpha$, $0<\alpha\leq2$, $-\alpha<a<+\infty$ and $0<p<+\infty$. Let $\lambda>0$ and $A_{\lambda,l}:=\{x\in\mathbb{R}^{n}|\,\lambda-l<|x|<\lambda\}$ be an annulus with small thickness $0<l<\lambda$. Suppose $\omega^{\lambda}\in\mathcal{L}_{\alpha}(\mathbb{R}^{n})\cap C^{1,1}_{loc}(A_{\lambda,l})$ if $0<\alpha<2$ ($\omega^{\lambda}\in C^{2}(A_{\lambda,l})$ if $\alpha=2$) and satisfies
\begin{equation}\label{nrp1}\\\begin{cases}
(-\Delta)^{\frac{\alpha}{2}}\omega^{\lambda}(x)-c_{\lambda}(x)\omega^{\lambda}(x)\geq0 \,\,\,\,\, \text{in} \,\,\, A_{\lambda,l}\cap B^{-}_{\lambda},\\
\text{negative minimum of} \,\, \omega^{\lambda}\,\, \text{is attained in the interior of}\,\, B_{\lambda}(0)\setminus\{0\} \,\,\text{if} \,\,\, B^{-}_{\lambda}\neq\emptyset,\\
\text{negative minimum of} \,\,\, \omega^{\lambda} \,\,\, \text{cannot be attained in} \,\,\, (B_{\lambda}(0)\setminus\{0\})\setminus A_{\lambda,l},
\end{cases}\end{equation}
where $c_{\lambda}(x):=p\max\{1,\theta^{a}\}\lambda^{a}\max\left\{u^{p-1}(x),u_{\lambda}^{p-1}(x)\right\}$. Then, we have \\
(i) there exists a sufficiently small constant $\delta_{0}>0$, such that, for all $0<\lambda\leq\delta_{0}$,
\begin{equation}\label{nrp1-1}
  \omega^{\lambda}(x)\geq0, \,\,\,\,\,\, \forall \,x\in A_{\lambda,l};
\end{equation}
(ii) there exists a sufficiently small $l_{0}>0$ depending on $\lambda$ continuously, such that, for all $0<l\leq l_{0}$,
\begin{equation}\label{nrp1-2}
  \omega^{\lambda}(x)\geq0, \,\,\,\,\,\, \forall \,x\in A_{\lambda,l}.
\end{equation}
\end{thm}
\begin{proof}
Suppose on contrary that \eqref{nrp1-1} and \eqref{nrp1-2} do not hold, we will obtain a contradiction for any $0<\lambda\leq\delta_{0}$ with constant $\delta_{0}$ small enough and any $0<l\leq l_{0}(\lambda)$ with $l_{0}(\lambda)$ sufficiently small respectively. By \eqref{nrp1} and our hypothesis, there exists $\tilde{x}\in A_{\lambda,l}\cap B^{-}_{\lambda}\subset\{x\in\mathbb{R}^{n}|\,\lambda-l<|x|<\lambda\}$ such that
\begin{equation}\label{2-12}
  \omega^{\lambda}(\tilde{x})=\min_{B_{\lambda}(0)\setminus\{0\}}\omega^{\lambda}(x)<0.
\end{equation}

For $0<\alpha<2$, our proof is similar to Theorem 2.2 in \cite{CLZ}. The key ingredient is, the same calculations as in the proof of formulae (2.18) and (2.19) in \cite{CLZ} will yield the following estimate at the negative minimum point $\tilde{x}$:
\begin{equation}\label{2-13}
  (-\Delta)^{\frac{\alpha}{2}}\omega^{\lambda}(\tilde{x})\leq\frac{C}{l^{\alpha}}\omega^{\lambda}(\tilde{x}).
\end{equation}
For $\alpha=2$, we can also obtain the same estimate as \eqref{2-13} at some point $x_{0}\in A_{\lambda,l}\cap B^{-}_{\lambda}$. To this end, we define
\begin{equation}\label{2-14}
  \zeta(x):=\cos\frac{|x|-\lambda+l}{l},
\end{equation}
then it follows that $\zeta(x)\in[\cos1,1]$ for any $x\in\overline{A_{\lambda,l}}=\{x\in\mathbb{R}^{n}\,|\,\lambda-l\leq|x|\leq\lambda\}$ and $-\frac{\Delta\zeta(x)}{\zeta(x)}\geq\frac{1}{l^2}$. Define
\begin{equation}\label{2-15}
  \overline{\omega^{\lambda}}(x):=\frac{\omega^{\lambda}(x)}{\zeta(x)}
\end{equation}
for $x\in\overline{A_{\lambda,l}}$. Then there exists a $x_{0}\in A_{\lambda,l}\cap B^{-}_{\lambda}$ such that
\begin{equation}\label{2-16}
  \overline{\omega^{\lambda}}(x_{0})=\min_{\overline{A_{\lambda,l}}}\overline{\omega^{\lambda}}(x)<0.
\end{equation}
Since
\begin{equation}\label{2-17}
  -\Delta\omega^{\lambda}(x_{0})=-\Delta\overline{\omega^{\lambda}}(x_{0})\zeta(x_{0})-2\nabla\overline{\omega^{\lambda}}(x_{0})\cdot\nabla\zeta(x_{0})
  -\overline{\omega^{\lambda}}(x_{0})\Delta\zeta(x_{0}),
\end{equation}
one immediately has
\begin{equation}\label{2-18}
 -\Delta\omega^{\lambda}(x_{0})\leq\frac{1}{l^2}\omega^{\lambda}(x_{0}).
\end{equation}

In conclusion, we have proved that for both $0<\alpha<2$ and $\alpha=2$, there exists some $\hat{x}\in A_{\lambda,l}\cap B^{-}_{\lambda}$ such that
\begin{equation}\label{2-19}
  (-\Delta)^{\frac{\alpha}{2}}\omega^{\lambda}(\hat{x})\leq\frac{C}{l^{\alpha}}\omega^{\lambda}(\hat{x})<0.
\end{equation}
At the same time, by \eqref{nrp1}, we also have
\begin{equation}\label{2-20}
  (-\Delta)^{\frac{\alpha}{2}}\omega^{\lambda}(\hat{x})\geq c_{\lambda}(\hat{x})\omega^{\lambda}(\hat{x}),
\end{equation}
where $c_{\lambda}(\hat{x}):=p\max\{1,\theta^{a}\}\lambda^{a}\max\left\{u^{p-1}(\hat{x}),u_{\lambda}^{p-1}(\hat{x})\right\}$. It follows immediately that
\begin{equation}\label{2-21}\\\begin{cases}
0\leq c_{\lambda}(\hat{x})\leq p\max\{1,\theta^{a}\}\lambda^{a}u^{p-1}(\hat{x})\leq p\max\{1,\theta^{a}\}\lambda^{a}M_{\lambda}^{p-1} \quad\quad \text{if} \,\, 1\leq p<+\infty, \\
0\leq c_{\lambda}(\hat{x})\leq p\max\{1,\theta^{a}\}\lambda^{a}u_{\lambda}^{p-1}(\hat{x})\leq p\max\{1,\theta^{a}\}\lambda^{a}N_{\lambda}^{1-p} \quad\quad \text{if} \,\, 0<p<1,
\end{cases}\end{equation}
where
\begin{equation}\label{2-22}
  M_{\lambda}:=\sup_{x\in B_{\lambda}(0)}u(x)<+\infty, \quad\quad N_{\lambda}:=\left(\inf_{x\in B_{\lambda}(0)\setminus\{0\}}u_{\lambda}(x)\right)^{-1}<+\infty,
\end{equation}
more precisely, by \eqref{2-1}, one easily verifies $N_{\lambda}\leq C\max\left\{1,\lambda^{n-\alpha}\right\}<+\infty$. Therefore, we can deduce from \eqref{2-19}, \eqref{2-20} and \eqref{2-21} that
\begin{equation}\label{2-23}\\\begin{cases}
\frac{C}{\lambda^{\alpha}}\leq\frac{C}{l^{\alpha}}\leq p\max\{1,\theta^{a}\}\lambda^{a}M_{\lambda}^{p-1} \quad\quad \text{if} \,\, 1\leq p<+\infty, \\
\\
\frac{C}{\lambda^{\alpha}}\leq\frac{C}{l^{\alpha}}\leq p\max\{1,\theta^{a}\}\lambda^{a}N_{\lambda}^{1-p} \quad\quad \text{if} \,\, 0<p<1.
\end{cases}\end{equation}
Note that $a>-\alpha$, we can derive a contradiction from \eqref{2-23} directly if $0<\lambda\leq\delta_{0}$ for some constant $\delta_{0}$ small enough, or if $0<l\leq l_{0}$ for some sufficiently small $l_{0}$ depending on $\lambda$ continuously. Thus \eqref{nrp1-1} and \eqref{nrp1-2} must hold.

Furthermore, by \eqref{nrp1}, we can actually deduce from $\omega^{\lambda}\geq0$ in $A_{\lambda,l}$ that
\begin{equation}\label{nrp1-3}
  \omega^{\lambda}\geq0 \,\,\,\,\,\, \text{in} \,\,\, B_{\lambda}(0)\setminus\{0\}.
\end{equation}
This completes the proof of Theorem \ref{NRP1}.
\end{proof}

In order to apply Theorem \ref{NRP1} to show \eqref{2-9}, we first prove that there exists a sufficiently small $\eta_{0}>0$ such that if $0<\lambda\leq\eta_{0}$, then
\begin{equation}\label{2-24}
  \omega^{\lambda}(x)\geq1, \,\,\,\,\,\,\,\,\, \forall \, x\in \overline{B_{\theta\lambda}(0)}\setminus\{0\},
\end{equation}
where $\theta:=\left[\frac{\min_{|x|\leq1}u}{1+\max_{|x|\leq1}u}\right]^{\frac{1}{n-\alpha}}\in(0,1)$. In fact, \eqref{2-1} implies that
\begin{equation}\label{2-25}
  u_{\lambda}(x)\geq\left(\frac{\lambda}{|x|}\right)^{n-\alpha}\min\Bigg\{\frac{C}{\left|\frac{\lambda^{2}x}{|x|^{2}}\right|^{n-\alpha}},\min_{|x|\leq1}u\Bigg\}
  \geq\min\left\{\frac{C}{\lambda^{n-\alpha}},\left(\frac{\lambda}{|x|}\right)^{n-\alpha}\min_{|x|\leq1}u\right\}
\end{equation}
for all $x\in\overline{B_{\theta\lambda}(0)}\setminus\{0\}$. Therefore, we have if $0<\lambda\leq\eta_{0}$ for some $\eta_{0}>0$ small enough, then
\begin{equation}\label{2-26}
  \omega^{\lambda}(x)=u_{\lambda}(x)-u(x)\geq\min\left\{\frac{C}{\lambda^{n-\alpha}},\left(\frac{\lambda}{|x|}\right)^{n-\alpha}\min_{|x|\leq1}u\right\}
  -\max_{|x|\leq\theta\lambda}u(x)\geq1
\end{equation}
for all $x\in \overline{B_{\theta\lambda}(0)}\setminus\{0\}$, so we arrive at \eqref{2-24}.

Now define $\epsilon_{0}:=\min\{\delta_{0},\eta_{0}\}$. If $0<\lambda\leq\epsilon_{0}$, let $l:=(1-\theta)\lambda\in(0,\lambda)$, then it follows from \eqref{2-11} and \eqref{2-24} that the conditions \eqref{nrp1} in Theorem \ref{NRP1} are satisfied, hence we can deduce from (i) in Theorem \ref{NRP1} that
\begin{equation}\label{2-27}
  \omega^{\lambda}(x)\geq0, \quad\quad \forall \,\, x\in A_{\lambda,l}.
\end{equation}
Therefore, we have proved for all $0<\lambda\leq\epsilon_{0}$, $B^{-}_{\lambda}=\emptyset$, that is,
\begin{equation}\label{2-28}
  \omega^{\lambda}(x)\geq0, \,\,\,\,\,\,\, \forall \, x\in B_{\lambda}(0)\setminus\{0\}.
\end{equation}
This completes Step 1.

\emph{Step 2. Dilate the sphere $S_{\lambda}$ outward until $\lambda=+\infty$ to derive lower bound estimates on $u$.} Step 1 provides us a start point to dilate the sphere $S_{\lambda}$ from near $\lambda=0$. Now we dilate the sphere $S_{\lambda}$ outward as long as \eqref{2-7} holds. Let
\begin{equation}\label{2-29}
  \lambda_{0}:=\sup\{\lambda>0\,|\, \omega^{\mu}\geq0 \,\, in \,\, B_{\mu}(0)\setminus\{0\}, \,\, \forall \, 0<\mu\leq\lambda\}\in(0,+\infty],
\end{equation}
and hence, one has
\begin{equation}\label{2-30}
  \omega^{\lambda_{0}}(x)\geq0, \quad\quad \forall \,\, x\in B_{\lambda_{0}}(0)\setminus\{0\}.
\end{equation}
In what follows, we will prove $\lambda_{0}=+\infty$ by contradiction arguments.

Suppose on contrary that $0<\lambda_{0}<+\infty$. In order to get a contradiction, we will first prove
\begin{equation}\label{2-31}
  \omega^{\lambda_{0}}(x)\equiv0, \,\,\,\,\,\,\forall \, x\in B_{\lambda_{0}}(0)\setminus\{0\}
\end{equation}
by using the \emph{Narrow region principle} (Theorem \ref{NRP1}) and contradiction arguments.

Suppose on contrary that \eqref{2-31} does not hold, that is, $\omega^{\lambda_{0}}\geq0$ but $\omega^{\lambda_{0}}$ is not identically zero in $B_{\lambda_{0}}(0)\setminus\{0\}$, then there exists a $x^{0}\in B_{\lambda_{0}}(0)\setminus\{0\}$ such that $\omega^{\lambda_{0}}(x^{0})>0$. We will obtain a contradiction with \eqref{2-29} via showing that the sphere $S_{\lambda}$ can be dilated outward a little bit further, more precisely, there exists a $\varepsilon>0$ small enough such that $\omega^{\lambda}\geq0$ in $B_{\lambda}(0)\setminus\{0\}$ for all $\lambda\in[\lambda_{0},\lambda_{0}+\varepsilon]$.

For that purpose, we will first show that
\begin{equation}\label{2-32}
  \omega^{\lambda_{0}}(x)>0, \,\,\,\,\,\, \forall \, x\in B_{\lambda_{0}}(0)\setminus\{0\}.
\end{equation}
Indeed, since we have assumed there exists a point $x^{0}\in B_{\lambda_{0}}(0)\setminus\{0\}$ such that $\omega^{\lambda_{0}}(x^{0})>0$, by continuity, there exists a small $\delta>0$ and a constant $c_{0}>0$ such that
\begin{equation}\label{2-33}
B_{\delta}(x^{0})\subset B_{\lambda_{0}}(0)\setminus\{0\} \,\,\,\,\,\, \text{and} \,\,\,\,\,\,
\omega^{\lambda_{0}}(x)\geq c_{0}>0, \,\,\,\,\,\,\,\, \forall \, x\in B_{\delta}(x^{0}).
\end{equation}

Since the positive solution $u$ to \eqref{PDE} also satisfies the integral equation \eqref{IE}, through direct calculations, we get
\begin{equation}\label{2-34}
  u(x)=C\int_{B_{\lambda_{0}}(0)}\frac{|y|^{a}}{|x-y|^{n-\alpha}}u^{p}(y)dy+C\int_{B_{\lambda_{0}}(0)}\frac{|y|^{a}}{\left|\frac{|y|}{\lambda_{0}}x-\frac{\lambda_{0}}{|y|}y\right|^{n-\alpha}}
  \left(\frac{\lambda_{0}}{|y|}\right)^{\tau}u_{\lambda_{0}}^{p}(y)dy
\end{equation}
for any $x\in\mathbb{R}^{n}$, where $\tau:=n+\alpha+2a-p(n-\alpha)>0$. By direct calculations, one can also verify that $u_{\lambda_{0}}$ satisfies the following integral equation
\begin{equation}\label{2-35}
  u_{\lambda_{0}}(x)=C\int_{\mathbb{R}^{n}}\frac{|y|^{a}}{|x-y|^{n-\alpha}}\left(\frac{\lambda_{0}}{|y|}\right)^{\tau}u_{\lambda_{0}}^{p}(y)dy
\end{equation}
for any $x\in\mathbb{R}^{n}\setminus\{0\}$, and hence, it follows immediately that
\begin{equation}\label{2-36}
  u_{\lambda_{0}}(x)=C\int_{B_{\lambda_{0}}(0)}\frac{|y|^{a}}{\left|\frac{|y|}{\lambda_{0}}x-\frac{\lambda_{0}}{|y|}y\right|^{n-\alpha}}u^{p}(y)dy
  +C\int_{B_{\lambda_{0}}(0)}\frac{|y|^{a}}{|x-y|^{n-\alpha}}
  \left(\frac{\lambda_{0}}{|y|}\right)^{\tau}u_{\lambda_{0}}^{p}(y)dy.
\end{equation}
From the integral equations \eqref{2-34} and \eqref{2-36}, one can derive that, for any $x\in B_{\lambda_{0}}(0)\setminus\{0\}$,
\begin{eqnarray}\label{2-37}
  &&\omega^{\lambda_{0}}(x)=u_{\lambda_{0}}(x)-u(x) \\
 \nonumber &=&C\int_{B_{\lambda_{0}}(0)}\Bigg(\frac{|y|^{a}}{|x-y|^{n-\alpha}}-\frac{|y|^{a}}{\left|\frac{|y|}{\lambda_{0}}x-\frac{\lambda_{0}}{|y|}y\right|^{n-\alpha}}\Bigg) \left(\left(\frac{\lambda_{0}}{|y|}\right)^{\tau}u_{\lambda_{0}}^{p}(y)-u^{p}(y)\right)dy\\
 \nonumber &>&C\int_{B_{\lambda_{0}}(0)}\Bigg(\frac{|y|^{a}}{|x-y|^{n-\alpha}}-\frac{|y|^{a}}{\left|\frac{|y|}{\lambda_{0}}x-\frac{\lambda_{0}}{|y|}y\right|^{n-\alpha}}\Bigg) \left(u_{\lambda_{0}}^{p}(y)-u^{p}(y)\right)dy\\
\nonumber &\geq& C\int_{B_{\lambda_{0}}(0)}\Bigg(\frac{|y|^{a}}{|x-y|^{n-\alpha}}-\frac{|y|^{a}}{\left|\frac{|y|}{\lambda_{0}}x-\frac{\lambda_{0}}{|y|}y\right|^{n-\alpha}}\Bigg)
\min\left\{u^{p-1}(y),u_{\lambda_{0}}^{p-1}(y)\right\}\omega^{\lambda_{0}}(y)dy\\
\nonumber &\geq& C\int_{B_{\delta}(x^{0})}\Bigg(\frac{|y|^{a}}{|x-y|^{n-\alpha}}-\frac{|y|^{a}}{\left|\frac{|y|}{\lambda_{0}}x-\frac{\lambda_{0}}{|y|}y\right|^{n-\alpha}}\Bigg)
\min\left\{u^{p-1}(y),u_{\lambda_{0}}^{p-1}(y)\right\}\omega^{\lambda_{0}}(y)dy>0,
\end{eqnarray}
thus we arrive at \eqref{2-32}. Furthermore, \eqref{2-37} also implies that there exists a $0<\eta<\lambda_{0}$ small enough such that, for any $x\in \overline{B_{\eta}(0)}\setminus\{0\}$,
\begin{equation}\label{2-38}
  \omega^{\lambda_{0}}(x)>c_{4}+C\int_{B_{\frac{\delta}{2}}(x^{0})}c_{3}^{a}\,c_{2}\,c_{1}^{p-1}c_{0} \, dy=:\widetilde{c}_{0}>0.
\end{equation}

Now we define
\begin{equation}\label{2-39}
  \tilde{l}_{0}:=\min_{\lambda\in[\lambda_{0},2\lambda_{0}]}l_{0}(\lambda)>0,
\end{equation}
where $l_{0}(\lambda)$ is given by Theorem \ref{NRP1}. For a fixed small $0<r_{0}<\frac{1}{2}\min\{\tilde{l}_{0},\lambda_{0}\}$, by \eqref{2-32} and \eqref{2-38}, we can define
\begin{equation}\label{2-40}
  m_{0}:=\inf_{x\in\overline{B_{\lambda_{0}-r_{0}}(0)}\setminus\{0\}}\omega^{\lambda_{0}}(x)>0.
\end{equation}

Since $u$ is uniformly continuous on arbitrary compact set $K\subset\mathbb{R}^{n}$ (say, $K=\overline{B_{4\lambda_{0}}(0)}$), we can deduce from \eqref{2-40} that, there exists a $0<\varepsilon_{1}<\frac{1}{2}\min\{\tilde{l}_{0},\lambda_{0}\}$ sufficiently small, such that, for any $\lambda\in[\lambda_{0},\lambda_{0}+\varepsilon_{1}]$,
\begin{equation}\label{2-41}
  \omega^{\lambda}(x)\geq\frac{m_{0}}{2}>0, \,\,\,\,\,\, \forall \, x\in\overline{B_{\lambda_{0}-r_{0}}(0)}\setminus\{0\}.
\end{equation}
In order to prove \eqref{2-41}, one should observe that \eqref{2-40} is equivalent to
\begin{equation}\label{2-42}
  |x|^{n-\alpha}u(x)-\lambda_{0}^{n-\alpha}u(x^{\lambda_{0}})\geq m_{0}\lambda_{0}^{n-\alpha}, \,\,\,\,\,\,\,\,\, \forall \, |x|\geq\frac{\lambda_{0}^{2}}{\lambda_{0}-r_{0}}.
\end{equation}
Since $u$ is uniformly continuous on $\overline{B_{4\lambda_{0}}(0)}$, we infer from \eqref{2-42} that there exists a $0<\varepsilon_{1}<\frac{1}{2}\min\{\tilde{l}_{0},\lambda_{0}\}$ sufficiently small, such that, for any $\lambda\in[\lambda_{0},\lambda_{0}+\varepsilon_{1}]$,
\begin{equation}\label{2-43}
  |x|^{n-\alpha}u(x)-\lambda^{n-\alpha}u(x^{\lambda})\geq \frac{m_{0}}{2}\lambda^{n-\alpha}, \,\,\,\,\,\,\,\,\, \forall \, |x|\geq\frac{\lambda^{2}}{\lambda_{0}-r_{0}},
\end{equation}
which is equivalent to \eqref{2-41}, hence we have proved \eqref{2-41}.

For any $\lambda\in[\lambda_{0},\lambda_{0}+\varepsilon_{1}]$, let $l:=\lambda-\lambda_{0}+r_{0}\in(0,\tilde{l}_{0})$, then it follows from \eqref{2-11} and \eqref{2-41} that the conditions \eqref{nrp1} in Theorem \ref{NRP1} are satisfied, hence we can deduce from (ii) in Theorem \ref{NRP1} that
\begin{equation}\label{2-44}
  \omega^{\lambda}(x)\geq0, \quad\quad \forall \,\, x\in A_{\lambda,l}.
\end{equation}
Therefore, we get from \eqref{2-41} and \eqref{2-44} that, $B^{-}_{\lambda}=\emptyset$ for all $\lambda\in[\lambda_{0},\lambda_{0}+\varepsilon_{1}]$, that is,
\begin{equation}\label{2-45}
  \omega^{\lambda}(x)\geq0, \,\,\,\,\,\,\, \forall \,\, x\in B_{\lambda}(0)\setminus\{0\},
\end{equation}
which contradicts with the definition \eqref{2-29} of $\lambda_{0}$. As a consequence, in the case $0<\lambda_{0}<+\infty$, \eqref{2-31} must hold true, that is,
\begin{equation}\label{2-46}
  \omega^{\lambda_{0}}\equiv0 \,\,\,\,\,\, \text{in} \,\,\, B_{\lambda_{0}}(0)\setminus\{0\}.
\end{equation}

However, by the second equality in \eqref{2-37} and \eqref{2-46}, we arrive at
\begin{eqnarray}\label{2-47}
 && 0=\omega^{\lambda_{0}}(x)=u_{\lambda_{0}}(x)-u(x)\\
 \nonumber &=&C\int_{B_{\lambda_{0}}(0)}\Bigg(\frac{|y|^{a}}{|x-y|^{n-\alpha}}-\frac{|y|^{a}}{\left|\frac{|y|}{\lambda_{0}}x-\frac{\lambda_{0}}{|y|}y\right|^{n-\alpha}}\Bigg) \left(\left(\frac{\lambda_{0}}{|y|}\right)^{\tau}-1\right)u^{p}(y)dy>0
\end{eqnarray}
for any $x\in B_{\lambda_{0}}(0)\setminus\{0\}$, which is absurd. Thus we must have $\lambda_{0}=+\infty$, that is,
\begin{equation}\label{2-48}
  u(x)\geq\left(\frac{\lambda}{|x|}\right)^{n-\alpha}u\left(\frac{\lambda^{2}x}{|x|^{2}}\right), \quad\quad \forall \,\, |x|\geq\lambda, \quad \forall \,\, 0<\lambda<+\infty.
\end{equation}
For arbitrary $|x|\geq1$, let $\lambda:=\sqrt{|x|}$, then \eqref{2-48} yields that
\begin{equation}\label{2-49}
  u(x)\geq\frac{1}{|x|^{\frac{n-\alpha}{2}}}u\left(\frac{x}{|x|}\right),
\end{equation}
and hence, we arrive at the following lower bound estimate:
\begin{equation}\label{2-50}
  u(x)\geq\left(\min_{x\in S_{1}}u(x)\right)\frac{1}{|x|^{\frac{n-\alpha}{2}}}:=\frac{C_{0}}{|x|^{\frac{n-\alpha}{2}}}, \quad\quad \forall \,\, |x|\geq1.
\end{equation}

The lower bound estimate \eqref{2-50} can be improved remarkably for $0<p<\frac{n+\alpha+2a}{n-\alpha}$ using the ``Bootstrap" iteration technique and the integral equation \eqref{IE}.

In fact, let $\mu_{0}:=\frac{n-\alpha}{2}$, we infer from the integral equation \eqref{IE} and \eqref{2-50} that, for $|x|\geq1$,
\begin{eqnarray}\label{2-51}
  u(x)&\geq&C\int_{2|x|\leq|y|\leq4|x|}\frac{1}{|x-y|^{n-\alpha}|y|^{p\mu_{0}-a}}dy \\
  \nonumber &\geq&\frac{C}{|x|^{n-\alpha}}\int_{2|x|\leq|y|\leq4|x|}\frac{1}{|y|^{p\mu_{0}-a}}dy \\
  \nonumber &\geq&\frac{C}{|x|^{n-\alpha}}\int^{4|x|}_{2|x|}r^{n-1-p\mu_{0}+a}dr \\
 \nonumber &\geq&\frac{C_{1}}{|x|^{p\mu_{0}-(a+\alpha)}}.
\end{eqnarray}
Let $\mu_{1}:=p\mu_{0}-(a+\alpha)$. Due to $0<p<\frac{n+\alpha+2a}{n-\alpha}$, our important observation is
\begin{equation}\label{2-52}
  \mu_{1}:=p\mu_{0}-(a+\alpha)<\mu_{0}.
\end{equation}
Thus we have obtained a better lower bound estimate than \eqref{2-50} after one iteration, that is,
\begin{equation}\label{2-53}
  u(x)\geq\frac{C_{1}}{|x|^{\mu_{1}}}, \quad\quad \forall \,\, |x|\geq1.
\end{equation}

For $k=0,1,2,\cdots$, define
\begin{equation}\label{2-54}
  \mu_{k+1}:=p\mu_{k}-(a+\alpha).
\end{equation}
Since $0<p<\frac{n+\alpha+2a}{n-\alpha}$, it is easy to see that the sequence $\{\mu_{k}\}$ is monotone decreasing with respect to $k$ and $n-p\mu_{k}+a>0$ for any $k=0,1,2,\cdots$. Continuing the above iteration process involving the integral equation \eqref{IE}, we have the following lower bound estimates for every $k=0,1,2,\cdots$,
\begin{equation}\label{2-55}
  u(x)\geq\frac{C_{k}}{|x|^{\mu_{k}}}, \quad\quad \forall \,\, |x|\geq1.
\end{equation}
Now Theorem \ref{lower0} follows easily from the obvious properties that as $k\rightarrow+\infty$,
\begin{equation}\label{2-56}
  \mu_{k}\rightarrow-\frac{a+\alpha}{1-p} \quad \text{if} \,\, 0<p<1;
  \quad\quad \mu_{k}\rightarrow-\infty \quad \text{if} \,\, 1\leq p<\frac{n+\alpha+2a}{n-\alpha}.
\end{equation}
This finishes our proof of Theorem \ref{lower0}.
\end{proof}

We have proved the nontrivial nonnegative solution $u$ to \eqref{PDE} is actually a positive solution which also satisfies the integral equation \eqref{IE}. For $0<p<\frac{n+\alpha+2a}{n-\alpha}$, the lower bound estimates in Theorem \ref{lower0} contradicts with the following integrability
\begin{equation}\label{2-57}
  C\int_{\mathbb{R}^{n}}\frac{u^{p}(x)}{|x|^{n-\alpha-a}}dx=u(0)<+\infty
\end{equation}
indicated by the integral equation \eqref{IE}. Therefore, we must have $u\equiv0$ in $\mathbb{R}^{n}$, that is, the unique nonnegative solution to PDE \eqref{PDE} is $u\equiv0$ in $\mathbb{R}^{n}$.

This concludes the proof of Theorem \ref{Thm0}.

\section{Proof of Theorem \ref{Thm1}}
In this section, we will prove Theorem \ref{Thm1} by applying contradiction arguments and the \emph{method of scaling spheres in integral forms}. Now suppose on the contrary that $u\geq0$ satisfies integral equations \eqref{IE1} but $u$ is not identically zero, then there exists some $\bar{x}\in\mathbb{R}^{n}$ such that $u(\bar{x})>0$. It follows from \eqref{IE1} immediately that
\begin{equation}\label{3-0-2}
  u(x)>0, \,\,\,\,\,\,\, \forall \,\, x\in\mathbb{R}^{n},
\end{equation}
i.e., $u$ is actually a positive solution in $\mathbb{R}^{n}$. Moreover, there exists a constant $C>0$, such that the solution $u$ satisfies the following lower bound:
\begin{equation}\label{3-1}
  u(x)\geq\frac{C}{|x|^{n-2m}} \,\,\,\,\,\,\,\,\,\,\,\,\, \text{for} \,\,\, |x|\geq1.
\end{equation}
Indeed, since $u>0$ satisfies the integral equation \eqref{IE1}, we can infer that
\begin{eqnarray}\label{3-2}
  u(x)&\geq& C_{n,m}\int_{|y|\leq\frac{1}{2}}\frac{|y|^{a}}{|x-y|^{n-2m}}u^{p}(y)dy\\
 \nonumber &\geq& \frac{C}{|x|^{n-2m}}\int_{|y|\leq\frac{1}{2}}|y|^{a}u^{p}(y)dy=:\frac{C}{|x|^{n-2m}}
\end{eqnarray}
for all $|x|\geq1$.

Next, we will apply the \emph{method of scaling spheres in integral forms} to show the following lower bound estimates for positive solution $u$, which contradict with the integral equations \eqref{IE1} for $0<p<\frac{n+2m+2a}{n-2m}$.
\begin{thm}\label{lower1}
Assume $\alpha=2m$ with $1\leq m<\frac{n}{2}$, $-2m<a<+\infty$, $0<p<\frac{n+2m+2a}{n-2m}$. Suppose $u$ is a positive solution to \eqref{IE1}, then it satisfies the following lower bound estimates: for $|x|\geq1$,
\begin{equation}\label{lb1-1}
  u(x)\geq C_{\kappa}|x|^{\kappa} \quad\quad \forall \, \kappa<\frac{2m+a}{1-p}, \quad\quad \text{if} \,\,\,\, 0<p<1;
\end{equation}
\begin{equation}\label{lb2-1}
  u(x)\geq C_{\kappa}|x|^{\kappa} \quad\quad \forall \, \kappa<+\infty, \quad\quad \text{if} \,\,\,\, 1\leq p<\frac{n+2m+2a}{n-2m}.
\end{equation}
\end{thm}
\begin{proof}
Given any $\lambda>0$, we first define the Kelvin transform of a function $u:\,\mathbb{R}^{n}\rightarrow\mathbb{R}$ centered at $0$ by
\begin{equation}\label{Kelvin1}
  u_{\lambda}(x)=\left(\frac{\lambda}{|x|}\right)^{n-2m}u\left(\frac{\lambda^{2}x}{|x|^{2}}\right)
\end{equation}
for arbitrary $x\in\mathbb{R}^{n}\setminus\{0\}$. It's obvious that the Kelvin transform $u_{\lambda}$ may have singularity at $0$ and $\lim_{|x|\rightarrow\infty}|x|^{n-2m}u_{\lambda}(x)=\lambda^{n-2m}u(0)>0$. By \eqref{Kelvin1}, one can infer from the regularity assumptions on $u$ that $u_{\lambda}\in C(\mathbb{R}^{n}\setminus\{0\})$.

Next, we will carry out the process of scaling spheres with respect to the origin $0\in\mathbb{R}^{n}$.

Let $\lambda>0$ be an arbitrary positive real number and let
\begin{equation}\label{3-1-0}
  \omega^{\lambda}(x):=u_{\lambda}(x)-u(x)
\end{equation}
for any $x\in B_{\lambda}(0)\setminus\{0\}$. By the definition of $u_{\lambda}$ and $\omega^{\lambda}$, we have
\begin{equation}\label{3-6}
\omega^{\lambda}(x)=u_{\lambda}(x)-u(x)=-\left(\frac{\lambda}{|x|}\right)^{n-2m}\omega^{\lambda}(x^{\lambda})=-\big(\omega^{\lambda}\big)_{\lambda}(x)
\end{equation}
for every $x\in B_{\lambda}(0)\setminus\{0\}$.

We will first show that, for $\lambda>0$ sufficiently small,
\begin{equation}\label{3-7}
  \omega^{\lambda}(x)\geq0, \,\,\,\,\,\, \forall \,\, x\in B_{\lambda}(0)\setminus\{0\}.
\end{equation}
Then, we start dilating the sphere $S_{\lambda}$ from near the origin $0$ outward as long as \eqref{3-7} holds, until its limiting position $\lambda=+\infty$ and derive lower bound estimates on $u$. Therefore, the scaling sphere process can be divided into two steps.

\emph{Step 1. Start dilating the sphere from near $\lambda=0$.} Define
\begin{equation}\label{3-8}
  B^{-}_{\lambda}:=\{x\in B_{\lambda}(0)\setminus\{0\} \, | \, \omega^{\lambda}(x)<0\}.
\end{equation}
We will show that, for $\lambda>0$ sufficiently small,
\begin{equation}\label{3-9}
  B^{-}_{\lambda}=\emptyset.
\end{equation}

Since $u$ is a positive solution to integral equations \eqref{IE1}, through direct calculations, we get
\begin{equation}\label{3-10}
  u(x)=C\int_{B_{\lambda}(0)}\frac{|y|^{a}}{|x-y|^{n-2m}}u^{p}(y)dy+C\int_{B_{\lambda}(0)}\frac{|y|^{a}}{\left|\frac{|y|}{\lambda}x-\frac{\lambda}{|y|}y\right|^{n-2m}}
  \left(\frac{\lambda}{|y|}\right)^{\tau}u_{\lambda}^{p}(y)dy
\end{equation}
for any $x\in\mathbb{R}^{n}$, where $\tau:=n+2m+2a-p(n-2m)>0$. By direct calculations, one can also verify that $u_{\lambda}$ satisfies the following integral equation
\begin{equation}\label{3-11}
  u_{\lambda}(x)=C\int_{\mathbb{R}^{n}}\frac{|y|^{a}}{|x-y|^{n-2m}}\left(\frac{\lambda}{|y|}\right)^{\tau}u_{\lambda}^{p}(y)dy
\end{equation}
for any $x\in\mathbb{R}^{n}\setminus\{0\}$, and hence, it follows immediately that
\begin{equation}\label{3-12}
  u_{\lambda}(x)=C\int_{B_{\lambda}(0)}\frac{|y|^{a}}{\left|\frac{|y|}{\lambda}x-\frac{\lambda}{|y|}y\right|^{n-2m}}u^{p}(y)dy
  +C\int_{B_{\lambda}(0)}\frac{|y|^{a}}{|x-y|^{n-2m}}
  \left(\frac{\lambda}{|y|}\right)^{\tau}u_{\lambda}^{p}(y)dy.
\end{equation}
From the integral equations \eqref{3-10} and \eqref{3-12}, one can derive that, for any $x\in B_{\lambda}^{-}$,
\begin{eqnarray}\label{3-13}
  &&0>\omega^{\lambda}(x)=u_{\lambda}(x)-u(x) \\
 \nonumber &=&C\int_{B_{\lambda}(0)}\Bigg(\frac{|y|^{a}}{|x-y|^{n-2m}}-\frac{|y|^{a}}{\left|\frac{|y|}{\lambda}x-\frac{\lambda}{|y|}y\right|^{n-2m}}\Bigg) \left(\left(\frac{\lambda}{|y|}\right)^{\tau}u_{\lambda}^{p}(y)-u^{p}(y)\right)dy\\
\nonumber &>& C\int_{B_{\lambda}^{-}}\Bigg(\frac{|y|^{a}}{|x-y|^{n-2m}}-\frac{|y|^{a}}{\left|\frac{|y|}{\lambda}x-\frac{\lambda}{|y|}y\right|^{n-2m}}\Bigg)
\max\left\{u^{p-1}(y),u_{\lambda}^{p-1}(y)\right\}\omega^{\lambda}(y)dy\\
\nonumber &\geq& C\int_{B_{\lambda}^{-}}\frac{|y|^{a}}{|x-y|^{n-2m}}\max\left\{u^{p-1}(y),u_{\lambda}^{p-1}(y)\right\}\omega^{\lambda}(y)dy.
\end{eqnarray}

Now we need the following Hardy-Littlewood-Sobolev inequality.
\begin{lem}\label{HL}(Hardy-Littlewood-Sobolev inequality)
Let $n\geq1$, $0<s<n$ and $1<p<q<\infty$ be such that $\frac{n}{q}=\frac{n}{p}-s$. Then we have
\begin{equation}\label{HLS}
  \Big\|\int_{\mathbb{R}^{n}}\frac{f(y)}{|x-y|^{n-s}}dy\Big\|_{L^{q}(\mathbb{R}^{n})}\leq C_{n,s,p,q}\|f\|_{L^{p}(\mathbb{R}^{n})}
\end{equation}
for all $f\in L^{p}(\mathbb{R}^{n})$.
\end{lem}

By Hardy-Littlewood-Sobolev inequality and \eqref{3-13}, we have, for any $\frac{n}{n-2m}<q<\infty$,
\begin{eqnarray}\label{3-14}
  \|\omega^{\lambda}\|_{L^{q}(B^{-}_{\lambda})}&\leq& C\left\||x|^{a}\max\left\{u^{p-1},u_{\lambda}^{p-1}\right\}\omega^{\lambda}\right\|_{L^{\frac{nq}{n+2mq}}(B^{-}_{\lambda})}\\
  \nonumber &\leq& C\left\||x|^{a}\max\left\{u^{p-1},u_{\lambda}^{p-1}\right\}\right\|_{L^{\frac{n}{2m}}(B^{-}_{\lambda})}\|\omega^{\lambda}\|_{L^{q}(B^{-}_{\lambda})}.
\end{eqnarray}
Since \eqref{3-2} implies $u_{\lambda}$ has positive lower bound near the origin $0$ and $a>-2m$, there exists a $\epsilon_{0}>0$ small enough, such that
\begin{equation}\label{3-15}
  C\left\||x|^{a}\max\left\{u^{p-1},u_{\lambda}^{p-1}\right\}\right\|_{L^{\frac{n}{2m}}(B^{-}_{\lambda})}\leq\frac{1}{2}
\end{equation}
for all $0<\lambda\leq\epsilon_{0}$, and hence \eqref{3-14} implies
\begin{equation}\label{3-16}
  \|\omega^{\lambda}\|_{L^{q}(B^{-}_{\lambda})}=0,
\end{equation}
which means $B^{-}_{\lambda}=\emptyset$. Therefore, we have proved for all $0<\lambda\leq\epsilon_{0}$, $B^{-}_{\lambda}=\emptyset$, that is,
\begin{equation}\label{3-17}
  \omega^{\lambda}(x)\geq0, \,\,\,\,\,\,\, \forall \, x\in B_{\lambda}(0)\setminus\{0\}.
\end{equation}
This completes Step 1.

\emph{Step 2. Dilate the sphere $S_{\lambda}$ outward until $\lambda=+\infty$ to derive lower bound estimates on $u$.} Step 1 provides us a start point to dilate the sphere $S_{\lambda}$ from near $\lambda=0$. Now we dilate the sphere $S_{\lambda}$ outward as long as \eqref{3-7} holds. Let
\begin{equation}\label{3-18}
  \lambda_{0}:=\sup\{\lambda>0\,|\, \omega^{\mu}\geq0 \,\, in \,\, B_{\mu}(0)\setminus\{0\}, \,\, \forall \, 0<\mu\leq\lambda\}\in(0,+\infty],
\end{equation}
and hence, one has
\begin{equation}\label{3-19}
  \omega^{\lambda_{0}}(x)\geq0, \quad\quad \forall \,\, x\in B_{\lambda_{0}}(0)\setminus\{0\}.
\end{equation}
In what follows, we will prove $\lambda_{0}=+\infty$ by contradiction arguments.

Suppose on contrary that $0<\lambda_{0}<+\infty$. In order to get a contradiction, we will first prove
\begin{equation}\label{3-20}
  \omega^{\lambda_{0}}(x)\equiv0, \,\,\,\,\,\,\forall \, x\in B_{\lambda_{0}}(0)\setminus\{0\}
\end{equation}
by using contradiction arguments.

Suppose on contrary that \eqref{3-20} does not hold, that is, $\omega^{\lambda_{0}}\geq0$ but $\omega^{\lambda_{0}}$ is not identically zero in $B_{\lambda_{0}}(0)\setminus\{0\}$, then there exists a $x^{0}\in B_{\lambda_{0}}(0)\setminus\{0\}$ such that $\omega^{\lambda_{0}}(x^{0})>0$. We will obtain a contradiction with \eqref{3-18} via showing that the sphere $S_{\lambda}$ can be dilated outward a little bit further, more precisely, there exists a $\varepsilon>0$ small enough such that $\omega^{\lambda}\geq0$ in $B_{\lambda}(0)\setminus\{0\}$ for all $\lambda\in[\lambda_{0},\lambda_{0}+\varepsilon]$.

For that purpose, we will first show that
\begin{equation}\label{3-21}
  \omega^{\lambda_{0}}(x)>0, \,\,\,\,\,\, \forall \, x\in B_{\lambda_{0}}(0)\setminus\{0\}.
\end{equation}
Indeed, since we have assumed there exists a point $x^{0}\in B_{\lambda_{0}}(0)\setminus\{0\}$ such that $\omega^{\lambda_{0}}(x^{0})>0$, by continuity, there exists a small $\delta>0$ and a constant $c_{0}>0$ such that
\begin{equation}\label{3-22}
B_{\delta}(x^{0})\subset B_{\lambda_{0}}(0)\setminus\{0\} \,\,\,\,\,\, \text{and} \,\,\,\,\,\,
\omega^{\lambda_{0}}(x)\geq c_{0}>0, \,\,\,\,\,\,\,\, \forall \, x\in B_{\delta}(x^{0}).
\end{equation}
From \eqref{3-22} and the integral equations \eqref{3-10} and \eqref{3-12}, one can derive that, for any $x\in B_{\lambda_{0}}(0)\setminus\{0\}$,
\begin{eqnarray}\label{3-23}
  &&\omega^{\lambda_{0}}(x)=u_{\lambda_{0}}(x)-u(x) \\
 \nonumber &=&C\int_{B_{\lambda_{0}}(0)}\Bigg(\frac{|y|^{a}}{|x-y|^{n-2m}}-\frac{|y|^{a}}{\left|\frac{|y|}{\lambda_{0}}x-\frac{\lambda_{0}}{|y|}y\right|^{n-2m}}\Bigg) \left(\left(\frac{\lambda_{0}}{|y|}\right)^{\tau}u_{\lambda_{0}}^{p}(y)-u^{p}(y)\right)dy\\
 \nonumber &>& C\int_{B_{\lambda_{0}}(0)}\Bigg(\frac{|y|^{a}}{|x-y|^{n-2m}}-\frac{|y|^{a}}{\left|\frac{|y|}{\lambda_{0}}x-\frac{\lambda_{0}}{|y|}y\right|^{n-2m}}\Bigg) \left(u_{\lambda_{0}}^{p}(y)-u^{p}(y)\right)dy\\
 \nonumber &\geq& C\int_{B_{\lambda_{0}}(0)}\Bigg(\frac{|y|^{a}}{|x-y|^{n-2m}}-\frac{|y|^{a}}{\left|\frac{|y|}{\lambda_{0}}x-\frac{\lambda_{0}}{|y|}y\right|^{n-2m}}\Bigg)
\min\left\{u^{p-1}(y),u_{\lambda_{0}}^{p-1}(y)\right\}\omega^{\lambda_{0}}(y)dy\\
\nonumber &\geq& C\int_{B_{\delta}(x^{0})}\Bigg(\frac{|y|^{a}}{|x-y|^{n-2m}}-\frac{|y|^{a}}{\left|\frac{|y|}{\lambda_{0}}x-\frac{\lambda_{0}}{|y|}y\right|^{n-2m}}\Bigg)
\min\left\{u^{p-1}(y),u_{\lambda_{0}}^{p-1}(y)\right\}\omega^{\lambda_{0}}(y)dy>0,
\end{eqnarray}
thus we arrive at \eqref{3-21}. Furthermore, \eqref{3-23} also implies that there exists a $0<\eta<\lambda_{0}$ small enough such that, for any $x\in \overline{B_{\eta}(0)}\setminus\{0\}$,
\begin{equation}\label{3-24}
  \omega^{\lambda_{0}}(x)>c_{4}+C\int_{B_{\frac{\delta}{2}}(x^{0})}c_{3}^{a}\,c_{2}\,c_{1}^{p-1}c_{0} \, dy=:\widetilde{c}_{0}>0.
\end{equation}

We fixed $0<r_{0}<\frac{1}{2}\lambda_{0}$ small enough, such that
\begin{equation}\label{3-25}
  C\left\||x|^{a}\max\left\{u^{p-1},u_{\lambda}^{p-1}\right\}\right\|_{L^{\frac{n}{2m}}(A_{\lambda_{0}+r_{0},2r_{0}})}\leq\frac{1}{2}
\end{equation}
for any $\lambda\in[\lambda_{0},\lambda_{0}+r_{0}]$, where the constant $C$ is the same as in \eqref{3-14} and the narrow region
\begin{equation}\label{3-26}
  A_{\lambda_{0}+r_{0},2r_{0}}:=\left\{x\in B_{\lambda_{0}+r_{0}}(0)\,|\,|x|>\lambda_{0}-r_{0}\right\}.
\end{equation}
By \eqref{3-13}, one can easily verify that inequality as \eqref{3-14} (with the same constant $C$) also holds for any $\lambda\in[\lambda_{0},\lambda_{0}+r_{0}]$, that is, for any $\frac{n}{n-2m}<q<\infty$,
\begin{equation}\label{3-27}
  \|\omega^{\lambda}\|_{L^{q}(B^{-}_{\lambda})}\leq C\left\||x|^{a}\max\left\{u^{p-1},u_{\lambda}^{p-1}\right\}\right\|_{L^{\frac{n}{2m}}(B^{-}_{\lambda})}\|\omega^{\lambda}\|_{L^{q}(B^{-}_{\lambda})}.
\end{equation}
From \eqref{3-21} and \eqref{3-24}, we can infer that
\begin{equation}\label{3-28}
  m_{0}:=\inf_{x\in\overline{B_{\lambda_{0}-r_{0}}(0)}\setminus\{0\}}\omega^{\lambda_{0}}(x)>0.
\end{equation}
Since $u$ is uniformly continuous on arbitrary compact set $K\subset\mathbb{R}^{n}$ (say, $K=\overline{B_{4\lambda_{0}}(0)}$), we can deduce from \eqref{3-28} that, there exists a $0<\varepsilon_{2}<r_{0}$ sufficiently small, such that, for any $\lambda\in[\lambda_{0},\lambda_{0}+\varepsilon_{2}]$,
\begin{equation}\label{3-29}
  \omega^{\lambda}(x)\geq\frac{m_{0}}{2}>0, \,\,\,\,\,\, \forall \, x\in\overline{B_{\lambda_{0}-r_{0}}(0)}\setminus\{0\}.
\end{equation}
The proof of \eqref{3-29} is entirely similar to that of \eqref{2-41}, so we omit the details.

By \eqref{3-29}, we know that for any $\lambda\in[\lambda_{0},\lambda_{0}+\varepsilon_{2}]$,
\begin{equation}\label{3-30}
  B_{\lambda}^{-}\subset A_{\lambda_{0}+r_{0},2r_{0}},
\end{equation}
and hence, estimates \eqref{3-25} and \eqref{3-27} yields
\begin{equation}\label{3-31}
  \|\omega^{\lambda}\|_{L^{q}(B^{-}_{\lambda})}=0.
\end{equation}
Therefore, for any $\lambda\in[\lambda_{0},\lambda_{0}+\varepsilon_{2}]$, we deduce from \eqref{3-31} that, $B^{-}_{\lambda}=\emptyset$, that is,
\begin{equation}\label{3-32}
  \omega^{\lambda}(x)\geq0, \,\,\,\,\,\,\, \forall \,\, x\in B_{\lambda}(0)\setminus\{0\},
\end{equation}
which contradicts with the definition \eqref{3-18} of $\lambda_{0}$. As a consequence, in the case $0<\lambda_{0}<+\infty$, \eqref{3-20} must hold true, that is,
\begin{equation}\label{3-33}
  \omega^{\lambda_{0}}\equiv0 \,\,\,\,\,\, \text{in} \,\,\, B_{\lambda_{0}}(0)\setminus\{0\}.
\end{equation}

However, by the second equality in \eqref{3-23} and \eqref{3-33}, we arrive at
\begin{eqnarray}\label{3-34}
 && 0=\omega^{\lambda_{0}}(x)=u_{\lambda_{0}}(x)-u(x)\\
 \nonumber &=&C\int_{B_{\lambda_{0}}(0)}\Bigg(\frac{|y|^{a}}{|x-y|^{n-2m}}-\frac{|y|^{a}}{\left|\frac{|y|}{\lambda_{0}}x-\frac{\lambda_{0}}{|y|}y\right|^{n-2m}}\Bigg) \left(\left(\frac{\lambda_{0}}{|y|}\right)^{\tau}-1\right)u^{p}(y)dy>0
\end{eqnarray}
for any $x\in B_{\lambda_{0}}(0)\setminus\{0\}$, which is absurd. Thus we must have $\lambda_{0}=+\infty$, that is,
\begin{equation}\label{3-35}
  u(x)\geq\left(\frac{\lambda}{|x|}\right)^{n-2m}u\left(\frac{\lambda^{2}x}{|x|^{2}}\right), \quad\quad \forall \,\, |x|\geq\lambda, \quad \forall \,\, 0<\lambda<+\infty.
\end{equation}
For arbitrary $|x|\geq1$, let $\lambda:=\sqrt{|x|}$, then \eqref{3-35} yields that
\begin{equation}\label{3-36}
  u(x)\geq\frac{1}{|x|^{\frac{n-2m}{2}}}u\left(\frac{x}{|x|}\right),
\end{equation}
and hence, we arrive at the following lower bound estimate:
\begin{equation}\label{3-37}
  u(x)\geq\left(\min_{x\in S_{1}}u(x)\right)\frac{1}{|x|^{\frac{n-2m}{2}}}:=\frac{C_{0}}{|x|^{\frac{n-2m}{2}}}, \quad\quad \forall \,\, |x|\geq1.
\end{equation}

The lower bound estimate \eqref{3-37} can be improved remarkably for $0<p<\frac{n+2m+2a}{n-2m}$ using the ``Bootstrap" iteration technique and the integral equations \eqref{IE1}.

In fact, let $\mu_{0}:=\frac{n-2m}{2}$, we infer from the integral equations \eqref{IE1} and \eqref{3-37} that, for $|x|\geq1$,
\begin{eqnarray}\label{3-38}
  u(x)&\geq&C\int_{2|x|\leq|y|\leq4|x|}\frac{1}{|x-y|^{n-2m}|y|^{p\mu_{0}-a}}dy \\
  \nonumber &\geq&\frac{C}{|x|^{n-2m}}\int_{2|x|\leq|y|\leq4|x|}\frac{1}{|y|^{p\mu_{0}-a}}dy \\
  \nonumber &\geq&\frac{C}{|x|^{n-2m}}\int^{4|x|}_{2|x|}r^{n-1-p\mu_{0}+a}dr \\
  \nonumber &\geq&\frac{C_{1}}{|x|^{p\mu_{0}-(a+2m)}}.
\end{eqnarray}
Let $\mu_{1}:=p\mu_{0}-(a+2m)$. Due to $0<p<\frac{n+2m+2a}{n-2m}$, our important observation is
\begin{equation}\label{3-39}
  \mu_{1}:=p\mu_{0}-(a+2m)<\mu_{0}.
\end{equation}
Thus we have obtained a better lower bound estimate than \eqref{3-37} after one iteration, that is,
\begin{equation}\label{3-40}
  u(x)\geq\frac{C_{1}}{|x|^{\mu_{1}}}, \quad\quad \forall \,\, |x|\geq1.
\end{equation}

For $k=0,1,2,\cdots$, define
\begin{equation}\label{3-41}
  \mu_{k+1}:=p\mu_{k}-(a+2m).
\end{equation}
Since $0<p<\frac{n+2m+2a}{n-2m}$, it is easy to see that the sequence $\{\mu_{k}\}$ is monotone decreasing with respect to $k$ and $n-p\mu_{k}+a>0$ for any $k=0,1,2,\cdots$. Continuing the above iteration process involving the integral equation \eqref{IE1}, we have the following lower bound estimates for every $k=0,1,2,\cdots$,
\begin{equation}\label{3-42}
  u(x)\geq\frac{C_{k}}{|x|^{\mu_{k}}}, \quad\quad \forall \,\, |x|\geq1.
\end{equation}
Now Theorem \ref{lower1} follows easily from the obvious properties that as $k\rightarrow+\infty$,
\begin{equation}\label{3-43}
 \mu_{k}\rightarrow-\frac{a+2m}{1-p} \quad \text{if} \,\, 0<p<1;
  \quad\quad \mu_{k}\rightarrow-\infty \quad \text{if} \,\, 1\leq p<\frac{n+2m+2a}{n-2m}.
\end{equation}
This finishes our proof of Theorem \ref{lower1}.
\end{proof}

We have proved the nontrivial nonnegative solution $u$ to integral equations \eqref{IE1} is actually a positive solution. For $0<p<\frac{n+2m+2a}{n-2m}$, the lower bound estimates in Theorem \ref{lower1} contradicts with the following integrability
\begin{equation}\label{3-44}
  C\int_{\mathbb{R}^{n}}\frac{u^{p}(x)}{|x|^{n-2m-a}}dx=u(0)<+\infty
\end{equation}
indicated by the integral equations \eqref{IE1}. Therefore, we must have $u\equiv0$ in $\mathbb{R}^{n}$, that is, the unique nonnegative solution to IEs \eqref{IE1} is $u\equiv0$ in $\mathbb{R}^{n}$.

This concludes the proof of Theorem \ref{Thm1}.

\section{Proof of Theorem \ref{Thm2}}
In this section, we will prove Theorem \ref{Thm2} via contradiction arguments and the \emph{(direct) method of scaling spheres}. Now suppose on the contrary that $u\geq0$ satisfies equation \eqref{PDE+} but $u$ is not identically zero, then there exists some $\bar{x}\in\mathbb{R}^{n}_{+}$ such that $u(\bar{x})>0$.

\subsection{Equivalence between PDE and IE}
In order to get a contradiction, we first need to show that the solution $u$ to PDE \eqref{PDE+} also satisfies the equivalent integral equation
\begin{equation}\label{IE+}
  u(x)=\int_{\mathbb{R}^{n}_{+}}G^{+}(x,y)|y|^{a}u^{p}(y)dy,
\end{equation}
where Green's function associated with $(-\Delta)^{\frac{\alpha}{2}}$ on $\mathbb{R}^{n}_{+}$ is given by
\begin{equation}\label{Green}
  G^{+}(x,y):=\frac{C_{n,\alpha}}{|x-y|^{n-\alpha}}\int^{\frac{4x_{n}y_{n}}{|x-y|^{2}}}_{0}\frac{b^{\frac{\alpha}{2}-1}}{(1+b)^{\frac{n}{2}}}db.
\end{equation}
\begin{thm}\label{equivalent+}
Assume $n>\alpha$, $0<\alpha\leq2$, $-\alpha<a<+\infty$ and $1\leq p<\infty$. Suppose $u$ is nonnegative solution to \eqref{PDE+}, then it also solves the integral equation \eqref{IE+}, and vice versa.
\end{thm}
\begin{proof}
First, assume $u$ is a nonnegative solutions to PDE \eqref{PDE+}, our goal is to show that $u$ also satisfies IE \eqref{IE+}. For arbitrary $R>0$, let $P_R:=(0,\cdots,0,R)$ and
\begin{equation}\label{4-2}
\tilde{u}_R(x)=\int_{B_R(P_R)}G_R^+(x,y)|y|^{a}u^{p}(y)dy,
\end{equation}
where Green's function $G_R^+$ for $(-\Delta)^{\frac{\alpha}{2}}$ on $B_R(P_R)$ is given by
\begin{equation}\label{4-3}
  G_R^+(x,y):=\frac{C_{n,\alpha}}{|x-y|^{n-\alpha}}\int^{\frac{t_{R}}{s_{R}}}_{0}\frac{b^{\frac{\alpha}{2}-1}}{(1+b)^{\frac{n}{2}}}db,
\end{equation}
where $t_{R}:=\left(1-\frac{|x-P_{R}|^{2}}{R^{2}}\right)\left(1-\frac{|y-P_{R}|^{2}}{R^{2}}\right)$ and $s_{R}:=\frac{|x-y|^{2}}{R^{2}}$.

Then, we can derive
\begin{equation}\label{4-4}\\\begin{cases}
(-\Delta)^{\frac{\alpha}{2}}\tilde{u}_R(x)=|x|^{a}u^{p}(x),\quad x\in B_R(P_R),\\
\tilde{u}_R(x)=0,\ \ \ \ \ \ \ x\notin B_R(P_R).
\end{cases}\end{equation}
Let $U_R(x)=u(x)-\tilde{u}_R(x)$, by PDE \eqref{PDE+}, we have
\begin{equation}\label{4-5}\\\begin{cases}
(-\Delta)^{\frac{\alpha}{2}}U_R(x)=0,\, \, \, \,\,\, x\in B_R(P_R),\\
U_R(x)\geq0, \quad\,\,\, x\notin B_R(P_R).
\end{cases}\end{equation}
By Lemma \ref{max} and maximum principles, for any $x\in B_R(P_R)$, we deduce
\begin{equation}\label{4-6}
  U_R(x)\geq0.
\end{equation}
Letting $R\rightarrow\infty$, we have
\begin{equation}\label{4-7}
u(x)\geq\int_{\mathbb{R}^n_+}G^+(x,y)|y|^{a}u^{p}(y)dy=:\tilde{u}(x).
\end{equation}
One can observe that $\tilde{u}$ is a solution of
\begin{equation}\label{4-8}\\\begin{cases}
(-\Delta)^{\frac{\alpha}{2}}\tilde{u}(x)=|x|^{a}u^{p}(x),  x\in \mathbb{R}^n_+,\\
\tilde{u}(x)=0,\ \ \ \ \ \ x\notin \mathbb{R}^n_+.
\end{cases}\end{equation}
Define $U:=u-\tilde{u}$, we have
\begin{equation}\label{4-9}\\\begin{cases}
(-\Delta)^{\frac{\alpha}{2}}U(x)=0,\ \ U(x)\geq 0,\ \ x\in \mathbb{R}^n_+,\\
U(x)=0,\ \ \ \ \ \ \ \ \ \ \ \ \ \ \ \ \ \ x\notin \mathbb{R}^n_+.
\end{cases}\end{equation}

Now we need the following Lemma from \cite{CLZC} on Liouville theorem for $\alpha$-harmonic functions on $\mathbb{R}^{n}_{+}$.
\begin{lem}\label{lemmas3}(\cite{CLZC})
Assume $0<\alpha<2$, $u\in L^{\infty}_{loc}(\overline{\mathbb{R}^{n}_{+}})\cap\mathcal{L}_{\alpha}(\mathbb{R}^{n})$ satisfies the following equation in the sense of distribution:
\begin{equation}\label{4-0}\\\begin{cases}
(-\Delta)^{\frac{\alpha}{2}}u(x)=0,\ \ u(x)\geq 0, \ \ x\in \mathbb{R}^n_+,\\
u(x)\equiv0, \ \ \ \ \ \ \ \ \ \ \ \ \ \ \ \ \ \ \ \ \ \ \ x\notin \mathbb{R}^n_+.
\end{cases}\end{equation}
Then, either $u\equiv 0$ or
\begin{equation}\label{4-1}\\\begin{cases}
u(x)=C{x_n}^{\frac{\alpha}{2}}, \ \ x\in \mathbb{R}^n_+,\\
u(x)=0, \ \ \ \ \ \ \ \ \ x\notin \mathbb{R}^n_+,
\end{cases}\end{equation}
for some positive constant $C>0$.
\end{lem}

The Liouville theorem for harmonic functions $u\in C(\overline{\mathbb{R}^{n}_{+}})$ in $\mathbb{R}^{n}_{+}$ with $u=0$ on $\partial\mathbb{R}^{n}_{+}$ (i.e., Lemma \ref{lemmas3} with $\alpha=2$) also holds (see e.g. \cite{D}). It can also be proved by using the Kelvin transform, asymptotic harmonic expansions in conjunction with the method of moving planes, we omit the details here. For generalizations to Liouville theorems for poly-harmonic functions on $\mathbb{R}^{n}_{+}$ with Navier boundary conditions, please refer to \cite{DQ2}.

From Lemma \ref{lemmas3} and Liouville theorem for harmonic functions on $\mathbb{R}^{n}_{+}$, we can deduce that either
\begin{equation}\label{eq2-1+}
  U(x)=0, \ \ \forall x\in \mathbb{R}^n,
\end{equation}
or there exists a positive constant $C_0$ such that
\begin{equation}\label{eq2-2+}\\\begin{cases}
U(x)=C_0{x_n}^{\frac{\alpha}{2}}, \ \ \ \ \ \ x\in \mathbb{R}^n_+,\\
U(x)=0, \ \ \ \ \ x\notin \mathbb{R}^n_+.
\end{cases}\end{equation}

We can obtain a contradiction in the second case by deriving a lower bound estimates of Green's function $G^+$. In fact, for each fixed $x\in\mathbb{R}^n_+$, by \eqref{Green}, there exists a $R_{x}>0$ sufficiently large, such that, for any $|y|\geq R$, one can derive
\begin{equation}\label{eq2-3+}\begin{split}
G^+(x,y)&=\frac{C_{n,\alpha}}{|x-y|^{n-\alpha}}\int^{\frac{4x_{n}y_{n}}{|x-y|^{2}}}_{0}\frac{b^{\frac{\alpha}{2}-1}}{(1+b)^{\frac{n}{2}}}db\\
&\geq \frac{C}{|x-y|^{n-\alpha}}\int^{\frac{4x_{n}y_{n}}{|x-y|^{2}}}_{0}b^{\frac{\alpha}{2}-1}db\\
&\geq \frac{C}{|x-y|^{n-\alpha}}\left(\frac{x_{n}y_{n}}{|x-y|^{2}}\right)^{\frac{\alpha}{2}}\\
&\geq \frac{C(x_{n}y_{n})^{\frac{\alpha}{2}}}{|x-y|^{n}}.
\end{split}\end{equation}
Therefore, for each fixed $x=(0,\cdots,0,x_{n})\in\mathbb{R}^n_+$, we get from \eqref{eq2-3+} and $a>-\alpha$ that
\begin{equation}\label{eq2-4+}\\\begin{split}
u(x)\geq \tilde{u}(x)&=\int_{\mathbb{R}_+^n}G^+(x,y)|y|^{a}u^p(y)dy\\
&\geq C_{0}^{p}\int_{\mathbb{R}_+^n}G^+(x,y)|y|^{a}(y_n^{\frac{\alpha}{2}})^p dy\\
&\geq C_1x_{n}^{\frac{\alpha}{2}}\int_{\mathbb{R}_+^n\setminus{B_R(0)}}\frac{y_n^\frac{\alpha(p+1)}{2}|y|^{a}}{|x-y|^n}dy\\
&\geq C_2x_{n}^{\frac{\alpha}{2}}\int_R^\infty y_n^\frac{\alpha(p+1)}{2}\int_R^\infty \frac{r^{n-2}(r^2+y_n^2)^{\frac{a}{2}}}{(r^2+|x_n-y_n|^2)^{\frac{n}{2}}}drdy_n\\
&\geq C_{3}x_{n}^{\frac{\alpha}{2}}\int_R^\infty y_n^{\frac{\alpha(p+1)}{2}+a-1}\int_1^\infty \frac{s^{n-2}}{(s^2+1)^{\frac{n-a}{2}}}dsdy_n=+\infty,
\end{split}\end{equation}
which is absurd. This implies that the second case \eqref{eq2-2+} can not happen. Therefore, we can derive from \eqref{eq2-1+} that
\begin{equation}\label{eq2-5+}
  u(x)=\tilde{u}(x)=\int_{\mathbb{R}^n_+}G^+(x,y)|y|^{a}u^p(y)dy,
\end{equation}
that is, $u$ also satisfies the integral equation \eqref{IE+}.

Conversely, assume that $u$ is a nonnegative classical solution of integral equation \eqref{IE+} on $\mathbb{R}^{n}_{+}$, then
\begin{eqnarray}\label{2e25+}
(-\Delta)^{\frac{\alpha}{2}}u(x)
\nonumber &=& \int_{\mathbb{R}^{n}_{+}}{\left[(-\Delta)^{\frac{\alpha}{2}}G^{+}(x,y)\right]}|y|^{a}u^{p}(y)dy
\\
\nonumber &=& \int_{\mathbb{R}^{n}_{+}}\delta(x-y)|y|^{a}u^{p}(y)dy=|x|^{a}u^{p}(x),
\end{eqnarray}
that is, $u$ also solves the PDE \eqref{PDE+}. This completes the proof of equivalence between PDE \eqref{PDE+} and IE \eqref{IE+}.
\end{proof}

Since the nonnegative solution $u$ with $u(\bar{x})>0$ also satisfy the integral equation \eqref{IE+}, it is actually a positive solution in $\mathbb{R}^{n}_{+}$, that is,
\begin{equation}\label{2-0+}
  u(x)>0 \quad\quad\quad \text{in} \,\, \mathbb{R}^{n}_{+}.
\end{equation}
Moreover, there exist a constant $C>0$, such that the solution $u$ satisfies the following lower bound:
\begin{equation}\label{2-1+}
  u(x)\geq C\frac{x_{n}^{\frac{\alpha}{2}}}{|x|^{n}} \,\,\,\,\,\,\,\,\,\,\,\,\, \text{for} \,\,\,\,|x|\geq1, \,\,\,\, x\in\mathbb{R}^{n}_{+}.
\end{equation}
Indeed, since $u>0$ also satisfy the integral equation \eqref{IE+}, we can deduce that
\begin{eqnarray}\label{2-2+}
  u(x)&=& C_{n,\alpha}\int_{\mathbb{R}^{n}_{+}}\frac{|y|^{a}}{|x-y|^{n-\alpha}}\left(\int^{\frac{4x_{n}y_{n}}{|x-y|^{2}}}_{0}\frac{b^{\frac{\alpha}{2}-1}}{(1+b)^{\frac{n}{2}}}db\right)
  u^{p}(y)dy\\
 \nonumber &\geq& C\int_{\mathbb{R}^{n}_{+}\cap\{|y|\leq\frac{1}{2}\}}\frac{|y|^{a}}{|x-y|^{n-\alpha}}
 \left(\int^{\frac{4x_{n}y_{n}}{|x-y|^{2}}}_{0}b^{\frac{\alpha}{2}-1}db\right)u^{p}(y)dy \\
 \nonumber &\geq& C\int_{\mathbb{R}^{n}_{+}\cap\{|y|\leq\frac{1}{2}\}}\frac{|y|^{a}}{|x-y|^{n-\alpha}}
 \left(\frac{4x_{n}y_{n}}{|x-y|^{2}}\right)^{\frac{\alpha}{2}}u^{p}(y)dy \\
 \nonumber &\geq& C\frac{x_{n}^{\frac{\alpha}{2}}}{|x|^{n}}\int_{\mathbb{R}^{n}_{+}\cap\{|y|\leq\frac{1}{2}\}}y_{n}^{\frac{\alpha}{2}}|y|^{a}u^{p}(y)dy
 =:C\frac{x_{n}^{\frac{\alpha}{2}}}{|x|^{n}}
\end{eqnarray}
for all $|x|\geq1$ and $x\in\mathbb{R}^{n}_{+}$.

\subsection{The method of scaling spheres}

In this subsection, we will apply the \emph{(direct) method of scaling spheres} to show the following lower bound estimates for positive solution $u$, which contradict with the integral equation \eqref{IE+} for $1\leq p<\frac{n+\alpha+2a}{n-\alpha}$.
\begin{thm}\label{lower2}
Assume $n>\alpha$, $0<\alpha\leq2$, $-\alpha<a<+\infty$ and $1\leq p<\frac{n+\alpha+2a}{n-\alpha}$. Suppose $u$ is a positive solution to \eqref{PDE+}, then it satisfies the following lower bound estimates: for all $x\in\mathbb{R}^{n}_{+}$ satisfying $|x|\geq1$ and $x_{n}\geq\frac{|x|}{\sqrt{n}}$,
\begin{equation}\label{lb2+}
  u(x)\geq C_{\kappa}|x|^{\kappa} \quad\quad \forall \, \kappa<+\infty.
\end{equation}
\end{thm}
\begin{proof}
Given any $\lambda>0$, we first define the Kelvin transform of a function $u:\,\mathbb{R}^{n}\rightarrow\mathbb{R}$ centered at $0$ by
\begin{equation}\label{Kelvin+}
  u_{\lambda}(x)=\left(\frac{\lambda}{|x|}\right)^{n-\alpha}u\left(\frac{\lambda^{2}x}{|x|^{2}}\right)
\end{equation}
for arbitrary $x\in\mathbb{R}^{n}\setminus\{0\}$. It's obvious that the Kelvin transform $u_{\lambda}$ may have singularity at $0$ and $\lim_{|x|\rightarrow\infty}|x|^{n-\alpha}u_{\lambda}(x)=\lambda^{n-\alpha}u(0)=0$. By \eqref{Kelvin+}, one can infer from the regularity assumptions on $u$ that $u_{\lambda}\in\mathcal{L}_{\alpha}(\mathbb{R}^{n})\cap C^{1,1}_{loc}(\mathbb{R}^{n}_{+})\cap C(\overline{\mathbb{R}^{n}_{+}}\setminus\{0\})$ if $0<\alpha<2$ and $u_{\lambda}\in C^{2}(\mathbb{R}^{n}_{+})\cap C(\overline{\mathbb{R}^{n}_{+}}\setminus\{0\})$ if $\alpha=2$. Furthermore, we can deduce from \eqref{PDE+} and \eqref{Kelvin+} that
\begin{equation}\label{2-3+}
   (-\Delta)^{\frac{\alpha}{2}}u_{\lambda}(x)=\left(\frac{\lambda}{|x|}\right)^{n+\alpha}\frac{\lambda^{2a}}{|x|^{a}}u^{p}\left(\frac{\lambda^{2}x}{|x|^{2}}\right)
   =\left(\frac{\lambda}{|x|}\right)^{\tau}|x|^{a}u_{\lambda}^{p}(x), \,\,\,\,\,\,\,\,\, x\in\mathbb{R}^{n}_{+},
\end{equation}
where $\tau:=n+\alpha+2a-p(n-\alpha)>0$.

Next, we will carry out the process of scaling spheres in $\overline{\mathbb{R}^{n}_{+}}$ with respect to the origin $0\in\mathbb{R}^{n}$. For this purpose, we need some definitions.

Let $\lambda>0$ be an arbitrary positive real number and let the scaling half sphere be
\begin{equation}\label{2-4+}
  S^{+}_{\lambda}:=\{x\in\overline{\mathbb{R}^{n}_{+}}:\, |x|=\lambda\}.
\end{equation}
We define the reflection of $x$ about the half sphere $S^{+}_{\lambda}$ by $x^{\lambda}:=\frac{\lambda^{2}x}{|x|^{2}}$ and define
\begin{equation}\label{2-5+}
  B^{+}_{\lambda}(0):=B_{\lambda}(0)\cap\mathbb{R}^{n}_{+}, \quad\quad \widetilde{B^{+}_{\lambda}}(0):=\{x\in\mathbb{R}^{n}_{+}: \, x^{\lambda}\in B^{+}_{\lambda}(0)\}.
\end{equation}

Let $\omega^{\lambda}(x):=u_{\lambda}(x)-u(x)$ for any $x\in \overline{B^{+}_{\lambda}(0)}\setminus\{0\}$. By the definition of $u_{\lambda}$ and $\omega^{\lambda}$, we have
\begin{equation}\label{2-6+}
\omega^{\lambda}(x)=u_{\lambda}(x)-u(x)=-\left(\frac{\lambda}{|x|}\right)^{n-\alpha}\omega^{\lambda}(x^{\lambda})=-\big(\omega^{\lambda}\big)_{\lambda}(x)
\end{equation}
for every $x\in\overline{B^{+}_{\lambda}(0)}\setminus\{0\}$.

We will first show that, for $\lambda>0$ sufficiently small,
\begin{equation}\label{2-7+}
  \omega^{\lambda}(x)\geq0, \,\,\,\,\,\, \forall \,\, x\in B^{+}_{\lambda}(0).
\end{equation}
Then, we start dilating the half sphere $S^{+}_{\lambda}$ from near the origin $0$ outward as long as \eqref{2-7+} holds, until its limiting position $\lambda=+\infty$ and derive lower bound estimates on $u$ in a cone. Therefore, the scaling sphere process can be divided into two steps.

\emph{Step 1. Start dilating the half sphere $S^{+}_{\lambda}$ from near $\lambda=0$.} Define
\begin{equation}\label{2-8+}
  (B^{+}_{\lambda})^{-}:=\{x\in B^{+}_{\lambda}(0) \, | \, \omega^{\lambda}(x)<0\}.
\end{equation}
We will show through contradiction arguments that, for $\lambda>0$ sufficiently small,
\begin{equation}\label{2-9+}
  (B^{+}_{\lambda})^{-}=\emptyset.
\end{equation}

Suppose \eqref{2-9+} does not hold, that is, $(B^{+}_{\lambda})^{-}\neq\emptyset$ and hence $\omega^{\lambda}$ is negative somewhere in $B^{+}_{\lambda}(0)$. For arbitrary $x\in B^{+}_{\lambda}(0)$, we get from \eqref{2-3+} that
\begin{eqnarray}\label{2-10+}
  (-\Delta)^{\frac{\alpha}{2}}\omega^{\lambda}(x)&=&\left(\frac{\lambda}{|x|}\right)^{\tau}|x|^{a}u_{\lambda}^{p}(x)-|x|^{a}u^{p}(x) \\
 \nonumber &\geq& |x|^{a}\left(u_{\lambda}^{p}(x)-u^{p}(x)\right)=p|x|^{a}\xi_{\lambda}^{p-1}(x)\omega^{\lambda}(x),
\end{eqnarray}
where $\xi_{\lambda}(x)$ is valued between $u(x)$ and $u_{\lambda}(x)$ by mean value theorem. Therefore, for all $x\in (B^{+}_{\lambda})^{-}$,
\begin{equation}\label{2-11+}
  (-\Delta)^{\frac{\alpha}{2}}\omega^{\lambda}(x)\geq p|x|^{a}u^{p-1}(x)\omega^{\lambda}(x)=:c_{\lambda}(x)\omega^{\lambda}(x).
\end{equation}

Now we need the following Theorem on the \emph{Narrow region principle} in $\overline{\mathbb{R}^{n}_{+}}$.
\begin{thm}\label{NRP1+}(Narrow region principle)
Assume $n>\alpha$, $0<\alpha\leq2$, $-\alpha<a<+\infty$ and $1\leq p<+\infty$. Let $\lambda>0$ and $A^{+}_{\lambda,l_{1},l_{2}}:=\{x\in B^{\lambda}_{+}(0)\,|\,|x|>\lambda-l_{1} \,\,\, \text{or} \,\,\, x_{n}<l_{2}\}$ be a narrow region in $\mathbb{R}^{n}_{+}$ with $0\leq l_{1},l_{2}\leq\lambda$ and $l_{1}+l_{2}\leq\lambda$. If $0<\alpha<2$, suppose that $\omega^{\lambda}\in\mathcal{L}_{\alpha}(\mathbb{R}^{n})\cap C^{1,1}_{loc}(A^{+}_{\lambda,l_{1},l_{2}})$ and satisfies
\begin{equation}\label{nrp1+}\\\begin{cases}
(-\Delta)^{\frac{\alpha}{2}}\omega^{\lambda}(x)-c_{\lambda}(x)\omega^{\lambda}(x)\geq0 \,\,\,\,\, \text{in} \,\,\, A^{+}_{\lambda,l_{1},l_{2}}\cap(B^{+}_{\lambda})^{-},\\
\text{negative minimum of} \,\, \omega^{\lambda}\,\, \text{is attained in the interior of}\,\, B^{+}_{\lambda}(0) \,\,\text{if} \,\,\, (B^{+}_{\lambda})^{-}\neq\emptyset,\\
\text{negative minimum of} \,\,\, \omega^{\lambda} \,\,\, \text{cannot be attained in} \,\,\, B^{+}_{\lambda}(0)\setminus A^{+}_{\lambda,l_{1},l_{2}},
\end{cases}\end{equation}
where $c_{\lambda}(x):=p|x|^{a}u^{p-1}(x)$. If $\alpha=2$, assume that $a>-1$, $\omega^{\lambda}\in C^{2}(A^{+}_{\lambda,l_{1},l_{2}})$ and satisfies
\begin{equation}\label{nrp2}\\\begin{cases}
-\Delta\omega^{\lambda}(x)-c_{\lambda}(x)\omega^{\lambda}(x)\geq0 \,\,\,\,\, \text{in} \,\,\, (B^{+}_{\lambda})^{-},\\
(B^{+}_{\lambda})^{-}\subset A^{+}_{\lambda,l_{1},l_{2}} \quad \text{if} \,\,\, (B^{+}_{\lambda})^{-}\neq\emptyset.
\end{cases}\end{equation}
Then, we have \\
(i) there exists a sufficiently small constant $\delta_{0}>0$, such that, for all $0<\lambda\leq\delta_{0}$,
\begin{equation}\label{nrp1-1+}
  \omega^{\lambda}(x)\geq0, \,\,\,\,\,\, \forall \,x\in A^{+}_{\lambda,l_{1},l_{2}};
\end{equation}
(ii) for arbitrarily fixed $\theta\in(0,1)$, there exists a sufficiently small $l_{0}>0$ depending on $\lambda$ continuously, such that, for all $l_{1}$ and $l_{2}$ satisfying $0<l_{1},l_{2}\leq l_{0}$, $l_{1}<\theta\lambda$ and $l_{1}+l_{2}<\lambda$,
\begin{equation}\label{nrp1-2+}
  \omega^{\lambda}(x)\geq0, \,\,\,\,\,\, \forall \,x\in A^{+}_{\lambda,l_{1},l_{2}}.
\end{equation}
\end{thm}
\begin{proof}
For $0<\alpha<2$, the idea of our proof is similar to that of the Narrow region principle in $\mathbb{R}^{n}$ (Theorem 2.2 in \cite{CLZ} and Theorem \ref{NRP1}). The key difference and ingredient is that, since
\begin{eqnarray}\label{narrow}
  A^{+}_{\lambda,l_{1},l_{2}}&=&R^{+}_{\lambda,l_{1},l_{2}}\cup L^{+}_{\lambda,l_{1},l_{2}} \\
 \nonumber &:=&\{x\in B^{\lambda}_{+}(0)\,|\,\lambda-l_{1}<|x|<\lambda\}\cup\{x\in B^{\lambda}_{+}(0)\,|\, 0<x_{n}<l_{2}\},
\end{eqnarray}
we need to discuss two different cases in order to get a contradiction, that is, the negative minimum point $\tilde{x}\in R^{+}_{\lambda,l_{1},l_{2}}\cap(B^{+}_{\lambda})^{-}$ or $\tilde{x}\in L^{+}_{\lambda,l_{1},l_{2}}\cap(B^{+}_{\lambda})^{-}$, respectively. Another main difference is that, we use the following identity
\begin{equation}\label{5-1}
  (-\Delta)^{\frac{\alpha}{2}}\omega^{\lambda}(\tilde{x})=C\left\{\int_{B^{+}_{\lambda}(0)}
  +\int_{\mathbb{R}^{n}_{+}\setminus B^{+}_{\lambda}(0)}
  +\int_{\mathbb{R}^{n}\setminus\mathbb{R}^{n}_{+}}\right\}\frac{\omega^{\lambda}(\tilde{x})-\omega^{\lambda}(y)}{|\tilde{x}-y|^{n+\alpha}}dy
\end{equation}
instead of the identity (2.15) in \cite{CLZ}. Then, by similar calculations as in (2.16) and (2.17) in \cite{CLZ}, we will get the following estimate
\begin{equation}\label{5-2}
  (-\Delta)^{\frac{\alpha}{2}}\omega^{\lambda}(\tilde{x})\leq C\left\{\int_{\mathbb{R}^{n}_{+}\setminus B^{+}_{\lambda}(0)}\frac{1}{|\tilde{x}-y|^{n+\alpha}}dy
  +\int_{\mathbb{R}^{n}\setminus\mathbb{R}^{n}_{+}}\frac{1}{|\tilde{x}-y|^{n+\alpha}}dy\right\}\omega^{\lambda}(\tilde{x})
\end{equation}
instead of (2.18) in \cite{CLZ}.

We first show (i). One can get from \eqref{5-2} that
\begin{equation}\label{5-4}
  (-\Delta)^{\frac{\alpha}{2}}\omega^{\lambda}(\tilde{x})\leq C\omega^{\lambda}(\tilde{x})\int_{\big(\mathbb{R}^{n}\setminus\mathbb{R}^{n}_{+}\big)\cap\big(B_{4\tilde{x}_{n}}(\tilde{x})\setminus B_{\tilde{x}_{n}}(\tilde{x})\big)}\frac{1}{|\tilde{x}-y|^{n+\alpha}}dy\leq \frac{C}{(\tilde{x}_{n})^{\alpha}}\omega^{\lambda}(\tilde{x}).
\end{equation}
Due to $a+\alpha>0$, combining \eqref{5-4} with \eqref{nrp1+} yields that
\begin{equation}\label{5-8}
  0<C\leq(\tilde{x}_{n})^{\alpha}|\tilde{x}|^{a}\leq|\tilde{x}|^{a+\alpha}\leq\lambda^{a+\alpha},
\end{equation}
which will lead immediately to a contradiction if $0<\lambda\leq\delta_{0}$ for some constant $\delta_{0}$ small enough.

Next, we prove (ii). If $\tilde{x}\in R^{+}_{\lambda,l_{1},l_{2}}\cap(B^{+}_{\lambda})^{-}$, then one infers from \eqref{5-2} that
\begin{equation}\label{5-3}
  (-\Delta)^{\frac{\alpha}{2}}\omega^{\lambda}(\tilde{x})\leq C\omega^{\lambda}(\tilde{x})\int_{\big(\mathbb{R}^{n}_{+}\setminus B^{+}_{\lambda}(0)\big)\cap\big(B_{4l_{1}}(\tilde{x})\setminus B_{l_{1}}(\tilde{x})\big)}\frac{1}{|\tilde{x}-y|^{n+\alpha}}dy\leq \frac{C}{l_{1}^{\alpha}}\omega^{\lambda}(\tilde{x}).
\end{equation}
Noting that $(1-\theta)\lambda<\lambda-l_{1}<|x|<\lambda$ for any $x\in R^{+}_{\lambda,l_{1},l_{2}}$, by \eqref{nrp1+} and \eqref{5-3}, we obtain
\begin{equation}\label{eq-a2}
  \frac{C}{l_{1}^{\alpha}}\leq p\max\left\{1,(1-\theta)^{a}\right\}\lambda^{a}M_{\lambda}^{p-1},
\end{equation}
where $M_{\lambda}:=\sup_{B_{\lambda}(0)}u<+\infty$. Consequently, one can derive a contradiction from \eqref{eq-a2} if $0<l_{1}\leq l_{0}$ for some sufficiently small $l_{0}$ depending on $\lambda$ continuously. If $\tilde{x}\in L^{+}_{\lambda,l_{1},l_{2}}\cap(B^{+}_{\lambda})^{-}$, then due to $a+\alpha>0$, combining \eqref{5-4} with \eqref{nrp1+} yields that
\begin{equation}\label{5-8}
  0<C\leq(\tilde{x}_{n})^{\alpha}|\tilde{x}|^{a}\leq\max\left\{l_{2}^{\alpha+a},l_{2}^{\alpha}\lambda^{a}\right\},
\end{equation}
which will lead immediately to a contradiction if $0<l_{2}\leq l_{0}$ for some sufficiently small $l_{0}$ depending on $\lambda$ continuously.

For $\alpha=2$, assume that $(B^{+}_{\lambda})^{-}\neq\emptyset$, then \eqref{nrp2} implies that $(B^{+}_{\lambda})^{-}\subset A^{+}_{\lambda,l_{1},l_{2}}$ and
\begin{equation}\label{eq-a0}\\\begin{cases}
\Delta\left(-\omega^{\lambda}\right)(x)\geq p|x|^{a}u^{p-1}(x)\omega^{\lambda}(x) \qquad \text{in} \,\,\, (B^{+}_{\lambda})^{-},\\
\omega^{\lambda}(x)=0 \qquad \text{on} \,\,\, \partial(B^{+}_{\lambda})^{-}.
\end{cases}\end{equation}
By the Alexandroff-Bakelman-Pucci maximum principle (see e.g. Theorem 9.1 in \cite{GT}), we can infer from \eqref{eq-a0} that
\begin{equation}\label{eq-a1}
  \sup_{(B^{+}_{\lambda})^{-}}\left(-\omega^{\lambda}\right)\leq C\||x|^{a}u^{p-1}\omega^{\lambda}\|_{L^{n}\left((B^{+}_{\lambda})^{-}\right)}
  \leq C\||x|^{a}u^{p-1}\|_{L^{n}\left((B^{+}_{\lambda})^{-}\right)}\sup_{(B^{+}_{\lambda})^{-}}\left(-\omega^{\lambda}\right),
\end{equation}
where the constant $C$ depends only on $n$, $\lambda$ and tends to $0$ as $\lambda\rightarrow0$. Since $a>-1$, if $m\left[(B^{+}_{\lambda})^{-}\right]$ is small enough such that $C\||x|^{a}u^{p-1}\|_{L^{n}\left((B^{+}_{\lambda})^{-}\right)}\leq\frac{1}{2}$, then it follows from \eqref{eq-a1} that $\sup_{(B^{+}_{\lambda})^{-}}\left(-\omega^{\lambda}\right)=0$ and hence $(B^{+}_{\lambda})^{-}=\emptyset$. As a consequence, we have $(B^{+}_{\lambda})^{-}=\emptyset$ provided that $\lambda$ is sufficiently small or $l_{1}$ and $l_{2}$ are small enough.

This finishes the proof of Theorem \ref{NRP1+}.
\end{proof}

For any $0<\lambda\leq\delta_{0}$, let $l_{1}:=\lambda\in[0,\lambda]$ and $l_{2}:=0\in[0,\lambda]$, then it follows from \eqref{2-11+} and
\begin{equation}\label{6-1}
  \liminf_{x\rightarrow0}\omega^{\lambda}(x)\geq0
\end{equation}
that the conditions \eqref{nrp1+} and \eqref{nrp2} in Theorem \ref{NRP1+} are satisfied, hence we can deduce from (i) in Theorem \ref{NRP1+} that
\begin{equation}\label{2-27+}
  \omega^{\lambda}(x)\geq0, \quad\quad \forall \,\, x\in A^{+}_{\lambda,l_{1},l_{2}}.
\end{equation}
Therefore, we have proved for all $0<\lambda\leq\delta_{0}$, $(B^{+}_{\lambda})^{-}=\emptyset$, that is,
\begin{equation}\label{2-28+}
  \omega^{\lambda}(x)\geq0, \,\,\,\,\,\,\, \forall \, x\in B^{+}_{\lambda}(0).
\end{equation}
This completes Step 1.

\emph{Step 2. Dilate the half sphere $S^{+}_{\lambda}$ outward until $\lambda=+\infty$ to derive lower bound estimates on $u$ in a cone.} Step 1 provides us a start point to dilate the half sphere $S^{+}_{\lambda}$ from near $\lambda=0$. Now we dilate the half sphere $S^{+}_{\lambda}$ outward as long as \eqref{2-7+} holds. Let
\begin{equation}\label{2-29+}
  \lambda_{0}:=\sup\{\lambda>0\,|\, \omega^{\mu}\geq0 \,\, in \,\, B^{+}_{\mu}(0), \,\, \forall \, 0<\mu\leq\lambda\}\in(0,+\infty],
\end{equation}
and hence, one has
\begin{equation}\label{2-30+}
  \omega^{\lambda_{0}}(x)\geq0, \quad\quad \forall \,\, x\in B^{+}_{\lambda_{0}}(0).
\end{equation}
In what follows, we will prove $\lambda_{0}=+\infty$ by contradiction arguments.

Suppose on contrary that $0<\lambda_{0}<+\infty$. In order to get a contradiction, we will first prove
\begin{equation}\label{2-31+}
  \omega^{\lambda_{0}}(x)\equiv0, \,\,\,\,\,\,\forall \, x\in B^{+}_{\lambda_{0}}(0)
\end{equation}
by using the \emph{Narrow region principle} (Theorem \ref{NRP1+}) and contradiction arguments.

Suppose on contrary that \eqref{2-31+} does not hold, that is, $\omega^{\lambda_{0}}\geq0$ but $\omega^{\lambda_{0}}$ is not identically zero in $B^{+}_{\lambda_{0}}(0)$, then there exists a $x^{0}\in B^{+}_{\lambda_{0}}(0)$ such that $\omega^{\lambda_{0}}(x^{0})>0$. We will obtain a contradiction with \eqref{2-29+} via showing that the half sphere $S^{+}_{\lambda}$ can be dilated outward a little bit further, more precisely, there exists a $\varepsilon>0$ small enough such that $\omega^{\lambda}\geq0$ in $B^{+}_{\lambda}(0)$ for all $\lambda\in[\lambda_{0},\lambda_{0}+\varepsilon]$.

For that purpose, we will first show that
\begin{equation}\label{2-32+}
  \omega^{\lambda_{0}}(x)>0, \,\,\,\,\,\, \forall \, x\in B^{+}_{\lambda_{0}}(0).
\end{equation}
Indeed, since we have assumed there exists a point $x^{0}\in B^{+}_{\lambda_{0}}(0)$ such that $\omega^{\lambda_{0}}(x^{0})>0$, by continuity, there exists a small $\delta>0$ and a constant $c_{0}>0$ such that
\begin{equation}\label{2-33+}
B_{\delta}(x^{0})\subset B^{+}_{\lambda_{0}}(0) \,\,\,\,\,\, \text{and} \,\,\,\,\,\,
\omega^{\lambda_{0}}(x)\geq c_{0}>0, \,\,\,\,\,\,\,\, \forall \, x\in B_{\delta}(x^{0}).
\end{equation}

Since the positive solution $u$ to \eqref{PDE+} also satisfies the integral equation \eqref{IE+}, through direct calculations, we get
\begin{equation}\label{2-34+}
  u(x)=\int_{B^{+}_{\lambda_{0}}(0)}G^{+}(x,y)|y|^{a}u^{p}(y)dy+\int_{B^{+}_{\lambda_{0}}(0)}G^{+}(x,y^{\lambda_{0}})
  \left(\frac{\lambda_{0}}{|y|}\right)^{\tau+n-\alpha}|y|^{a}u_{\lambda_{0}}^{p}(y)dy
\end{equation}
for any $x\in\mathbb{R}^{n}$, where $\tau:=n+\alpha+2a-p(n-\alpha)>0$. By direct calculations, one can also verify that $u_{\lambda_{0}}$ satisfies the following integral equation
\begin{equation}\label{2-35+}
  u_{\lambda_{0}}(x)=\int_{\mathbb{R}^{n}_{+}}G^{+}(x,y)\left(\frac{\lambda_{0}}{|y|}\right)^{\tau}|y|^{a}u_{\lambda_{0}}^{p}(y)dy
\end{equation}
for any $x\in\mathbb{R}^{n}\setminus\{0\}$, and hence, it follows immediately that
\begin{eqnarray}\label{2-36+}
  u_{\lambda_{0}}(x)&=&\int_{B^{+}_{\lambda_{0}}(0)}G^{+}(x,y^{\lambda_{0}})\left(\frac{\lambda_{0}}{|y|}\right)^{n-\alpha}|y|^{a}u^{p}(y)dy \\
 \nonumber \quad\quad &&+\int_{B^{+}_{\lambda_{0}}(0)}G^{+}(x,y)\left(\frac{\lambda_{0}}{|y|}\right)^{\tau}|y|^{a}u_{\lambda_{0}}^{p}(y)dy.
\end{eqnarray}
Observe that for any $x,y\in B^{+}_{\lambda_{0}}(0)$,
\begin{equation}\label{pointwise}
  |x-y|<\left|\frac{|y|}{\lambda_{0}}x-\frac{\lambda_{0}}{|y|}y\right|,
\end{equation}
and Green's function $G^{+}(x,y)$ is monotone decreasing about the distance $|x-y|$ provided the product $x_{n}y_{n}$ remains unchanged, one easily verifies
\begin{equation}\label{monotone}
  G^{+}(x,y)-\left(\frac{\lambda_{0}}{|y|}\right)^{n-\alpha}G^{+}(x,y^{\lambda_{0}})=G^{+}(x,y)-G^{+}\left(\frac{|y|}{\lambda_{0}}x,\frac{\lambda_{0}}{|y|}y\right)>0
\end{equation}
for any $x,y\in B^{+}_{\lambda_{0}}(0)$. Therefore, from the integral equations \eqref{2-34+} and \eqref{2-36+}, one can derive that, for any $x\in B^{+}_{\lambda_{0}}(0)$,
\begin{eqnarray}\label{2-37+}
  &&\omega^{\lambda_{0}}(x)=u_{\lambda_{0}}(x)-u(x) \\
 \nonumber &=&\int_{B^{+}_{\lambda_{0}}(0)}\Bigg(G^{+}(x,y)-\left(\frac{\lambda_{0}}{|y|}\right)^{n-\alpha}G^{+}(x,y^{\lambda_{0}})\Bigg) |y|^{a}\left(\left(\frac{\lambda_{0}}{|y|}\right)^{\tau}u_{\lambda_{0}}^{p}(y)-u^{p}(y)\right)dy\\
 \nonumber &>&\int_{B^{+}_{\lambda_{0}}(0)}\Bigg(G^{+}(x,y)-\left(\frac{\lambda_{0}}{|y|}\right)^{n-\alpha}G^{+}(x,y^{\lambda_{0}})\Bigg) |y|^{a}\left(u_{\lambda_{0}}^{p}(y)-u^{p}(y)\right)dy\\
\nonumber &\geq&p\int_{B^{+}_{\lambda_{0}}(0)}\Bigg(G^{+}(x,y)-\left(\frac{\lambda_{0}}{|y|}\right)^{n-\alpha}G^{+}(x,y^{\lambda_{0}})\Bigg)
u^{p-1}(y)|y|^{a}\omega^{\lambda_{0}}(y)dy\\
\nonumber &\geq&p\int_{B_{\delta}(x^{0})}\Bigg(G^{+}(x,y)-\left(\frac{\lambda_{0}}{|y|}\right)^{n-\alpha}G^{+}(x,y^{\lambda_{0}})\Bigg)
u^{p-1}(y)|y|^{a}\omega^{\lambda_{0}}(y)dy>0,
\end{eqnarray}
thus we arrive at \eqref{2-32+}.

Now we define
\begin{equation}\label{2-39+}
  \tilde{l}_{0}:=\min_{\lambda\in[\lambda_{0},2\lambda_{0}]}l_{0}(\lambda)>0,
\end{equation}
where $l_{0}(\lambda)$ is given by Theorem \ref{NRP1+}. For a fixed small $0<r_{0}<\frac{1}{4}\min\{\tilde{l}_{0},\lambda_{0}\}$, by \eqref{2-32+}, we can define
\begin{equation}\label{2-40+}
  m_{0}:=\inf_{x\in B^{+}_{\lambda_{0}}(0)\setminus A^{+}_{\lambda_{0},r_{0},r_{0}}}\omega^{\lambda_{0}}(x)>0.
\end{equation}

Since $u$ is uniformly continuous on arbitrary compact set $K\subset\overline{\mathbb{R}^{n}_{+}}$ (say, $K=\overline{B^{+}_{2\lambda_{0}}(0)}$), we can deduce from \eqref{2-40+} that, there exists a $0<\varepsilon_{3}<\frac{1}{4}\min\{\tilde{l}_{0},\lambda_{0}\}$ sufficiently small, such that, for any $\lambda\in[\lambda_{0},\lambda_{0}+\varepsilon_{3}]$,
\begin{equation}\label{2-41+}
  \omega^{\lambda}(x)\geq\frac{m_{0}}{2}>0, \,\,\,\,\,\, \forall \, x\in B^{+}_{\lambda_{0}}(0)\setminus A^{+}_{\lambda_{0},r_{0},r_{0}}.
\end{equation}
The proof of \eqref{2-41+} is completely similar to that of \eqref{2-41}, so we omit the details.

For any $\lambda\in[\lambda_{0},\lambda_{0}+\varepsilon_{3}]$, let $l_{1}:=\lambda-\lambda_{0}+r_{0}\in(0,\tilde{l}_{0})$ and $l_{2}:=r_{0}\in(0,\tilde{l}_{0})$, then it follows from \eqref{2-11+}, \eqref{2-41+} and
\begin{equation}\label{6-2}
  \liminf_{x\rightarrow0}\omega^{\lambda}(x)\geq0
\end{equation}
that the conditions \eqref{nrp1+} and \eqref{nrp2} in Theorem \ref{NRP1+} are satisfied, hence we can deduce from (ii) in Theorem \ref{NRP1+} that
\begin{equation}\label{2-44+}
  \omega^{\lambda}(x)\geq0, \quad\quad \forall \,\, x\in A^{+}_{\lambda,l_{1},l_{2}}.
\end{equation}
Therefore, we get from \eqref{2-41+} and \eqref{2-44+} that, $(B^{+}_{\lambda})^{-}=\emptyset$ for all $\lambda\in[\lambda_{0},\lambda_{0}+\varepsilon_{3}]$, that is,
\begin{equation}\label{2-45+}
  \omega^{\lambda}(x)\geq0, \,\,\,\,\,\,\, \forall \,\, x\in B^{+}_{\lambda}(0),
\end{equation}
which contradicts with the definition \eqref{2-29+} of $\lambda_{0}$. As a consequence, in the case $0<\lambda_{0}<+\infty$, \eqref{2-31+} must hold true, that is,
\begin{equation}\label{2-46+}
  \omega^{\lambda_{0}}\equiv0 \,\,\,\,\,\, \text{in} \,\,\, B^{+}_{\lambda_{0}}(0).
\end{equation}

However, by the second equality in \eqref{2-37+} and \eqref{2-46+}, we arrive at
\begin{eqnarray}\label{2-47+}
 && 0=\omega^{\lambda_{0}}(x)=u_{\lambda_{0}}(x)-u(x)\\
 \nonumber &=&\int_{B^{+}_{\lambda_{0}}(0)}\Bigg(G^{+}(x,y)-\left(\frac{\lambda_{0}}{|y|}\right)^{n-\alpha}G^{+}(x,y^{\lambda_{0}})\Bigg) \left(\left(\frac{\lambda_{0}}{|y|}\right)^{\tau}-1\right)|y|^{a}u^{p}(y)dy>0
\end{eqnarray}
for any $x\in B^{+}_{\lambda_{0}}(0)$, which is absurd. Thus we must have $\lambda_{0}=+\infty$, that is,
\begin{equation}\label{2-48+}
  u(x)\geq\left(\frac{\lambda}{|x|}\right)^{n-\alpha}u\left(\frac{\lambda^{2}x}{|x|^{2}}\right), \quad\quad \forall \,\, |x|\geq\lambda, \quad x\in\mathbb{R}^{n}_{+}, \quad \forall \,\, 0<\lambda<+\infty.
\end{equation}
For arbitrary $|x|\geq1$, $x\in\mathbb{R}^{n}_{+}$, let $\lambda:=\sqrt{|x|}$, then \eqref{2-48+} yields that
\begin{equation}\label{2-49+}
  u(x)\geq\frac{1}{|x|^{\frac{n-\alpha}{2}}}u\left(\frac{x}{|x|}\right),
\end{equation}
and hence, we arrive at the following lower bound estimate:
\begin{equation}\label{2-50+}
  u(x)\geq\left(\min_{x\in S^{+}_{1},\,x_{n}\geq\frac{1}{\sqrt{n}}}u(x)\right)\frac{1}{|x|^{\frac{n-\alpha}{2}}}:=\frac{C_{0}}{|x|^{\frac{n-\alpha}{2}}}, \quad\quad \forall \,\, |x|\geq1, \,\,\, x_{n}\geq\frac{|x|}{\sqrt{n}}.
\end{equation}

The lower bound estimate \eqref{2-50+} can be improved remarkably for $1\leq p<\frac{n+\alpha+2a}{n-\alpha}$ using the ``Bootstrap" iteration technique and the integral equation \eqref{IE+}.

In fact, let $\mu_{0}:=\frac{n-\alpha}{2}$, we infer from the integral equation \eqref{IE+} and \eqref{2-50+} that, for any $|x|\geq1$ and $x_{n}\geq\frac{|x|}{\sqrt{n}}$,
\begin{eqnarray}\label{2-51+}
  u(x)&\geq&C\int_{2|x|\leq|y|\leq4|x|,\,y_{n}\geq\frac{|y|}{\sqrt{n}}}\frac{1}{|x-y|^{n-\alpha}}
  \left(\int_{0}^{\frac{4x_{n}y_{n}}{|x-y|^{2}}}\frac{b^{\frac{\alpha}{2}-1}}{(1+b)^{\frac{n}{2}}}db\right)\frac{1}{|y|^{p\mu_{0}-a}}dy \\
  \nonumber &\geq&\frac{C}{|x|^{n-\alpha}}\int_{2|x|\leq|y|\leq4|x|,\,y_{n}\geq\frac{|y|}{\sqrt{n}}}
  \left(\int_{0}^{\frac{1}{2n}}\frac{b^{\frac{\alpha}{2}-1}}{(1+b)^{\frac{n}{2}}}db\right)\frac{1}{|y|^{p\mu_{0}-a}}dy \\
  \nonumber &\geq&\frac{C}{|x|^{n-\alpha}}\int^{4|x|}_{2|x|}r^{n-1-p\mu_{0}+a}dr \\
  \nonumber &\geq&\frac{C_{1}}{|x|^{p\mu_{0}-(a+\alpha)}}.
\end{eqnarray}
Let $\mu_{1}:=p\mu_{0}-(a+\alpha)$. Due to $1\leq p<\frac{n+\alpha+2a}{n-\alpha}$, our important observation is
\begin{equation}\label{2-52+}
  \mu_{1}:=p\mu_{0}-(a+\alpha)<\mu_{0}.
\end{equation}
Thus we have obtained a better lower bound estimate than \eqref{2-50+} after one iteration, that is,
\begin{equation}\label{2-53+}
  u(x)\geq\frac{C_{1}}{|x|^{\mu_{1}}}, \quad\quad \forall \,\, |x|\geq1, \,\,\, x_{n}\geq\frac{|x|}{\sqrt{n}}.
\end{equation}

For $k=0,1,2,\cdots$, define
\begin{equation}\label{2-54+}
  \mu_{k+1}:=p\mu_{k}-(a+\alpha).
\end{equation}
Since $1\leq p<\frac{n+\alpha+2a}{n-\alpha}$, it is easy to see that the sequence $\{\mu_{k}\}$ is monotone decreasing with respect to $k$ and $n-p\mu_{k}+a>0$ for any $k=0,1,2,\cdots$. Continuing the above iteration process involving the integral equation \eqref{IE+}, we have the following lower bound estimates for every $k=0,1,2,\cdots$,
\begin{equation}\label{2-55+}
  u(x)\geq\frac{C_{k}}{|x|^{\mu_{k}}}, \quad\quad \forall \,\, |x|\geq1, \,\,\, x_{n}\geq\frac{|x|}{\sqrt{n}}.
\end{equation}
Now Theorem \ref{lower2} follows easily from the obvious properties that as $k\rightarrow+\infty$,
\begin{equation}\label{2-56+}
   \mu_{k}\rightarrow-\infty \quad \text{if} \,\, 1\leq p<\frac{n+\alpha+2a}{n-\alpha}.
\end{equation}
This finishes our proof of Theorem \ref{lower2}.
\end{proof}

We have proved the nontrivial nonnegative solution $u$ to \eqref{PDE+} is actually a positive solution which also satisfies the integral equation \eqref{IE+}. For $1\leq p<\frac{n+\alpha+2a}{n-\alpha}$, the lower bound estimates in Theorem \ref{lower2} contradicts with the following integrability indicated by the integral equation \eqref{IE+}, that is,
\begin{eqnarray}\label{2-57+}
  +\infty&>&u(e_{n})\geq \int_{|y|\geq2,\,y_{n}\geq\frac{|y|}{\sqrt{n}}}\frac{C}{|e_{n}-y|^{n-\alpha}}
  \left(\int_{0}^{\frac{4y_{n}}{|e_{n}-y|^{2}}}\frac{b^{\frac{\alpha}{2}-1}}{(1+b)^{\frac{n}{2}}}db\right)|y|^{a}u^{p}(y)dy \\
  \nonumber &\geq&\int_{|y|\geq2,\,y_{n}\geq\frac{|y|}{\sqrt{n}}}\frac{C}{|y|^{n-\alpha}}
  \left(\int_{0}^{\frac{1}{\sqrt{n}|y|}}\frac{b^{\frac{\alpha}{2}-1}}{(1+b)^{\frac{n}{2}}}db\right)|y|^{a}u^{p}(y)dy \\
  \nonumber &\geq&C\int_{|y|\geq2,\,y_{n}\geq\frac{|y|}{\sqrt{n}}}\frac{u^{p}(y)}{|y|^{n-\frac{\alpha}{2}-a}}dy,
\end{eqnarray}
where $e_{n}:=(0,\cdots,0,1)\in\mathbb{R}^{n}_{+}$. Therefore, we must have $u\equiv0$ in $\overline{\mathbb{R}^{n}_{+}}$, that is, the unique nonnegative solution to PDE \eqref{PDE+} is $u\equiv0$ in $\overline{\mathbb{R}^{n}_{+}}$.

This concludes our proof of Theorem \ref{Thm2}.

\section{Proof of Theorem \ref{Thm3}}
In this section, we will prove Theorem \ref{Thm3} via contradiction arguments and the \emph{method of scaling spheres in integral forms}. Now suppose on the contrary that $u\geq0$ satisfies the integral equations \eqref{IE++} but $u$ is not identically zero, then there exists some $\bar{x}\in\mathbb{R}^{n}_{+}$ such that $u(\bar{x})>0$. It follows immediately from \eqref{IE++} that
\begin{equation}\label{8-2}
  u(x)>0, \,\,\,\,\,\,\, \forall \,\, x\in\mathbb{R}^{n}_{+},
\end{equation}
i.e., $u$ is actually a positive solution to IEs \eqref{IE++} on $\mathbb{R}^{n}_{+}$. Furthermore, there exists a constant $C>0$, such that the solution $u$ satisfies the following lower bound:
\begin{equation}\label{2-1++}
  u(x)\geq C\frac{x_{n}}{|x|^{n-2m+2}} \,\,\,\,\,\,\,\,\,\,\,\,\, \text{for} \,\,\,\,|x|\geq1, \,\,\,\, x\in\mathbb{R}^{n}_{+}.
\end{equation}
Indeed, since $u>0$ satisfies the integral equations \eqref{IE++} on $\mathbb{R}^{n}_{+}$, we can deduce that
\begin{eqnarray}\label{2-2++}
  u(x)&=& C_{n,m}\int_{\mathbb{R}^{n}_{+}}\left(\frac{1}{|x-y|^{n-2m}}-\frac{1}{|\bar{x}-y|^{n-2m}}\right)(y_{n})^{a}u^{p}(y)dy\\
 \nonumber &\geq& C\int_{\mathbb{R}^{n}_{+}\cap\{|y|\leq\frac{1}{2}\}}\frac{x_{n}y_{n}}{|x-y|^{n-2m+2}}(y_{n})^{a}u^{p}(y)dy \\
 \nonumber &\geq& C\frac{x_{n}}{|x|^{n-2m+2}}\int_{\mathbb{R}^{n}_{+}\cap\{|y|\leq\frac{1}{2}\}}(y_{n})^{a+1}u^{p}(y)dy
 =:C\frac{x_{n}}{|x|^{n-2m+2}}
\end{eqnarray}
for all $|x|\geq1$ and $x\in\mathbb{R}^{n}_{+}$.

Next, we will apply the \emph{method of scaling spheres in integral forms} to show the following lower bound estimates for positive solution $u$, which contradict with the integral equations \eqref{IE++} for $1\leq p<p_{c}(a):=\frac{n+2m+2a}{n-2m}$.
\begin{thm}\label{lower3}
Assume $1\leq m<\frac{n}{2}$, $-\frac{2m}{n}<a<+\infty$ and $1\leq p<\frac{n+2m+2a}{n-2m}$. Suppose $u$ is a positive solution to IEs \eqref{IE++}, then it satisfies the following lower bound estimates: for all $x\in\mathbb{R}^{n}_{+}$ satisfying $|x|\geq1$ and $x_{n}\geq\frac{|x|}{\sqrt{n}}$,
\begin{equation}\label{lb2++}
  u(x)\geq C_{\kappa}|x|^{\kappa} \quad\quad \forall \, \kappa<+\infty.
\end{equation}
\end{thm}
\begin{proof}
Given any $\lambda>0$, we first define the Kelvin transform of a function $u:\,\mathbb{R}^{n}\rightarrow\mathbb{R}$ centered at $0$ by
\begin{equation}\label{Kelvin++}
  u_{\lambda}(x)=\left(\frac{\lambda}{|x|}\right)^{n-2m}u\left(\frac{\lambda^{2}x}{|x|^{2}}\right)
\end{equation}
for arbitrary $x\in\mathbb{R}^{n}\setminus\{0\}$. It's obvious that the Kelvin transform $u_{\lambda}$ may have singularity at $0$ and $\lim_{|x|\rightarrow\infty}|x|^{n-2m}u_{\lambda}(x)=\lambda^{n-2m}u(0)=0$. By \eqref{Kelvin++}, one can infer from the regularity assumptions on $u$ that $u_{\lambda}\in C(\overline{\mathbb{R}^{n}_{+}}\setminus\{0\})$.

Next, we will carry out the process of scaling spheres in $\overline{\mathbb{R}^{n}_{+}}$ with respect to the origin $0\in\mathbb{R}^{n}$.

Let $\lambda>0$ be an arbitrary positive real number and let
\begin{equation}\label{10-0}
  \omega^{\lambda}(x):=u_{\lambda}(x)-u(x)
\end{equation}
for any $x\in \overline{B^{+}_{\lambda}(0)}\setminus\{0\}$. By the definition of $u_{\lambda}$ and $\omega^{\lambda}$, we have
\begin{equation}\label{2-6++}
\omega^{\lambda}(x)=u_{\lambda}(x)-u(x)=-\left(\frac{\lambda}{|x|}\right)^{n-2m}\omega^{\lambda}(x^{\lambda})=-\big(\omega^{\lambda}\big)_{\lambda}(x)
\end{equation}
for every $x\in\overline{B^{+}_{\lambda}(0)}\setminus\{0\}$.

We will first show that, for $\lambda>0$ sufficiently small,
\begin{equation}\label{2-7++}
  \omega^{\lambda}(x)\geq0, \,\,\,\,\,\, \forall \,\, x\in B^{+}_{\lambda}(0).
\end{equation}
Then, we start dilating the half sphere $S^{+}_{\lambda}$ from near the origin $0$ outward as long as \eqref{2-7++} holds, until its limiting position $\lambda=+\infty$ and derive lower bound estimates on $u$ in a cone. Therefore, the scaling sphere process can be divided into two steps.

\emph{Step 1. Start dilating the half sphere $S^{+}_{\lambda}$ from near $\lambda=0$.} Define
\begin{equation}\label{2-8++}
  (B^{+}_{\lambda})^{-}:=\{x\in B^{+}_{\lambda}(0) \, | \, \omega^{\lambda}(x)<0\}.
\end{equation}
We will show that, for $\lambda>0$ sufficiently small,
\begin{equation}\label{2-9++}
  (B^{+}_{\lambda})^{-}=\emptyset.
\end{equation}

Since $u$ is a positive solution to integral equations \eqref{IE++}, through direct calculations, we get
\begin{equation}\label{2-34++}
  u(x)=\int_{B^{+}_{\lambda}(0)}G_{m}^{+}(x,y)(y_{n})^{a}u^{p}(y)dy+\int_{B^{+}_{\lambda}(0)}G_{m}^{+}(x,y^{\lambda})
  \left(\frac{\lambda}{|y|}\right)^{\tau+n-2m}(y_{n})^{a}u_{\lambda}^{p}(y)dy
\end{equation}
for any $x\in\overline{\mathbb{R}^{n}_{+}}$, where $\tau:=n+2m+2a-p(n-2m)>0$. By direct calculations, one can also verify that $u_{\lambda}$ satisfies the following integral equation
\begin{equation}\label{2-35++}
  u_{\lambda}(x)=\int_{\mathbb{R}^{n}_{+}}G_{m}^{+}(x,y)\left(\frac{\lambda}{|y|}\right)^{\tau}(y_{n})^{a}u_{\lambda}^{p}(y)dy
\end{equation}
for any $x\in\overline{\mathbb{R}^{n}_{+}}\setminus\{0\}$, and hence, it follows immediately that
\begin{eqnarray}\label{2-36++}
  u_{\lambda}(x)&=&\int_{B^{+}_{\lambda}(0)}G_{m}^{+}(x,y^{\lambda})\left(\frac{\lambda}{|y|}\right)^{n-2m}(y_{n})^{a}u^{p}(y)dy \\
 \nonumber \quad\quad &&+\int_{B^{+}_{\lambda}(0)}G_{m}^{+}(x,y)\left(\frac{\lambda}{|y|}\right)^{\tau}(y_{n})^{a}u_{\lambda}^{p}(y)dy.
\end{eqnarray}
Observe that for any $x,y\in B^{+}_{\lambda}(0)$,
\begin{equation}\label{pointwise++}
  |x-y|<\left|\frac{|y|}{\lambda}x-\frac{\lambda}{|y|}y\right|,
\end{equation}
and Green's function $G_{m}^{+}(x,y)$ is monotone decreasing about the distance $|x-y|$ provided the product $x_{n}y_{n}$ remains unchanged, one easily verifies
\begin{equation}\label{monotone++}
  G_{m}^{+}(x,y)-\left(\frac{\lambda}{|y|}\right)^{n-2m}G_{m}^{+}(x,y^{\lambda})=G_{m}^{+}(x,y)-G_{m}^{+}\left(\frac{|y|}{\lambda}x,\frac{\lambda}{|y|}y\right)>0
\end{equation}
for any $x,y\in B^{+}_{\lambda}(0)$. Therefore, from the integral equations \eqref{2-34++} and \eqref{2-36++}, one can derive that, for any $x\in B^{+}_{\lambda}(0)$,
\begin{eqnarray}\label{2-37++}
  &&\omega^{\lambda}(x)=u_{\lambda}(x)-u(x) \\
 \nonumber &=&\int_{B^{+}_{\lambda}(0)}\Bigg(G_{m}^{+}(x,y)-\left(\frac{\lambda}{|y|}\right)^{n-2m}G_{m}^{+}(x,y^{\lambda})\Bigg) (y_{n})^{a}\left(\left(\frac{\lambda}{|y|}\right)^{\tau}u_{\lambda}^{p}(y)-u^{p}(y)\right)dy\\
\nonumber &>&p\int_{(B^{+}_{\lambda})^{-}}\Bigg(G_{m}^{+}(x,y)-\left(\frac{\lambda}{|y|}\right)^{n-2m}G_{m}^{+}(x,y^{\lambda})\Bigg)
(y_{n})^{a}u^{p-1}(y)\omega^{\lambda}(y)dy\\
\nonumber &\geq&p\int_{(B^{+}_{\lambda})^{-}}\frac{C_{n,m}}{|x-y|^{n-2m}}(y_{n})^{a}u^{p-1}(y)\omega^{\lambda}(y)dy.
\end{eqnarray}

By Hardy-Littlewood-Sobolev inequality and \eqref{2-37++}, we have, for any $\frac{n}{n-2m}<q<\infty$,
\begin{eqnarray}\label{3-14++}
  \|\omega^{\lambda}\|_{L^{q}((B^{+}_{\lambda})^{-})}&\leq& C\left\|(y_{n})^{a}u^{p-1}\omega^{\lambda}\right\|_{L^{\frac{nq}{n+2mq}}((B^{+}_{\lambda})^{-})}\\
  \nonumber &\leq& C\left\|(y_{n})^{a}u^{p-1}\right\|_{L^{\frac{n}{2m}}((B^{+}_{\lambda})^{-})}\|\omega^{\lambda}\|_{L^{q}((B^{+}_{\lambda})^{-})}.
\end{eqnarray}
Since $u\in C(\overline{\mathbb{R}^{n}_{+}})$, there exists a $\epsilon_{0}>0$ small enough, such that
\begin{equation}\label{3-15++}
  C\left\|(y_{n})^{a}u^{p-1}\right\|_{L^{\frac{n}{2m}}((B^{+}_{\lambda})^{-})}\leq\frac{1}{2}
\end{equation}
for all $0<\lambda\leq\epsilon_{0}$, and hence \eqref{3-14++} implies
\begin{equation}\label{3-16++}
  \|\omega^{\lambda}\|_{L^{q}((B^{+}_{\lambda})^{-})}=0,
\end{equation}
which means $(B^{+}_{\lambda})^{-}=\emptyset$. Therefore, we have proved for all $0<\lambda\leq\epsilon_{0}$, $(B^{+}_{\lambda})^{-}=\emptyset$, that is,
\begin{equation}\label{3-17++}
  \omega^{\lambda}(x)\geq0, \,\,\,\,\,\,\, \forall \, x\in B^{+}_{\lambda}(0).
\end{equation}
This completes Step 1.

\emph{Step 2. Dilate the half sphere $S^{+}_{\lambda}$ outward until $\lambda=+\infty$ to derive lower bound estimates on $u$ in a cone.} Step 1 provides us a start point to dilate the half sphere $S^{+}_{\lambda}$ from near $\lambda=0$. Now we dilate the half sphere $S^{+}_{\lambda}$ outward as long as \eqref{2-7++} holds. Let
\begin{equation}\label{2-29++}
  \lambda_{0}:=\sup\{\lambda>0\,|\, \omega^{\mu}\geq0 \,\, in \,\, B^{+}_{\mu}(0), \,\, \forall \, 0<\mu\leq\lambda\}\in(0,+\infty],
\end{equation}
and hence, one has
\begin{equation}\label{2-30++}
  \omega^{\lambda_{0}}(x)\geq0, \quad\quad \forall \,\, x\in B^{+}_{\lambda_{0}}(0).
\end{equation}
In what follows, we will prove $\lambda_{0}=+\infty$ by contradiction arguments.

Suppose on contrary that $0<\lambda_{0}<+\infty$. In order to get a contradiction, we will first prove
\begin{equation}\label{2-31++}
  \omega^{\lambda_{0}}(x)\equiv0, \,\,\,\,\,\,\forall \, x\in B^{+}_{\lambda_{0}}(0)
\end{equation}
by using contradiction arguments.

Suppose on contrary that \eqref{2-31++} does not hold, that is, $\omega^{\lambda_{0}}\geq0$ but $\omega^{\lambda_{0}}$ is not identically zero in $B^{+}_{\lambda_{0}}(0)$, then there exists a $x^{0}\in B^{+}_{\lambda_{0}}(0)$ such that $\omega^{\lambda_{0}}(x^{0})>0$. We will obtain a contradiction with \eqref{2-29++} via showing that the half sphere $S^{+}_{\lambda}$ can be dilated outward a little bit further, more precisely, there exists a $\varepsilon>0$ small enough such that $\omega^{\lambda}\geq0$ in $B^{+}_{\lambda}(0)$ for all $\lambda\in[\lambda_{0},\lambda_{0}+\varepsilon]$.

For that purpose, we will first show that
\begin{equation}\label{2-32++}
  \omega^{\lambda_{0}}(x)>0, \,\,\,\,\,\, \forall \, x\in B^{+}_{\lambda_{0}}(0).
\end{equation}
Indeed, since we have assumed there exists a point $x^{0}\in B^{+}_{\lambda_{0}}(0)$ such that $\omega^{\lambda_{0}}(x^{0})>0$, by continuity, there exists a small $\delta>0$ and a constant $c_{0}>0$ such that
\begin{equation}\label{2-33++}
B_{\delta}(x^{0})\subset B^{+}_{\lambda_{0}}(0) \,\,\,\,\,\, \text{and} \,\,\,\,\,\,
\omega^{\lambda_{0}}(x)\geq c_{0}>0, \,\,\,\,\,\,\,\, \forall \, x\in B_{\delta}(x^{0}).
\end{equation}
From \eqref{2-33++} and the integral equations \eqref{2-34++} and \eqref{2-36++}, one can derive that, for any $x\in B^{+}_{\lambda_{0}}(0)$,
\begin{eqnarray}\label{9-37++}
  &&\omega^{\lambda_{0}}(x)=u_{\lambda_{0}}(x)-u(x) \\
 \nonumber &=&\int_{B^{+}_{\lambda_{0}}(0)}\Bigg(G_{m}^{+}(x,y)-\left(\frac{\lambda_{0}}{|y|}\right)^{n-2m}G_{m}^{+}(x,y^{\lambda_{0}})\Bigg) (y_{n})^{a}\left(\left(\frac{\lambda_{0}}{|y|}\right)^{\tau}u_{\lambda_{0}}^{p}(y)-u^{p}(y)\right)dy\\
 \nonumber &>&\int_{B^{+}_{\lambda_{0}}(0)}\Bigg(G_{m}^{+}(x,y)-\left(\frac{\lambda_{0}}{|y|}\right)^{n-2m}G_{m}^{+}(x,y^{\lambda_{0}})\Bigg) (y_{n})^{a}\left(u_{\lambda_{0}}^{p}(y)-u^{p}(y)\right)dy\\
\nonumber &\geq&p\int_{B^{+}_{\lambda_{0}}(0)}\Bigg(G_{m}^{+}(x,y)-\left(\frac{\lambda_{0}}{|y|}\right)^{n-2m}G_{m}^{+}(x,y^{\lambda_{0}})\Bigg)
(y_{n})^{a}u^{p-1}(y)\omega^{\lambda_{0}}(y)dy\\
\nonumber &\geq&p\int_{B_{\delta}(x^{0})}\Bigg(G_{m}^{+}(x,y)-\left(\frac{\lambda_{0}}{|y|}\right)^{n-2m}G_{m}^{+}(x,y^{\lambda_{0}})\Bigg)
(y_{n})^{a}u^{p-1}(y)\omega^{\lambda_{0}}(y)dy>0,
\end{eqnarray}
thus we arrive at \eqref{2-32++}.

Now, we choose a $0<r_{0}<\frac{1}{4}\lambda_{0}$ small enough, such that
\begin{equation}\label{9-0}
  C\left\|(y_{n})^{a}u^{p-1}\right\|_{L^{\frac{n}{2m}}(A^{+}_{\lambda_{0}+r_{0},2r_{0},r_{0}})}\leq\frac{1}{2},
\end{equation}
where the constant $C$ is the same as in \eqref{3-15++} and the narrow region
\begin{equation}\label{9-1}
  A^{+}_{\lambda_{0}+r_{0},2r_{0},r_{0}}:=\left\{x\in B^{+}_{\lambda_{0}+r_{0}}(0)\,\big|\,|x|>\lambda_{0}-r_{0} \,\,\, \text{or} \,\,\, x_{n}<r_{0}\right\}.
\end{equation}
By \eqref{2-37++}, one can easily verify that inequality as \eqref{3-14++} (with the same constant $C$) also holds for any $\lambda\in[\lambda_{0},\lambda_{0}+r_{0}]$, that is, for any $\frac{n}{n-2m}<q<\infty$,
\begin{equation}\label{9-2}
  \|\omega^{\lambda}\|_{L^{q}((B^{+}_{\lambda})^{-})}\leq C\left\|(y_{n})^{a}u^{p-1}\right\|_{L^{\frac{n}{2m}}((B^{+}_{\lambda})^{-})}\|\omega^{\lambda}\|_{L^{q}((B^{+}_{\lambda})^{-})}.
\end{equation}

By \eqref{2-32++}, we can define
\begin{equation}\label{2-40++}
  m_{0}:=\inf_{x\in B^{+}_{\lambda_{0}}(0)\setminus A^{+}_{\lambda_{0},r_{0},r_{0}}}\omega^{\lambda_{0}}(x)>0,
\end{equation}
where the narrow region
\begin{equation}\label{9-3}
  A^{+}_{\lambda,l_{1},l_{2}}:=\left\{x\in B^{+}_{\lambda}(0)\,\big|\,|x|>\lambda-l_{1} \,\,\, \text{or} \,\,\, x_{n}<l_{2}\right\}
\end{equation}
with $l_{1}>0$, $l_{2}>0$ and $l_{1}+l_{2}<\lambda$. Since $u$ is uniformly continuous on arbitrary compact set $K\subset\overline{\mathbb{R}^{n}_{+}}$ (say, $K=\overline{B^{+}_{2\lambda_{0}}(0)}$), we can deduce from \eqref{2-40++} that, there exists a $0<\varepsilon_{4}<r_{0}$ sufficiently small, such that, for any $\lambda\in[\lambda_{0},\lambda_{0}+\varepsilon_{4}]$,
\begin{equation}\label{2-41++}
  \omega^{\lambda}(x)\geq\frac{m_{0}}{2}>0, \,\,\,\,\,\, \forall \, x\in B^{+}_{\lambda_{0}}(0)\setminus A^{+}_{\lambda_{0},r_{0},r_{0}}.
\end{equation}
The proof of \eqref{2-41++} is completely similar to that of \eqref{2-41}, so we omit the details.

For any $\lambda\in[\lambda_{0},\lambda_{0}+\varepsilon_{4}]$, it follows from \eqref{2-41++} that
\begin{equation}\label{9-4}
  (B^{+}_{\lambda})^{-}\subset A^{+}_{\lambda,\lambda-\lambda_{0}+r_{0},r_{0}}\subset A^{+}_{\lambda_{0}+r_{0},2r_{0},r_{0}}.
\end{equation}
As a consequence of \eqref{9-0}, \eqref{9-2} and \eqref{9-4}, we get
\begin{equation}\label{9-5}
  \|\omega^{\lambda}\|_{L^{q}((B^{+}_{\lambda})^{-})}=0,
\end{equation}
and hence  $(B^{+}_{\lambda})^{-}=\emptyset$ for all $\lambda\in[\lambda_{0},\lambda_{0}+\varepsilon_{4}]$, that is,
\begin{equation}\label{2-45++}
  \omega^{\lambda}(x)\geq0, \,\,\,\,\,\,\, \forall \,\, x\in B^{+}_{\lambda}(0),
\end{equation}
which contradicts with the definition \eqref{2-29++} of $\lambda_{0}$. As a consequence, in the case $0<\lambda_{0}<+\infty$, \eqref{2-31++} must hold true, that is,
\begin{equation}\label{2-46++}
  \omega^{\lambda_{0}}\equiv0 \,\,\,\,\,\, \text{in} \,\,\, B^{+}_{\lambda_{0}}(0).
\end{equation}

However, by the second equality in \eqref{9-37++} and \eqref{2-46++}, we arrive at
\begin{eqnarray}\label{2-47++}
 && 0=\omega^{\lambda_{0}}(x)=u_{\lambda_{0}}(x)-u(x)\\
 \nonumber &=&\int_{B^{+}_{\lambda_{0}}(0)}\Bigg(G_{m}^{+}(x,y)-\left(\frac{\lambda_{0}}{|y|}\right)^{n-2m}G_{m}^{+}(x,y^{\lambda_{0}})\Bigg) (y_{n})^{a}\left(\left(\frac{\lambda_{0}}{|y|}\right)^{\tau}-1\right)u^{p}(y)dy>0
\end{eqnarray}
for any $x\in B^{+}_{\lambda_{0}}(0)$, which is absurd. Thus we must have $\lambda_{0}=+\infty$, that is,
\begin{equation}\label{2-48++}
  u(x)\geq\left(\frac{\lambda}{|x|}\right)^{n-2m}u\left(\frac{\lambda^{2}x}{|x|^{2}}\right), \quad\quad \forall \,\, |x|\geq\lambda, \quad x\in\mathbb{R}^{n}_{+}, \quad \forall \,\, 0<\lambda<+\infty.
\end{equation}
For arbitrary $|x|\geq1$, $x\in\mathbb{R}^{n}_{+}$, let $\lambda:=\sqrt{|x|}$, then \eqref{2-48++} yields that
\begin{equation}\label{2-49++}
  u(x)\geq\frac{1}{|x|^{\frac{n-2m}{2}}}u\left(\frac{x}{|x|}\right),
\end{equation}
and hence, we arrive at the following lower bound estimate:
\begin{equation}\label{2-50++}
  u(x)\geq\left(\min_{x\in S^{+}_{1},\,x_{n}\geq\frac{1}{\sqrt{n}}}u(x)\right)\frac{1}{|x|^{\frac{n-2m}{2}}}:=\frac{C_{0}}{|x|^{\frac{n-2m}{2}}}, \quad\quad \forall \,\, |x|\geq1, \,\,\, x_{n}\geq\frac{|x|}{\sqrt{n}}.
\end{equation}

The lower bound estimate \eqref{2-50++} can be improved remarkably for $1\leq p<\frac{n+2m+2a}{n-2m}$ using the ``Bootstrap" iteration technique and the integral equations \eqref{IE++}.

In fact, let $\mu_{0}:=\frac{n-2m}{2}$, we infer from the integral equations \eqref{IE++} and \eqref{2-50++} that, for any $|x|\geq1$ and $x_{n}\geq\frac{|x|}{\sqrt{n}}$,
\begin{eqnarray}\label{2-51++}
  u(x)&\geq&C\int_{2|x|\leq|y|\leq4|x|,\,y_{n}\geq\frac{|y|}{\sqrt{n}}}\left(\frac{1}{|x-y|^{n-2m}}-\frac{1}{|\bar{x}-y|^{n-2m}}\right)\frac{(y_{n})^{a}}{|y|^{p\mu_{0}}}dy \\
  \nonumber &\geq&C\int_{2|x|\leq|y|\leq4|x|,\,y_{n}\geq\frac{|y|}{\sqrt{n}}}
  \frac{x_{n}y_{n}}{|x-y|^{n-2m+2}}\cdot\frac{(y_{n})^{a}}{|y|^{p\mu_{0}}}dy \\
  \nonumber &\geq&\frac{C}{|x|^{n-2m+1}}\int^{4|x|}_{2|x|}r^{n-p\mu_{0}+a}dr \\
  \nonumber &\geq&\frac{C_{1}}{|x|^{p\mu_{0}-(2m+a)}}.
\end{eqnarray}
Let $\mu_{1}:=p\mu_{0}-(2m+a)$. Due to $1\leq p<\frac{n+2m+2a}{n-2m}$, our important observation is
\begin{equation}\label{2-52++}
  \mu_{1}:=p\mu_{0}-(2m+a)<\mu_{0}.
\end{equation}
Thus we have obtained a better lower bound estimate than \eqref{2-50++} after one iteration, that is,
\begin{equation}\label{2-53++}
  u(x)\geq\frac{C_{1}}{|x|^{\mu_{1}}}, \quad\quad \forall \,\, |x|\geq1, \,\,\, x_{n}\geq\frac{|x|}{\sqrt{n}}.
\end{equation}

For $k=0,1,2,\cdots$, define
\begin{equation}\label{2-54++}
  \mu_{k+1}:=p\mu_{k}-(2m+a).
\end{equation}
Since $1\leq p<\frac{n+2m+2a}{n-2m}$, it is easy to see that the sequence $\{\mu_{k}\}$ is monotone decreasing with respect to $k$ and $n-p\mu_{k}+a>0$ for any $k=0,1,2,\cdots$. Continuing the above iteration process involving the integral equations \eqref{IE++}, we have the following lower bound estimates for every $k=0,1,2,\cdots$,
\begin{equation}\label{2-55++}
  u(x)\geq\frac{C_{k}}{|x|^{\mu_{k}}}, \quad\quad \forall \,\, |x|\geq1, \,\,\, x_{n}\geq\frac{|x|}{\sqrt{n}}.
\end{equation}
Now Theorem \ref{lower3} follows easily from the obvious properties that as $k\rightarrow+\infty$,
\begin{equation}\label{2-56++}
   \mu_{k}\rightarrow-\infty \quad \text{if} \,\, 1\leq p<\frac{n+2m+2a}{n-2m}.
\end{equation}
This finishes our proof of Theorem \ref{lower3}.
\end{proof}

We have proved the nontrivial nonnegative solution $u$ to integral equations \eqref{IE++} is actually a positive solution. For $1\leq p<\frac{n+2m+2a}{n-2m}$, the lower bound estimates in Theorem \ref{lower3} contradicts with the following integrability indicated by the integral equations \eqref{IE++}, that is,
\begin{eqnarray}\label{2-57++}
  +\infty&>&u(e_{n})\geq C\int_{|y|\geq2,\,y_{n}\geq\frac{|y|}{\sqrt{n}}}\left(\frac{1}{|e_{n}-y|^{n-2m}}-\frac{1}{|\overline{e_{n}}-y|^{n-2m}}\right)(y_{n})^{a}u^{p}(y)dy \\
  \nonumber &\geq& C\int_{|y|\geq2,\,y_{n}\geq\frac{|y|}{\sqrt{n}}}\frac{(y_{n})^{a+1}}{|e_{n}-y|^{n-2m+2}}
  u^{p}(y)dy \\
  \nonumber &\geq&C\int_{|y|\geq2,\,y_{n}\geq\frac{|y|}{\sqrt{n}}}\frac{u^{p}(y)}{|y|^{n-2m+1-a}}dy,
\end{eqnarray}
where $e_{n}:=(0,\cdots,0,1)\in\mathbb{R}^{n}_{+}$. Therefore, we must have $u\equiv0$ in $\overline{\mathbb{R}^{n}_{+}}$, that is, the unique nonnegative solution to IEs \eqref{IE++} is $u\equiv0$ in $\overline{\mathbb{R}^{n}_{+}}$.

This concludes our proof of Theorem \ref{Thm3}.

\section{Proof of Theorem \ref{ball}}

In this section, we will prove Theorem \ref{ball} via contradiction arguments and the \emph{method of scaling spheres in local way}. Without loss of generality, we may assume the radius $R=1$.

Now suppose on the contrary that $u\geq0$ satisfies the super-critical problems \eqref{Dball} and \eqref{Nball} for H\'{e}non-Hardy equations but $u$ is not identically zero, then there exists some $\bar{x}\in B_{1}(0)$ such that $u(\bar{x})>0$. Then, by maximum principles and induction, we can infer from \eqref{Dball} and \eqref{Nball} that $u(x)>0$ in $B_{1}(0)$, and derive the equivalence between PDEs \eqref{Dball}, \eqref{Nball} and the following integral equation
\begin{equation}\label{IEball}
  u(x)=\int_{B_{1}(0)}G_{\alpha}(x,y)|y|^{a}u^{p}(y)dy,
\end{equation}
where
\begin{equation}\label{GREENball-0}
  G_{\alpha}(x,y):=\frac{C_{n,\alpha}}{|x-y|^{n-\alpha}}\int_{0}^{\frac{(1-|x|^{2})(1-|y|^{2})}{|x-y|^{2}}}\frac{b^{\frac{\alpha}{2}-1}}{(1+b)^{\frac{n}{2}}}db \quad\quad \text{if} \,\,\, x,y\in B_{1}(0),
\end{equation}
and $G_{\alpha}(x,y):=0$ if $x$ or $y\in\mathbb{R}^{n}\setminus B_{1}(0)$ is Green's function in $B_{1}(0)$ for $(-\Delta)^{\frac{\alpha}{2}}$ with Dirichlet boundary conditions when $0<\alpha\leq2$, and
\begin{equation}\label{GREENball-1}
  G_{\alpha}(x,y):=C_{n,\alpha}\Bigg(\frac{1}{|x-y|^{n-\alpha}}-\frac{1}{\Big|\frac{x}{|x|}-|x|y\Big|^{n-\alpha}}\Bigg) \quad\quad \text{if} \,\,\, x,y\in B_{1}(0),
\end{equation}
and $G_{\alpha}(x,y):=0$ if $x$ or $y\in\mathbb{R}^{n}\setminus B_{1}(0)$ is Green's function in $B_{1}(0)$ for $(-\Delta)^{\frac{\alpha}{2}}$ with Navier boundary conditions when $\alpha=2m$ and $1\leq m<\frac{n}{2}$. This means, the positive classical solution $u$ to \eqref{Dball} and \eqref{Nball} also satisfies the equivalent integral equation \eqref{IEball}, and vice versa.

Next, we will apply the \emph{method of scaling spheres in local way} to show the following lower bound estimate for asymptotic behaviour of positive solution $u$ as $x\rightarrow0$, which contradicts with the integral equation \eqref{IEball} for $p_{c}(a)<p<+\infty$.
\begin{thm}\label{lower4}
Assume $n>\alpha$, $0<\alpha<2$ or $\alpha=2m$ with $1\leq m<\frac{n}{2}$, $-\alpha<a<+\infty$ and $\frac{n+\alpha+2a}{n-\alpha}<p<+\infty$. Suppose $u$ is a positive solution to super-critical problems \eqref{Dball} or \eqref{Nball}, then it satisfies the following lower bound estimates: for all $0<|x|\leq\frac{1}{100}$,
\begin{equation}\label{lb2++b}
  u(x)\geq \frac{C_{\kappa}}{|x|^{\kappa}} \quad\quad \forall \, \kappa<+\infty.
\end{equation}
\end{thm}
\begin{proof}
Given any $0<\lambda<1$, we define the Kelvin transform of $u$ centered at $0$ by
\begin{equation}\label{Kelvin++b}
  u_{\lambda}(x)=\left(\frac{\lambda}{|x|}\right)^{n-\alpha}u\left(\frac{\lambda^{2}x}{|x|^{2}}\right)
\end{equation}
for arbitrary $x\in\{x\in\overline{B_{1}(0)}\,|\,\lambda^{2}\leq|x|\leq1\}$.

Now, we will carry out the process of scaling spheres in $B_{1}(0)$ with respect to the origin $0\in\mathbb{R}^{n}$.

Let $0<\lambda<1$ be an arbitrary positive real number and let $\omega^{\lambda}(x):=u_{\lambda}(x)-u(x)$ for any $x\in B_{\lambda}(0)\setminus\overline{B_{\lambda^{2}}(0)}$. By the definition of $u_{\lambda}$ and $\omega^{\lambda}$, we have
\begin{equation}\label{2-6++b}
\omega^{\lambda}(x)=u_{\lambda}(x)-u(x)=-\left(\frac{\lambda}{|x|}\right)^{n-\alpha}\omega^{\lambda}(x^{\lambda})=-\big(\omega^{\lambda}\big)_{\lambda}(x)
\end{equation}
for every $x\in B_{\lambda}(0)\setminus\overline{B_{\lambda^{2}}(0)}$.

We will first show that, for $0<\lambda<1$ sufficiently close to $1$,
\begin{equation}\label{2-7++b}
  \omega^{\lambda}(x)\leq0, \,\,\,\,\,\, \forall \,\, x\in B_{\lambda}(0)\setminus\overline{B_{\lambda^{2}}(0)}.
\end{equation}
Then, we start shrinking the sphere $S_{\lambda}$ from near the unit sphere $S_{1}$ inward as long as \eqref{2-7++b} holds, until its limiting position $\lambda=0$ and derive lower bound estimates on asymptotic behaviour of $u$ as $x\rightarrow0$. Therefore, the scaling sphere process can be divided into two steps.

\emph{Step 1. Start shrinking the sphere $S_{\lambda}$ from near $\lambda=1$.} Define
\begin{equation}\label{2-8++b}
  (B_{\lambda}\setminus\overline{B_{\lambda^{2}}})^{+}:=\{x\in B_{\lambda}(0)\setminus\overline{B_{\lambda^{2}}(0)} \, | \, \omega^{\lambda}(x)>0\}.
\end{equation}
We will show that, for $0<\lambda<1$ sufficiently close to $1$,
\begin{equation}\label{2-9++b}
  (B_{\lambda}\setminus\overline{B_{\lambda^{2}}})^{+}=\emptyset.
\end{equation}

Since the positive solution $u$ to super-critical problems \eqref{Dball} and \eqref{Nball} also satisfies the integral equation \eqref{IEball}, through direct calculations, we get, for any $0<\lambda<1$,
\begin{equation}\label{2-34++b}
  u(x)=\int_{B_{\lambda}(0)}G_{\alpha}(x,y)|y|^{a}u^{p}(y)dy+\int_{B_{\lambda}(0)\setminus\overline{B_{\lambda^{2}}(0)}}G_{\alpha}(x,y^{\lambda})
  \left(\frac{\lambda}{|y|}\right)^{\tau+n-\alpha}|y|^{a}u_{\lambda}^{p}(y)dy
\end{equation}
for any $x\in\overline{B_{1}(0)}$, where $\tau:=n+\alpha+2a-p(n-\alpha)<0$. By direct calculations, one can also verify that $u_{\lambda}$ satisfies the following integral equation
\begin{equation}\label{2-35++b}
  u_{\lambda}(x)=\int_{B_{1}(0)}G_{\alpha}(x^{\lambda},y)\left(\frac{\lambda}{|x|}\right)^{n-\alpha}|y|^{a}u^{p}(y)dy
\end{equation}
for any $x\in\{x\in\overline{B_{1}(0)}\,|\,\lambda^{2}\leq|x|\leq1\}$, and hence, it follows immediately that
\begin{eqnarray}\label{2-36++b}
  u_{\lambda}(x)&=&\int_{B_{\lambda}(0)}G_{\alpha}(x^{\lambda},y)\left(\frac{\lambda}{|x|}\right)^{n-\alpha}|y|^{a}u^{p}(y)dy \\
 \nonumber \quad\quad &&+\int_{B_{\lambda}(0)\setminus\overline{B_{\lambda^{2}}(0)}}G_{\alpha}(x^{\lambda},y^{\lambda})\left(\frac{\lambda^{2}}{|x|\cdot|y|}\right)^{n-\alpha}\left(\frac{\lambda}{|y|}\right)^{\tau}
 |y|^{a}u_{\lambda}^{p}(y)dy.
\end{eqnarray}
Therefore, we have, for any $x\in B_{1}(0)\setminus\overline{B_{\lambda^{2}}(0)}$,
\begin{eqnarray}\label{omega}
  && \omega^{\lambda}(x)=u_{\lambda}(x)-u(x) \\
  \nonumber &=& \int_{B_{\lambda}(0)\setminus\overline{B_{\lambda^{2}}(0)}}\left[\left(\frac{\lambda^{2}}{|x|\cdot|y|}\right)^{n-\alpha}G_{\alpha}(x^{\lambda},y^{\lambda})
  -\left(\frac{\lambda}{|y|}\right)^{n-\alpha}G_{\alpha}(x,y^{\lambda})\right]|y|^{a}\left(\frac{\lambda}{|y|}\right)^{\tau}u^{p}_{\lambda}(y) \\
 \nonumber && -\left[G_{\alpha}(x,y)-\left(\frac{\lambda}{|x|}\right)^{n-\alpha}G_{\alpha}(x^{\lambda},y)\right]|y|^{a}u^{p}(y)dy \\
 \nonumber && +\int_{B_{\lambda^{2}}(0)}\left[\left(\frac{\lambda}{|x|}\right)^{n-\alpha}G_{\alpha}(x^{\lambda},y)-G_{\alpha}(x,y)\right]|y|^{a}u^{p}(y)dy,
\end{eqnarray}
where $\tau:=n+\alpha+2a-p(n-\alpha)<0$.

Now we need the following Lemma on properties of Green's function $G_{\alpha}(x,y)$.
\begin{lem}\label{Gball}
Green's function $G_{\alpha}(x,y)$ satisfies the following point-wise estimates:
\begin{flalign}
\nonumber &\text{$(i)\,\, 0\leq G_{\alpha}(x,y)\leq\frac{C'}{|x-y|^{n-\alpha}}, \quad\quad \forall \,\, x,y\in\mathbb{R}^{n};$}& \\
\nonumber &\text{$(ii) \,\, G_{\alpha}(x,y)\geq\frac{C''}{|x-y|^{n-\alpha}}, \quad\quad \forall \,\, 0\leq|x|,|y|\leq\frac{1}{10};$}& \\
\nonumber &\text{$(iii) \,\, \left(\frac{\lambda}{|x|}\right)^{n-\alpha}G_{\alpha}(x^{\lambda},y)-G_{\alpha}(x,y)\leq0, \quad\quad \forall \,\, \lambda^{2}<|x|<\lambda, \,\, 0\leq|y|<\lambda;$}& \\
\nonumber &\text{$(iv) \,\, \left(\frac{\lambda^{2}}{|x|\cdot|y|}\right)^{n-\alpha}G_{\alpha}(x^{\lambda},y^{\lambda})-\left(\frac{\lambda}{|y|}\right)^{n-\alpha}
  G_{\alpha}(x,y^{\lambda})\leq G_{\alpha}(x,y)-\left(\frac{\lambda}{|x|}\right)^{n-\alpha}G_{\alpha}(x^{\lambda},y),$}& \\
\nonumber &\text{$\qquad\qquad\qquad\qquad\qquad\qquad\qquad\qquad\qquad\qquad\qquad\qquad\qquad\qquad\quad \forall \,\, \lambda^{2}<|x|,|y|<\lambda.$}&
\end{flalign}
\end{lem}

Lemma \ref{Gball} can be proved by direct calculations, so we omit the details here.

From Lemma \ref{Gball} and the integral equations \eqref{omega}, one can derive that, for any $x\in B_{\lambda}(0)\setminus\overline{B_{\lambda^{2}}(0)}$,
\begin{eqnarray}\label{2-37++b}
  &&\omega^{\lambda}(x)=u_{\lambda}(x)-u(x) \\
 \nonumber &\leq&\int_{B_{\lambda}(0)\setminus\overline{B_{\lambda^{2}}(0)}}\Bigg[G_{\alpha}(x,y)-\left(\frac{\lambda}{|x|}\right)^{n-\alpha}G_{\alpha}(x^{\lambda},y)\Bigg] |y|^{a}\left[\left(\frac{\lambda}{|y|}\right)^{\tau}u_{\lambda}^{p}(y)-u^{p}(y)\right]dy\\
\nonumber &<&\int_{B_{\lambda}(0)\setminus\overline{B_{\lambda^{2}}(0)}}\Bigg(G_{\alpha}(x,y)-\left(\frac{\lambda}{|x|}\right)^{n-\alpha}G_{\alpha}(x^{\lambda},y)\Bigg) |y|^{a}\left(u_{\lambda}^{p}(y)-u^{p}(y)\right)dy\\
\nonumber &\leq&C\int_{\left(B_{\lambda}\setminus\overline{B_{\lambda^{2}}}\right)^{+}}\frac{|y|^{a}}{|x-y|^{n-\alpha}}u_{\lambda}^{p-1}(y)\omega^{\lambda}(y)dy.
\end{eqnarray}

By Hardy-Littlewood-Sobolev inequality and \eqref{2-37++b}, we have, for any $\frac{n}{n-\alpha}<q<\infty$,
\begin{eqnarray}\label{3-14++b}
  \|\omega^{\lambda}\|_{L^{q}((B_{\lambda}\setminus\overline{B_{\lambda^{2}}})^{+})}&\leq& C\left\||x|^{a}u_{\lambda}^{p-1}\omega^{\lambda}\right\|_{L^{\frac{nq}{n+\alpha q}}((B_{\lambda}\setminus\overline{B_{\lambda^{2}}})^{+})}\\
  \nonumber &\leq& C\left\||x|^{a}u^{p-1}_{\lambda}\right\|_{L^{\frac{n}{\alpha}}((B_{\lambda}\setminus\overline{B_{\lambda^{2}}})^{+})}
  \cdot\|\omega^{\lambda}\|_{L^{q}((B_{\lambda}\setminus\overline{B_{\lambda^{2}}})^{+})}.
\end{eqnarray}
Since $u\in C(\overline{B_{1}(0)})$, there exists a $\epsilon_{0}>0$ small enough, such that
\begin{equation}\label{3-15++b}
  C\left\||x|^{a}u^{p-1}_{\lambda}\right\|_{L^{\frac{n}{\alpha}}((B_{\lambda}\setminus\overline{B_{\lambda^{2}}})^{+})}\leq\frac{1}{2}
\end{equation}
for all $1-\epsilon_{0}\leq\lambda<1$, and hence \eqref{3-14++b} implies
\begin{equation}\label{3-16++b}
  \|\omega^{\lambda}\|_{L^{q}((B_{\lambda}\setminus\overline{B_{\lambda^{2}}})^{+})}=0,
\end{equation}
which means $(B_{\lambda}\setminus\overline{B_{\lambda^{2}}})^{+}=\emptyset$. Therefore, we have proved for all $1-\epsilon_{0}\leq\lambda<1$, $(B_{\lambda}\setminus\overline{B_{\lambda^{2}}})^{+}=\emptyset$, that is,
\begin{equation}\label{3-17++b}
  \omega^{\lambda}(x)\leq0, \,\,\,\,\,\,\, \forall \, x\in B_{\lambda}(0)\setminus\overline{B_{\lambda^{2}}(0)}.
\end{equation}
This completes Step 1.

\emph{Step 2. Shrink the sphere $S_{\lambda}$ inward until $\lambda=0$ to derive lower bound estimates on asymptotic behaviour of $u$ as $x\rightarrow0$.} Step 1 provides us a start point to shrink the sphere $S_{\lambda}$ from near $\lambda=1$. Now we shrink the sphere $S_{\lambda}$ inward as long as \eqref{2-7++b} holds. Let
\begin{equation}\label{2-29++b}
  \lambda_{0}:=\inf\{0<\lambda<1\,|\, \omega^{\mu}\leq0 \,\, in \,\, B_{\mu}(0)\setminus\overline{B_{\mu^{2}}(0)}, \,\, \forall \, \lambda\leq\mu<1\}\in[0,1),
\end{equation}
and hence, one has
\begin{equation}\label{2-30++b}
  \omega^{\lambda_{0}}(x)\leq0, \quad\quad \forall \,\, x\in B_{\lambda_{0}}(0)\setminus\overline{B_{\lambda_{0}^{2}}(0)}.
\end{equation}
In what follows, we will prove $\lambda_{0}=0$ by contradiction arguments.

Suppose on contrary that $0<\lambda_{0}<1$. In order to get a contradiction, we will first prove
\begin{equation}\label{2-31++b}
  \omega^{\lambda_{0}}(x)\equiv0, \,\,\,\,\,\,\forall \, x\in B_{\lambda_{0}}(0)\setminus\overline{B_{\lambda_{0}^{2}}(0)}
\end{equation}
by using contradiction arguments.

Suppose on contrary that \eqref{2-31++b} does not hold, that is, $\omega^{\lambda_{0}}\leq0$ but $\omega^{\lambda_{0}}$ is not identically zero in $B_{\lambda_{0}}(0)\setminus\overline{B_{\lambda_{0}^{2}}(0)}$, then there exists a $x^{0}\in B_{\lambda_{0}}(0)\setminus\overline{B_{\lambda_{0}^{2}}(0)}$ such that $\omega^{\lambda_{0}}(x^{0})<0$. We will obtain a contradiction with \eqref{2-29++b} via showing that the sphere $S_{\lambda}$ can be shrunk inward a little bit further, more precisely, there exists a $0<\varepsilon<\lambda_{0}$ small enough such that $\omega^{\lambda}\leq0$ in $B_{\lambda}(0)\setminus\overline{B_{\lambda^{2}}(0)}$ for all $\lambda\in[\lambda_{0}-\varepsilon,\lambda_{0}]$.

For that purpose, we will first show that
\begin{equation}\label{2-32++b}
  \omega^{\lambda_{0}}(x)<0, \,\,\,\,\,\, \forall \, x\in B_{\lambda_{0}}(0)\setminus\overline{B_{\lambda_{0}^{2}}(0)}.
\end{equation}
Indeed, since we have assumed there exists a point $x^{0}\in B_{\lambda_{0}}(0)\setminus\overline{B_{\lambda_{0}^{2}}(0)}$ such that $\omega^{\lambda_{0}}(x^{0})<0$, by continuity, there exists a small $\delta>0$ and a constant $c_{0}>0$ such that
\begin{equation}\label{2-33++b}
B_{\delta}(x^{0})\subset B_{\lambda_{0}}(0)\setminus\overline{B_{\lambda_{0}^{2}}(0)} \,\,\,\,\,\, \text{and} \,\,\,\,\,\,
\omega^{\lambda_{0}}(x)\leq -c_{0}<0, \,\,\,\,\,\,\,\, \forall \, x\in B_{\delta}(x^{0}).
\end{equation}
From \eqref{2-33++b}, Lemma \ref{Gball} and the integral equations \eqref{2-34++b} and \eqref{2-36++b}, one can derive that, for any $x\in B_{\lambda_{0}}(0)\setminus\overline{B_{\lambda_{0}^{2}}(0)}$,
\begin{eqnarray}\label{9-37++b}
  &&\omega^{\lambda_{0}}(x)=u_{\lambda_{0}}(x)-u(x) \\
 \nonumber &\leq&\int_{B_{\lambda_{0}}(0)\setminus\overline{B_{\lambda_{0}^{2}}(0)}}
 \Bigg[G_{\alpha}(x,y)-\left(\frac{\lambda_{0}}{|x|}\right)^{n-\alpha}G_{\alpha}(x^{\lambda_{0}},y)\Bigg]|y|^{a}
 \left[\left(\frac{\lambda_{0}}{|y|}\right)^{\tau}u_{\lambda_{0}}^{p}(y)-u^{p}(y)\right]dy\\
\nonumber &<&\int_{B_{\lambda_{0}}(0)\setminus\overline{B_{\lambda_{0}^{2}}(0)}}
\Bigg(G_{\alpha}(x,y)-\left(\frac{\lambda_{0}}{|x|}\right)^{n-\alpha}G_{\alpha}(x^{\lambda_{0}},y)\Bigg)|y|^{a}
\left(u_{\lambda_{0}}^{p}(y)-u^{p}(y)\right)dy\\
\nonumber &\leq&p\int_{B_{\delta}(x_{0})}\Bigg(G_{\alpha}(x,y)-\left(\frac{\lambda_{0}}{|x|}\right)^{n-\alpha}G_{\alpha}(x^{\lambda_{0}},y)\Bigg)|y|^{a}
u_{\lambda_{0}}^{p-1}(y)\omega^{\lambda_{0}}(y)dy<0,
\end{eqnarray}
thus we arrive at \eqref{2-32++b}.

Now, we choose a $0<r_{0}<\frac{1}{4}\min\{\lambda_{0}-\lambda_{0}^{2},\lambda_{0}^{2}\}$ small enough, such that
\begin{equation}\label{9-0b}
  C\left\||x|^{a}u_{\lambda}^{p-1}\right\|_{L^{\frac{n}{\alpha}}\big(A_{\lambda_{0},r_{0}}\cup A_{\lambda_{0}^{2}+r_{0},2r_{0}}\big)}\leq\frac{1}{2}
\end{equation}
for any $\lambda\in[\lambda_{0}-\frac{r_{0}}{2},\lambda_{0}]$, where the constant $C$ is the same as in \eqref{3-15++b} and the narrow region
\begin{equation}\label{9-1b}
  A_{r,l}:=\left\{x\in B_{r}(0)\,\big|\,|x|>r-l\right\}
\end{equation}
for $r>0$ and $0<l<r$. By \eqref{2-37++b}, one can easily verify that inequality as \eqref{3-14++b} (with the same constant $C$) also holds for any $\lambda\in[\lambda_{0}-\frac{r_{0}}{2},\lambda_{0}]$, that is, for any $\frac{n}{n-\alpha}<q<\infty$,
\begin{equation}\label{9-2b}
  \|\omega^{\lambda}\|_{L^{q}\left((B_{\lambda}\setminus\overline{B_{\lambda^{2}}})^{+}\right)}\leq C\left\||x|^{a}u^{p-1}_{\lambda}\right\|_{L^{\frac{n}{\alpha}}\left((B_{\lambda}\setminus\overline{B_{\lambda^{2}}})^{+}\right)}
  \cdot\|\omega^{\lambda}\|_{L^{q}\left((B_{\lambda}\setminus\overline{B_{\lambda^{2}}})^{+}\right)}.
\end{equation}

By \eqref{2-32++b}, we can define
\begin{equation}\label{2-40++b}
  M_{0}:=\sup_{x\in \overline{B_{\lambda_{0}-r_{0}}(0)}\setminus B_{\lambda_{0}^{2}+r_{0}}(0)}\omega^{\lambda_{0}}(x)<0.
\end{equation}
Since $u$ is uniformly continuous on arbitrary compact set $K\subset\overline{B_{1}(0)}$ (say, $K=\{x\in\overline{B_{1}(0)}\,|\,\frac{\lambda_{0}^{2}+r_{0}}{2}\leq|x|\leq\lambda_{0}-r_{0}\}$), we can deduce from \eqref{2-40++b} that, there exists a $0<\varepsilon_{5}<\frac{r_{0}}{2}$ sufficiently small, such that, for any $\lambda\in[\lambda_{0}-\varepsilon_{5},\lambda_{0}]$,
\begin{equation}\label{2-41++b}
  \omega^{\lambda}(x)\leq\frac{M_{0}}{2}<0, \,\,\,\,\,\, \forall \, x\in \overline{B_{\lambda_{0}-r_{0}}(0)}\setminus B_{\lambda_{0}^{2}+r_{0}}(0).
\end{equation}
The proof of \eqref{2-41++b} is completely similar to that of \eqref{2-41}, so we omit the details.

For any $\lambda\in[\lambda_{0}-\varepsilon_{5},\lambda_{0}]$, it follows from \eqref{2-41++b} that
\begin{equation}\label{9-4b}
  (B_{\lambda}\setminus\overline{B_{\lambda^{2}}})^{+}\subset A_{\lambda_{0},r_{0}}\cup A_{\lambda_{0}^{2}+r_{0},2r_{0}}.
\end{equation}
As a consequence of \eqref{9-0b}, \eqref{9-2b} and \eqref{9-4b}, we get
\begin{equation}\label{9-5b}
  \|\omega^{\lambda}\|_{L^{q}\left((B_{\lambda}\setminus\overline{B_{\lambda^{2}}})^{+}\right)}=0,
\end{equation}
and hence $(B_{\lambda}\setminus\overline{B_{\lambda^{2}}})^{+}=\emptyset$ for all $\lambda\in[\lambda_{0}-\varepsilon_{5},\lambda_{0}]$, that is,
\begin{equation}\label{2-45++b}
  \omega^{\lambda}(x)\leq0, \,\,\,\,\,\,\, \forall \,\, x\in B_{\lambda}(0)\setminus\overline{B_{\lambda^{2}}(0)},
\end{equation}
which contradicts with the definition \eqref{2-29++b} of $\lambda_{0}$. As a consequence, in the case $0<\lambda_{0}<1$, \eqref{2-31++b} must hold true, that is,
\begin{equation}\label{2-46++b}
  \omega^{\lambda_{0}}\equiv0 \,\,\,\,\,\, \text{in} \,\,\, B_{\lambda_{0}}(0)\setminus\overline{B_{\lambda_{0}^{2}}(0)}.
\end{equation}

However, by the first inequality in \eqref{9-37++b} and \eqref{2-46++b}, we arrive at
\begin{eqnarray}\label{2-47++b}
 && 0=\omega^{\lambda_{0}}(x)=u_{\lambda_{0}}(x)-u(x)\\
 \nonumber &\leq&\int_{B_{\lambda_{0}}(0)\setminus\overline{B_{\lambda_{0}^{2}}(0)}}
 \Bigg[G_{\alpha}(x,y)-\left(\frac{\lambda_{0}}{|x|}\right)^{n-\alpha}G_{\alpha}(x^{\lambda_{0}},y)\Bigg]|y|^{a}
 \left[\left(\frac{\lambda_{0}}{|y|}\right)^{\tau}-1\right]u^{p}(y)dy<0
\end{eqnarray}
for any $x\in B_{\lambda_{0}}(0)\setminus\overline{B_{\lambda_{0}^{2}}(0)}$, which is absurd. Thus we must have $\lambda_{0}=0$, that is,
\begin{equation}\label{2-48++b}
  u(x)\geq\left(\frac{\lambda}{|x|}\right)^{n-\alpha}u\left(\frac{\lambda^{2}x}{|x|^{2}}\right), \quad\quad \forall \,\, \lambda^{2}\leq|x|\leq\lambda, \quad \forall \,\, 0<\lambda<1.
\end{equation}
For arbitrary $0<|x|\leq\frac{1}{2}$, let $0<\lambda:=\sqrt{\frac{|x|}{2}}\leq\frac{1}{2}$, then \eqref{2-48++b} yields that
\begin{equation}\label{2-49++b}
  u(x)\geq\left(\frac{1}{2|x|}\right)^{\frac{n-\alpha}{2}}u\left(\frac{x}{2|x|}\right),
\end{equation}
and hence, we arrive at the following lower bound estimate on asymptotic behaviour of $u$ as $x\rightarrow0$:
\begin{equation}\label{2-50++b}
  u(x)\geq\left(\min_{x\in S_{\frac{1}{2}}}u(x)\right)\left(\frac{1}{2|x|}\right)^{\frac{n-\alpha}{2}}:=\frac{C_{0}}{|x|^{\frac{n-\alpha}{2}}}, \quad\quad \forall \,\, 0<|x|\leq\frac{1}{2}.
\end{equation}

The lower bound estimate \eqref{2-50++b} can be improved remarkably for $p_{c}(a):=\frac{n+\alpha+2a}{n-\alpha}<p<+\infty$ using the ``Bootstrap" iteration technique and the integral equation \eqref{IEball}.

In fact, let $\mu_{0}:=\frac{n-\alpha}{2}$, we infer from the integral equation \eqref{IEball}, Lemma \ref{Gball} and \eqref{2-50++b} that, for any $0<|x|\leq\frac{1}{100}$,
\begin{eqnarray}\label{2-51++b}
  u(x)&\geq&C\int_{\frac{1}{4}|x|\leq|y|\leq\frac{1}{2}|x|}G_{\alpha}(x,y)|y|^{a}\frac{1}{|y|^{p\mu_{0}}}dy \\
  \nonumber &\geq&C\int_{\frac{1}{4}|x|\leq|y|\leq\frac{1}{2}|x|}\frac{1}{|x-y|^{n-\alpha}}\cdot\frac{1}{|y|^{p\mu_{0}-a}}dy \\
  \nonumber &\geq&\frac{C}{|x|^{n-\alpha}}\int^{\frac{|x|}{2}}_{\frac{|x|}{4}}r^{n-1-p\mu_{0}+a}dr \\
  \nonumber &\geq&\frac{C_{1}}{|x|^{p\mu_{0}-(\alpha+a)}}.
\end{eqnarray}
Now, let $\mu_{1}:=p\mu_{0}-(\alpha+a)$. Due to $p_{c}(a):=\frac{n+\alpha+2a}{n-\alpha}<p<+\infty$, our important observation is
\begin{equation}\label{2-52++b}
  \mu_{1}:=p\mu_{0}-(\alpha+a)>\mu_{0}.
\end{equation}
Thus we have obtained a better lower bound estimate than \eqref{2-50++b} after one iteration, that is,
\begin{equation}\label{2-53++b}
  u(x)\geq\frac{C_{1}}{|x|^{\mu_{1}}}, \quad\quad \forall \,\, 0<|x|\leq\frac{1}{100}.
\end{equation}

For $k=0,1,2,\cdots$, define
\begin{equation}\label{2-54++b}
  \mu_{k+1}:=p\mu_{k}-(\alpha+a).
\end{equation}
Since $p_{c}(a):=\frac{n+\alpha+2a}{n-\alpha}<p<+\infty$, it is easy to see that the sequence $\{\mu_{k}\}$ is monotone increasing with respect to $k$. Continuing the above iteration process involving the integral equation \eqref{IEball}, we have the following lower bound estimates for every $k=0,1,2,\cdots$,
\begin{equation}\label{2-55++b}
  u(x)\geq\frac{C_{k}}{|x|^{\mu_{k}}}, \quad\quad \forall \,\, 0<|x|\leq\frac{1}{100}.
\end{equation}
Now Theorem \ref{lower4} follows easily from the obvious properties that as $k\rightarrow+\infty$,
\begin{equation}\label{2-56++b}
   \mu_{k}\rightarrow+\infty \quad \text{if} \,\,\, \frac{n+\alpha+2a}{n-\alpha}<p<+\infty.
\end{equation}
This finishes our proof of Theorem \ref{lower4}.
\end{proof}

We have proved the nontrivial nonnegative solution $u$ to the super-critical problems \eqref{Dball} and \eqref{Nball} for H\'{e}non-Hardy equations in $\overline{B_{1}(0)}$ is actually a positive solution which also satisfies the integral equation \eqref{IEball}. For $p_{c}(a):=\frac{n+\alpha+2a}{n-\alpha}<p<+\infty$, one can easily observe that the lower bound estimates on asymptotic behaviour of $u$ as $x\rightarrow0$ in Theorem \ref{lower4} indicate strong singularity of $u$ at the origin $0$, which obviously contradicts with the integral equation \eqref{IEball}. Therefore, we must have $u\equiv0$ in $\overline{B_{1}(0)}$, that is, the unique nonnegative solution to the super-critical problems \eqref{Dball} and \eqref{Nball} is $u\equiv0$ in $\overline{B_{1}(0)}$.

This concludes our proof of Theorem \ref{ball}.

\end{document}